\documentclass[a4paper,12pt]{article}
\usepackage[T1]{fontenc}
\usepackage[latin1]{inputenc}
\usepackage{amsmath}
\usepackage{amssymb}
\usepackage{amsthm}
\usepackage{amsfonts}
\usepackage{hyperref}
\newcommand{\R}{\mathbb{R}}
\newcommand{\C}{\mathbb{C}}
\newcommand{\T}{\mathbb{T}}
\newcommand{\Z}{\mathbb{Z}}
\newcommand{\N}{\mathbb{N}}
\newcommand{\tb}{\underline{\theta}}
\newcommand{\ul}[1]{\underline{#1}}
\newcommand{\ld}{\lambda}
\newcommand{\Ld}{\Lambda}

\theoremstyle{plain}
\newtheorem{thm}{Theorem}[section]
\newtheorem{lem}{Lemma}[section]
\newtheorem{prop}{Proposition}[section]
\newtheorem{cor}{Corollary}[section]
\newtheorem{stm}{Statement}[section]
\newtheorem*{scl}{Theorem A (scale recurrence lemma)}
\newtheorem*{dmf}{Theorem B (dimension formula)}
\newtheorem*{cor6.4}{Corollary 6.4}

\theoremstyle{definition}
\newtheorem{defn}{Definition}[section]

\theoremstyle{remark}

\title{Scale Recurrence Lemma and Dimension Formula for Cantor Sets in the Complex Plane}
\author{Carlos Gustavo T. de A. Moreira\thanks{
                  IMPA}
        \and
        Alex Mauricio Zamudio\thanks{
                      UFRJ
                  }\thanks{This work was
                  supported by CNPq and CAPES.}
        }

\begin{document}

\maketitle

\begin{abstract}
We will prove a multidimensional conformal version of the scale recurrence lemma of Moreira and Yoccoz \cite{MY} for Cantor sets in the complex plane. We then use this new recurrence lemma, together with the ideas in \cite{M}, to prove that under the right hypothesis for the Cantor sets $K_1,...,K_n$ and the function $h:\C^{n}\to \R^{l}$, the following formula holds
\[HD(h(K_1\times K_2 \times ...\times K_n))=\min \{l,HD(K_1)+...+HD(K_n)\}.\]
\end{abstract}

\section{Introduction.}

In this paper, we prove a version of the {\em scale recurrence lemma} of Moreira and Yoccoz (see subsection 3.2 of \cite{MY}) in the context of Cantor sets in the complex plane. We will use this new version, together with other results, to prove a dimension formula for projections of products of complex Cantor sets. More precisely, given conformal regular Cantor sets $K_1,...,K_n$ in $\C$, and a $C^1$ function $h:\C^{n}\to \R^{l}$, we prove that, under natural hypothesis, one has
\begin{equation}\label{eq:DimFor}
HD(h(K_1\times K_2 \times ...\times K_n))=\min \{l,HD(K_1)+...+HD(K_n)\}.
\end{equation}

Our results will be proved for conformal regular Cantor sets. Those are Cantor sets which are maximal invariant sets for an expanding map, whose derivative is conformal at the points in the Cantor set. Rigorous definitions will be given in section \ref{BDef}. The investigation of such Cantor sets is important because they appear in the study of homoclinic bifurcations for automorphisms of $\C^2$, as shown by Araujo in \cite{H}. We expect that conformal regular Cantor sets in $\C$ will play a role in the study of homoclinic bifurcations for automorphisms of $\C^2$, similar to regular Cantor sets in $\R$ in the study of homoclinic bifurcations for surface diffeomorphisms. 

The study of homoclinic bifurcations has proved to be fruitful in the understanding of dynamics for surface diffeomorphisms. Complicated dynamical phenomena arise from homoclicic bifurcations. For example the Newhouse phenomenon, which is the coexistence of an infinite number of periodic attractors for a generic set, inside an open set, of diffeomorphism.

The scale recurrence lemma was an important step in the solution to Palis conjecture, about the arithmetic difference of Cantor sets, by Moreira and Yoccoz. They proved that there is an open and dense subset, inside the set of pairs of regular Cantor sets with sum of Hausdorff dimensions bigger than one, such that any pair $(K_1,K_2)$ in this subset verifies $int(K_1-K_2)\neq \emptyset$. The theorem of Moreira and Yoccoz is for regular Cantor sets in the real line. Together with Araujo we are working in proving an analogous result for conformal regular Cantor sets in the complex plane. The scale recurrence lemma in these pages is a fundamental tool for this future work.  

Furthermore, Moreira and Yoccoz were able to use their solution to Palis conjecture in the study of homoclinic bifurcations for surface diffeomorphisms (see \cite{MY2}). They proved that given a surface difeomorphism $F$ with a homoclinic quadratic tangency associated to a horseshoe with dimension larger than one, the set of diffeomorphisms close to $F$ presenting a stable tangency has positive density at $F$. One of the main reasons to study conformal regular Cantor sets is to apply the ideas in \cite{MY2} to the context of homoclinic bifurcations for automorphims of $\C^2$, the scale recurrence lemma certainly is an important step in this direction.

Another development in the subject was given by Lopez \cite{TJE}. He generalized the work \cite{MY} for a product of several Cantor sets in the real line. In this paper we will consider a scale recurrence lemma for a product of several Cantor sets in the complex plane.

On the other hand, we have the dimension formula as an application of the scale recurrence lemma. The study of this type of dimension formulas is motivated by a classical theorem of Marstrand, generalized by Mattila and others. Denote by $G(m,l)$ the set of $l$-dimensional linear subspaces of $\R^m$, for $s\in G(m,l)$ denote by $\pi_s$ the orthogonal projection on $s$. Marstrand theorem states that given $F\subset \R^m$, a Borel subset, we have
\[HD(\pi_s(F))=\min\{l, HD(F)\},\]
for almost all $s\in G(m,l)$, with respect to a volume measure on $G(m,l)$. In the particular case when $F=K_1\times ....\times K_n$ is a product of regular Cantor sets, one has $HD(K_1\times ....\times K_n)=HD(K_1)+...+HD(K_n)$. Thus, our formula, Eq. (\ref{eq:DimFor}), corresponds to Marstrand formula replacing $\pi_s$ by $h$. The difference between our result and the classical Marstrand theorem is that our theorem gives explicit hypothesis for $h$ in order to obtain the dimension formula, our theorem is not an "almost all" result, it holds under explicit generic conditions on the map $h$ and the Cantor sets. The formal statement of the dimension formula proven in this paper is the following:
\begin{dmf}
Let $K_1$,..., $K_n$ be $C^m$, $m\geq 2$, conformal regular Cantor sets generated by expanding maps $g_1$,...,$g_n$, respectively. Suppose all of them are not essentially affine. Assume that there exist periodic points $p_j \in K_j$, with period $n_j$, for $1\leq j\leq n$, such that if we write
\[Dg_j^{n_j}(p_j)=\frac{1}{r_j} R_{-v_j},\]
where $R_v$ is the rotation matrix by an angle $v\in \T$, then
\begin{align*}
(\log r_1,0,..,0&;v_1,0,..,0),\\
&\vdots \\
(0,...,\log r_{n-1}&;0,...,v_{n-1},0),\\
(-\log r_n,...,-\log r_n &;0,...,0,v_n),
\end{align*}
generate a dense subgroup of $\R^{n-1}\times \T^{n}$. Let $h$ be any $C^{1}$ function defined on a neighborhood of $K_1\times...\times K_n$ into $\R^l$ such that there exist a point $x_0\in K_1\times ...\times K_n$ where $Dh(x_0)$ verifies the transversality hypotheses, then
\[HD(h(K_1\times...\times K_n))=\min \{l,HD(K_1)+...+HD(K_n)\}.\]
\end{dmf}
The transversality hypotheses means that for any subset $A\subset \{1,...,n\}$, the linear map $Dh(x_0): \C^n\to \R^l$ satisfies
\[dim(Im(Dh(x_0)|_{\{z_j=0:j\notin A\}}))=\min \{l,2\cdot \# A\}.\]
This is the minimum assumption one needs in order to have the dimension formula for linear maps. Proper definitions of all other objects are given in the next section.

This type of problem has already been investigated by other authors, we mention some of them. Peres and Shmerkin \cite{PS} proved that for $K_1, K_2 \subset \R$ attractors for self-affine i.f.s. (iterated functions system) given by maps $\{r_ix+t_i\}_{i=1}^{n}$, $\{r_i'x+t_i'\}_{i=1}^{n'}$, if there are $j,k$ such that $\log (r_j)/ \log (r_k')$ is irrational, then 
\[HD(K_1+ \lambda \cdot K_2)= \min\{1, HD(K_1)+HD(K_2)\},\]
for all $\lambda\neq 0$.

On the other hand, Moreira \cite{M} studied the same formula for $K_1, K_2\subset \R$ regular Cantor sets. He proved that the formula holds provided one of the Cantor sets is not essentially affine. Moreira's proof uses the scale recurrence lemma of \cite{MY}. 

In another work, Hochman and Shmerkin \cite{HS} proved a dimension formula without assuming any type of affinity or non-affinity in the attractors or Cantor sets. They proved (in fact, this is a corollary of their main theorem) that for $K_1,...,K_n$ attractors for i.f.s. on $\R$, one has
\[HD(\lambda_1 K_1+...+\lambda_n K_n)= \min\{1, HD(K_1)+...+HD(K_n)\},\]
for all $\lambda_i\neq 0$, $i=1,...,n$, provided that a certain set is dense in the group $(\R^n,+)/\Delta$, where $\Delta$ is the diagonal. This set depends on the derivative of the contractions of the i.f.s. on periodic points. The technique used by Hochman and Shmerkin is different from the approach of Moreira.

Apart from the motivations given by Marstrand theorem and dynamical systems, there are other reasons to study sets of the form $K_1+K_2$, where $K_1, K_2$ are dynamically defined Cantor sets. There are applications in number theory as well. In \cite{M2}, Moreira used his dimension formula to prove that fractal dimensions of the Lagrange spectrum grow continuously. More precisely, he proved that the function
\[d(t)=HD(L\cap (-\infty , t)),\]
where $L\subset \R$ is the Lagrange spectrum, is continuous.

In this paper, we will adapt the methods used by Moreira and Yoccoz to the context of Cantor sets in the complex plane. We will consider an arbitrary finite number of Cantor sets, not just two. This will lead us to consider a new type of scale recurrence lemma for complex Cantor sets. In section 2 we give basic definitions and results. In this section we state, without proof, the scale recurrence lemma. Section 3 is dedicated to the proof of the dimension formula. Finally, in section 4 we prove the scale recurrence lemma.

\subsection{Notation}

We briefly present some notation and conventions used in the paper.
\begin{itemize}
    \item The equality $P:=Q$ will mean that the quantity $P$ is defined by the expression $Q$.
    \item We will work with the space $\C^n$ and identify it with $\R^{2n}$. We will usually write elements of $\C^n$ as $(z_1,...,z_n)$, where $z_j\in \C$. When we say that $B:\C^n \to \R^l$ is linear, we mean $\R$-linear.
    \item We consider $\R^m$ endowed with the $l^2$ norm, denoted by $|\cdot|$. For $x=(x_1,...,x_m)\in \R^m$, the $l^2$ norm is given by $|x|^2:=x_1^2+...+x_m^2$.
    \item For a set $X$ contained in a metric space $(Y,d_Y)$, we will denote its diameter by $diam(X)$. For $\delta>0$, we will consider the $\delta$-neighborhood of $X$, which is defined by
    \[V_{\delta}(X)=\{y\in Y: d_Y(y,X)<\delta\}.\]
    The open ball centered at $x$ with radius $\delta$ is denoted by $B_{\delta}(x)$. Given $X_1, X_2 \subset Y$, we denote the distance between $X_1$ and $X_2$ by $dist(X_1,X_2)$.
    \item We denote the Hausdorff dimension of a set $X$ by $HD(X)$.
    \item Given a finite set $X$, we denote the number of its elements by $\# X$.
    \item For a set $X$ in a topological space, we denote its closure by $\overline{X}$ and its interior by $int(X)$.
    \item Given a linear map $A:\R^m\to \R^k$ between Euclidean spaces, we usually identify $A$ with its matrix representation with respect to the canonical bases. $\|A\|$ will denote the norm of $A$, which is defined as $\sup_{x\neq 0} \frac{|Ax|}{|x|}$. We will also use the minimum norm of $A$ (which is not a norm), denoted by $m(A)$ and defined as $\inf_{x\neq 0} \frac{|Ax|}{|x|}$. We will say that $A$ is conformal if it is a linear isomorphism and $\|A\|=m(A)$, this implies that $A$ preserves angles and $|Ax|=\|A\|\cdot |x|$ for all $x\in \R^m$.
    \item $\T$ denotes the space $\R/(2\pi\Z)$. It is a commutative group, we endow $\T$ with the unique invariant distance giving diameter $2\pi$ to the space. Denote by $\|v\|$ the distance from $v\in \T$ to the zero element $0\in \T$.
    \item Given functions $\phi$ and $\psi$ with intersecting domains, we denote by $\|\phi-\psi\|$ the supremum distance between the functions. This is
    \[\|\phi-\psi\|=\sup_{x\in Dom(\phi) \cap Dom(\psi)} |\phi(x)-\psi(x)|,\]
    where $Dom(\phi)$, $Dom(\psi)$ are the domains of $\phi$, $\psi$, respectively. However, when $A$, $B$ are linear maps then $\|A-B\|$ will denote the norm of the linear map $A-B$.
\end{itemize}

\section{Basic Definitions}\label{BDef}

In this section we define the objects and present the principal tools that will play a role in the paper. Most of the proofs of the facts stated in this section follow from standard techniques, thus we leave them without proof. For proofs we refer the reader to chapter 1 of \cite{Z}.

\subsection{Conformal Regular Cantor Set}
A $C^m$ \textit{regular Cantor set} (or \textit{dynamically defined Cantor set}) on the complex plane is given by the following data:
\begin{itemize}
\item A finite set $\mathbb{A}$ of \textit{letters} and a set $B\subset \mathbb{A}\times \mathbb{A}$ of \textit{admissible} pairs.
\item For each $a\in \mathbb{A}$ a compact connected set $G(a)\subset \C$.
\item A $C^m$ function $g:V\to \C$ defined in an open neighbourhood $V$ of $\bigsqcup_{a\in \mathbb{A}} G(a)$.
\end{itemize}
This data must verify the following assumptions:
\begin{itemize}
\item The sets $G(a)$, $a\in \mathbb{A}$, are pairwise disjoint.
\item $(a,b)\in B$ implies $G(b)\subset g(G(a))$, otherwise $G(b)\cap g(G(a))=\emptyset$.
\item For each $a\in \mathbb{A}$ the restriction $g|_{G(a)}$ can be extended to a $C^m$ diffeomorphism from an open neighbourhood of $G(a)$ onto its image such that $m(Dg)>1$ (where $m(A)=\inf_{v\neq 0} \frac{|Av|}{|v|}$ is the minimum norm of the linear map $A$).
\item The subshift $(\Sigma, \sigma)$ induced by $B$ 
\[\Sigma^+=\{\ul{a}=(a_0,a_1,...)\in \mathbb{A}^{\mathbb{N}}:(a_i,a_{i+1})\in B,\,\,\forall i\geq 0\},\]
$\sigma(a_0,a_1,a_2,...)=(a_1,a_2,...)$ is topologically mixing.
\end{itemize}

Once we have such data we can define a Cantor set (i.e. totally disconnected, perfect compact set) on the complex plane
\[K=\bigcap_{n\geq 0}g^{-n}\left(\bigsqcup_{a\in \mathbb{A}} G(a)\right).\]
We will say that the regular Cantor set is conformal if for all $x\in K$ the linear map $Dg(x):\R^2\to \R^2$ is conformal. We will write only $K$ to represent all the data that takes to define a conformal regular Cantor set. All Cantor sets in this paper will be conformal regular Cantor sets, we will usually refer to them just as Cantor sets. 

The degree of differentiability, $m$, can be any real number bigger than one. If $m$ is not an integer then $g$ being $C^m$ means that it is $C^{[m]}$, where $[m]$ is the integer part of $m$, and $D^{[m]}g$ is Holder with exponent $m-[m]$. To prove our results we will assume that $m\geq 2$.

We can actually suppose that the sets $G(a)$ verify $G(a)=\overline{int(G(a))}$, this assumption will be relevant when proving the dimension formula because we want the sets $G(a)$ to contain volume.

\begin{lem}\label{lem:g*}
Let $K$ be a conformal Cantor set, then there exist a family of sets $G^*(a)\subset \C$, for $a\in \mathbb{A}$, such that:
\begin{itemize}
\item[(i)] $G^*(a)$ is open and connected.
\item[(ii)] $G(a)\subset G^*(a)$, and $g|_{G(a)}$ can be extended to an open neighborhood of $\overline{G^*(a)}$, such that it is a diffeomorphism from this neighbourhood onto its image and $m(Dg)>1$.
\item[(iii)] The sets $\overline{G^*(a)}$, $a\in \mathbb{A}$, are pairwise disjoint. 
\item[(iv)] $(a,b)\in B$ implies $\overline{G^*(b)}\subset g(G^*(a))$, and $(a,b)\notin B$ implies $\overline{G^*(b)}\cap \overline{g(G^*(a))}=\emptyset$. 
\end{itemize}
\end{lem}

\subsection{Limit Geometry}
Associated to a Cantor set $K$ we define the sets
\begin{align*}
	\Sigma^{fin}&=\{(a_0,...,a_n):(a_i,a_{i+1})\in B \},\\
	\Sigma^{-}&=\{(...,a_{-n},...,a_{-1},a_0): (a_i,a_{i+1})\in B\}.
\end{align*}
Given $\ul{a}=(a_0,...,a_n)$, $\ul{b}=(b_0,...,b_m)$, $\tb^1=(...,\theta^1_{-1},\theta^1_0)$, $\tb^2=(...,\theta^2_{-1},\theta^2_0)$, we will use the following notations:
\begin{itemize}
\item If $a_n=b_0$, $\ul{a}\ul{b}=(a_0,...,a_n,b_1...,b_m)$.
\item If $\theta^1_0=a_0$, $\tb^1\ul{a}=(...,\theta^1_{-1},\theta^1_0,a_1,...,a_n)$.
\item If $\theta^1_0=\theta^2_0$, $\tb^1\wedge \tb^2= (\theta^1_{-j},...,\theta^1_{0})$, where $j$ is such that $\theta^1_{-i}=\theta^2_{-i}$, for all $0\leq i\leq j$, and $\theta^1_{-j-1}\neq \theta^2_{-j-1}$.
\item If $\theta^1_0=a_n$, $\tb^1\wedge \ul{a}= (\theta^1_{-j},...,\theta^1_{0})$, where $j$ is such that $\theta^1_{-i}=a_{n-i}$, for all $0\leq i\leq j$, and $\theta^1_{-j-1}\neq a_{n-j-1}$.
\end{itemize}

For $\ul{a}=(a_0,...,a_n)\in \Sigma^{fin}$ define
\[G(\ul{a})=\{x\in \bigsqcup_{a\in \mathbb{A}} G(a) : g^j(x)\in G(a_j),\,\,j=0,1,..,n\},\]
and the function $f_{\ul{a}}:G(a_n) \to G(\ul{a})$ given by
\[f_{\ul{a}}=g|_{G(a_0)}^{-1}\circ ... \circ g|_{G(a_{n-1})}^{-1}.\]
Denote by $K(\ul{a})$ the set $K\cap G(\ul{a})$. For each $a\in \mathbb{A}$ we choose an arbitrary point $c_a \in K(a)$, using this, define $c_{\ul{a}}\in G(\ul{a})$ by
\[c_{\ul{a}}=f_{\ul{a}}(c_{a_n}).\]

Notice that $Df_{\ul{a}}(c_{a_n})$ is a conformal matrix in $\R^2$, then it is equal to a positive real number times a rotation matrix, denote the angle of rotation by $v_{\ul{a}}\in \R/(2\pi \Z).$ In this way we have a preferred point and direction for each $G(\ul{a})$. We also define \[r_{\ul{a}}=diam(G(\ul{a})).\]

Given $\tb=(...,\theta_{-n},...,\theta_{-1},\theta_{0})\in\Sigma^{-}$, let $\tb^n=(\theta_{-n},...,\theta_0)$ and define $k^{\tb}_n:G(\theta_0)\to \C$ by
\[k^{\tb}_n=\phi_{\tb^n}\circ f_{\tb^n},\]
where $\phi_{\tb^n}$ is the unique map in 
\[Aff(\C)=\{A(z)=az+b:a,b\in\C,\,\,a\neq 0\}\]
such that $\phi_{\tb^n}(c_{\tb^n})=0$, $D\phi_{\tb^n}(c_{\tb^n})e^{iv_{\tb^n}}\in \mathbb{R}^+$, $diam(\phi_{\tb^n}(G(\tb^n)))=1$. For the next theorem we consider $k^{\tb}_n$ extended to a small open neighborhood $G^*(\theta_0)$ of $G(\theta_0)$ (as in lemma \ref{lem:g*}).

\begin{thm}
Let $K$ be a $C^m$ conformal Cantor set. For any $\tb\in \Sigma^-$, the family of functions $k^{\tb}_n:G^*(\theta_0)\to \C$ converges in the $C^{[m]}$ topology, with an exponential rate of convergence independent of $\tb$, to a $C^m$ function $k^{\tb}:G^*(\theta_0)\to \C$. The function $k^{\tb}$ is a diffeomorphism onto its image and the derivative $Dk^{\tb}(x)$ is conformal for all $x\in K(\theta_0)$.

Moreover, if $m\geq 2$ then there is a constant $C>0$ such that given $\tb^1, \tb^2 \in \Sigma^-$ ending with the same letter
\begin{align*} 
\sup_z \left[|k^{\tb^1}\circ (k^{\tb^2})^{-1}(z)-z|+ \|D(k^{\tb^1}\circ (k^{\tb^2})^{-1})(z)-I\|\right]\leq C diam (G(\tb^1\wedge \tb^2)).
\end{align*}
\end{thm}

The function $k^{\tb}$ is called a limit geometry of $K$. Notice that the convergence being independent of $\tb$ implies that $D^l k^{\tb}$, for $0\leq l\leq [m]$, depends continuously on $\tb$.

For $\tb \in \Sigma^{-}$, $\ul{a}\in \Sigma^{fin}$, such that $\ul{a}$ starts with the last letter of $\tb$, define
\begin{align*}
&G^{\tb}(\ul{a})=k^{\tb}(G(\ul{a})),\,\, K^{\tb}(\ul{a})=k^{\tb}(K(\ul{a})),\,\, c^{\tb}_{\ul{a}}=k^{\tb}(c_{\ul{a}})\\
&exp(iv^{\tb}_{\ul{a}})=\frac{Dk^{\tb}(c_{\ul{a}})}{\|Dk^{\tb}(c_{\ul{a}})\|} exp(iv_{\ul{a}}),\,\, r^{\tb}_{\ul{a}}=diam(G^{\tb}(\ul{a})).
\end{align*}
Let $F_{\ul{a}}^{\tb}$ be the affine map determined by the equation
\[k^{\tb}\circ f_{\ul{a}}=F_{\ul{a}}^{\tb}\circ k^{\tb\ul{a}}.\]
$F_{\ul{a}}^{\tb}$ maps $0$ to $c_{\ul{a}}^{\tb}$ and can be written using $r_{\ul{a}}^{\tb}\in \mathbb{R}^+$, $v_{\ul{a}}^{\tb}\in \mathbb{R}/2\pi\mathbb{Z}$ as
\[F_{\ul{a}}^{\tb}(z)=r_{\ul{a}}^{\tb} exp(iv_{\ul{a}}^{\tb})z+c_{\ul{a}}^{\tb}.\]
\begin{defn}\label{def:NEssAff}
We will say that a $C^m$, $m\geq 2$, Cantor set $K$ is \textit{not essentialy affine} if there exist $\tb^1,\tb^2\in \Sigma^{-}$, ending in the same letter, and $z_0\in K^{\tb^2}(\theta^2_0)$ such that
\[D^2(k^{\tb^1}\circ (k^{\tb^2})^{-1})(z_0)\neq 0.\]
\end{defn}

\subsection{Mass Distribution Principle}

Typically, estimating the Hausdorff dimension from below is harder than from above. One usual technique is the mass distribution principle that we state below.
\begin{prop}
Let $F\subset \R^l$ be a Borel measurable set, $\nu$ a Borel measure with $\nu(F)>0$ and $a,b>0$, $s>0$ such that 
\[\nu(u)\leq a\cdot diam(u)^s,\]
for all $u$ measurable with $diam(u)<b$. Then $H_s(F)\geq \nu(F)/a$, in particular $HD(F)\geq s$.
\end{prop}

The next proposition is a consequence of the mass distribution principle and it will be used to prove the desired dimension formula. Its proof is not difficult and can be found in section 1.3 of \cite{Z}.

Let $N$ be the node set of a rooted tree with the property that every node has finite index. $N$ can be described in the following way: there is a marked element $p_0\in N$ called the root of $N$, for each $p\in N$ we have a finite set $Ch(p)\subset N$ called the children of $p$, if $p\neq q$ then $Ch(p)\cap Ch(q)=\emptyset$, for any $q\in N$ there is a sequence $q_0,q_1,...,q_m$ such that $q_0=p_0$, $q_m=q$, and $q_{i+1}\in Ch(q_i)$, $i=0,...,m-1$, such $q$ is called an $m$-level node of $N$. Denote by $I(k)$ the set of $k$-level nodes. $N$ can be written as the disjoint union
\[N=\sqcup_{k=0}^{\infty} I(k).\]

\begin{cor}\label{distribucion}
Suppose we have a set $N$ as described above and assume that for each $p\in N$ we have a Borel measurable set $G(p)\subset \R^l$ with the following properties:
\begin{itemize}
\item[(a)] If $p\in Ch(q)$ then $\overline{G(p)}\subset G(q)$.
\item[(b)] If $p_1,p_2\in Ch(q)$, $p_1\neq p_2$, then $\overline{G(p_1)}\cap \overline{G(p_2)}=\emptyset$.
\item[(c)] The supremum $\sup \{ diam(G(p)): p\in I(k)\}$ goes to zero as $k$ goes to infinity.
\item[(d)] There is a constant $\mu>1$ such that for any $p\in Ch(q)$ we have $diam(G(p))\geq \mu^{-1} diam(G(q))$.
\item[(e)] There is a contant $\mu>1$ such that for any $p\in N$ the set $G(p)$ contains a ball of radius $\mu^{-1}diam(G(p))$.
\item[(f)] There is a number $s>0$ such that for any $q\in N$
\[\sum_{p\in Ch(q)} diam(G(p))^s\geq diam(G(q))^s.\]
\end{itemize}
Let $F$ be the set
\[F=\bigcap_{k=0}^{\infty} \bigcup_{p\in I(k)} G(p).\]
Then $HD(F)\geq s$. 
\end{cor}

\subsection{Not Essentially Real Cantor Sets}

In this subsection we will present a hypothesis in the Cantor set that will guarantee it is indeed two dimensional. This will be important since it will imply that the renormalization operators, defined in the next section, are acting on the right space.  

\begin{defn}
We will say that a Cantor set $K$ is \textit{essentially real} if there exists $\tb\in\Sigma^-$ such that the limit Cantor set $K^{\tb}(\theta_0)$ is contained in a straight line. 
\end{defn}

It is not difficult to prove that $K$ is essentially real if and only if for every $\tb\in\Sigma^-$ the limit Cantor set $K^{\tb}(\theta_0)$ is contained in a straight line.  Moreover, one can prove that $K$ being essentially real is equivalent to $K$ being contained in a $C^1$ one dimensional manifold embedded on the plane. For the proof of the scale recurrence lemma we will suppose that the Cantor set is not essentially real, in such case one is able to control the quantity of elements $G^{\tb}(\ul{a})$ close to an arbitrary line, this is the content of the next lemma.

Given $c>0,\rho>0$ define
\[\Sigma(c,\rho)=\{\ul{a}\in\Sigma^{fin}:\,\,c^{-1}\rho\leq r_{\ul{a}}\leq c \rho\},\]
we can think of this as the set of $G(\ul{a})$ having approximate size $\rho$. Using standard techniques (see \cite{PT95} or \cite{Z}) one can prove that there is a constant $C>0$, depending only in $c$ and the Cantor set $K$ and not depending on $\rho$, such that
\[C^{-1} \rho^{-HD(K)} \leq \#\Sigma(c,\rho)\leq C \rho^{-HD(K)}.\]
Suppose we have fixed a constant $C_5>0$. Let $(a,b)\in B$, a subset $D\subset \Sigma(C_5,\rho)$ is called a {\em discretization} of $K(a,b)$ of order $\rho$ if
\[\bigcup_{\ul{a}\in D}K(\ul{a})=K(a,b).\]
\begin{lem}\label{lem:cono}
Let $K$ be a Cantor set not essentially real. There exist an angle $\alpha \in (0,\pi/2)$ and numbers $\rho_2>0$, $a\in (0,1)$, depending only on $C_5$ and the Cantor set $K$, such that for any limit geometry $k^{\tb}$, $x\in G^{\tb}(\theta_0)$, line $L$, $s\in \mathbb{A}$, $D$ discretization of $K(\theta_0, s)$ of order less than $\rho_2$
\[\#\{\ul{a}\in D:\,\, G^{\tb}(\ul{a})\cap Cone(x,L,\alpha)\neq \emptyset \}\leq a\cdot \#D.\]
Where $Cone(x,L,\alpha)$ is the set of $z\in \C$ such that the vector $z-x$ forms an angle of measure less than $\alpha$ with the line $L$.
\end{lem}
Another use of the not essentially real hypothesis will be given in the next lemma. Let $K$ be a Cantor set, for $x \in K$ consider the set
\[K^{dir}_x :=\bigcap_{\delta>0} \overline{\left\{\frac{y-x}{|y-x|}:y\in B_{\delta}(x)\cap (K\setminus \{x\})\right\}}.\]
If $K$ is not essentially real then the set $K^{dir}_x$ has two linearly independent vectors (over $\R$) and then the following lemma holds for $K$.

\begin{lem}\label{lem:HigherDerivative}
Let $K$ be a Cantor set and $f$ a $C^l$ function from a neighborhood of $K$ into $\R^2$. Suppose that $f$ is conformal at $K$, i.e. $Df(x)$ is conformal for all $x\in K$, and $K^{dir}_x$ has two linearly independent vectors (over $\R$), for all $x\in K$. Then, for all $x\in K$, the $l$-linear map $D^l f(x):\R^2\times ...\times \R^2 \to \R^2$ is conformal, i.e. there is a complex number $c^l_x$ such that
\[D^l f(x)(z_1,...,z_l)=c^l_x\cdot z_1 \cdot z_2 \cdot ... \cdot z_l.\]
The operation $\cdot$ in the right hand side of the last equality corresponds to complex multiplication.
\end{lem}

In particular, if a Cantor set is not essentially real and not essentially affine then for the values $z_0 \in K$, $\tb^1,\tb^2\in \Sigma^{-}$, given by definition \ref{def:NEssAff}, there is a non zero complex number $d_0$ such that
\[D^2(k^{\tb^1}\circ (k^{\tb^2})^{-1})(z_0)(v,w)=d_0\cdot v \cdot w.\]

\subsection{Renormalization Operator}

From now on we will be working with a finite set of conformal regular Cantor sets $K_1,...,K_n$. To each one of them we have various objects associated, as defined in the previous subsections. We will use subscripts and superscripts to differentiate the objects from one Cantor set to the other. For example we use $\Sigma_j(c,\rho)$ for the set $\Sigma(c,\rho)$, which was defined in the last subsection, associated to the Cantor set $K_j$. We will denote by $d_j$ the Hausdorff dimension of the Cantor set $K_j$. In this section we will define renormalization operators, which are operators associated to the family $K_1,...,K_n$ of Cantor sets. 

Consider the space
\begin{align*}
J=\R^{n-1}\times \T^n,
\end{align*}
where $\T=\R/(2\pi \Z)$. $J$ is an abelian locally compact group, we put on it the unique invariant metric such that the distance between $(t_1,...,t_{n-1},v_1,...,v_n)$ and the zero element is $\max_j\{|t_j|,\|v_j\|\}$.

For $(\ul{b}^1,...,\ul{b}^n)\in \Sigma^{fin}_1\times ... \times \Sigma^{fin}_n$ define the operator
\begin{align*}
T_{\ul{b}^1,...,\ul{b}^n}:\Sigma^-_1\times ...\times \Sigma^-_n \times J \to \Sigma^-_1\times ...\times \Sigma^-_n \times J,
\end{align*}
given by
\begin{align*}
T_{\ul{b}^1,...,\ul{b}^n}&(\tb^1,...,\tb^n,t_1,...,t_{n-1},v_1,...,v_n)\\
&=(\tb^1\ul{b}^1,...,\tb^n\ul{b}^n,t_1+\log \frac{r_{\ul{b}^1}^{\tb^1}}{r_{\ul{b}^n}^{\tb^n}},...,t_{n-1}+\log \frac{r_{\ul{b}^{n-1}}^{\tb^{n-1}}}{r_{\ul{b}^n}^{\tb^n}},v_1+v_{\ul{b}^1}^{\tb^1},...,v_n+v_{\ul{b}^n}^{\tb^n}).
\end{align*}
These are called renormalization operators. They will appear in the statement of the scale recurrence lemma. For $r>0$ we also define the set
\[J_r=\{(t_1,...,t_{n-1})\in\R^{n-1}: |t_j|\leq r,\,\, 1\leq j\leq n-1\}\times \T^n,\]
and denote by $\nu$ the Haar measure on $J$ giving measure $(2\pi)^n$ to the set $J_{1/2}$.

\subsection{Scale Recurrence Lemma}

In this subsection we state one of the principal results in the paper. This is a multidimensional conformal version of the scale recurrence lemma of Moreira, Yoccoz \cite{MY}. The proof is technical and will be left for the end of the paper.

\begin{scl}\label{lem:scl}
Let $K_1$, $K_2$,..., $K_{n}$ be $C^m$ conformal regular Cantor sets with $m\geq 2$. Suppose they are not essentially affine and not essentially real. Denote by $d_j$ the Hausdorff dimension of $K_j$, $1\leq j\leq n$. If $r,c_0$ are conveniently large, there exist $c_1,c_2,c_3,\rho_0>0$ with the following properties: given $0<\rho<\rho_0$, and a family $F(\ul{a}^1,...,\ul{a}^n)$ of subsets of $J_r$, $(\ul{a}^1,...,\ul{a}^n)\in \Sigma_1(c_0,\rho)\times ...\times \Sigma_n(c_0,\rho)$, such that
\[\nu(J_r\setminus F(\ul{a}^1,...,\ul{a}^n))\leq c_1, \forall (\ul{a}^1,...,\ul{a}^n),\] 
there is another family $F^*(\ul{a}^1,...,\ul{a}^n)$ of subsets of $J_r$ satisfying:
\begin{itemize}
\item[(i)] For any $(\ul{a}^1,...,\ul{a}^n)$, $F^*(\ul{a}^1,...,\ul{a}^n)$ is contained in the $c_2 \rho$-neighborhood of $F(\ul{a}^1,...,\ul{a}^n)$.
\item[(ii)] Let $(\ul{a}^1,...,\ul{a}^n)\in \Sigma_1(c_0,\rho)\times ...\times \Sigma_n(c_0,\rho)$, $(t,v)\in F^*(\ul{a}^1,...,\ul{a}^n)$; there exist at least $c_3\rho ^{-(d_1+d_2+...+d_n)}$ tuples $(\ul{b}^1,...,\ul{b}^n)\in \Sigma_1(c_0,\rho)\times ...\times \Sigma_n(c_0,\rho)$ (with $\ul{b}^1$,..., $\ul{b}^n$ starting with the last letter of $\ul{a}^1$,..., $\ul{a}^n$) such that, if $\tb^i \in \Sigma_i^{-}$ ends with $\ul{a}^i$, i=1,...,n, and
\[T_{\ul{b}^1,...,\ul{b}^n}(\tb^1,...,\tb^n,t,v)=(\tb^1\ul{b}^1,...,\tb^n\ul{b}^n,\tilde{t},\tilde{v})\]
the $\rho$-neighborhood of $(\tilde{t},\tilde{v})\in J$ is contained in $F^*(\ul{b}^1,...,\ul{b}^n)$.
\item[(iii)] $\nu(F^*(\ul{a}^1,...,\ul{a}^n))\geq \nu(J_r)/2$ for at least half of the $(\ul{a}^1,...,\ul{a}^n)\in \Sigma_1(c_0,\rho)\times ...\times \Sigma_n(c_0,\rho)$.
\end{itemize}
\end{scl}

\section{Dimension Formula}

In this section we will prove the dimension formula (theorem B), which is one of the main theorems on the paper. First, we will introduce some notation. Secondly, we will present the discrete Marstrand property, which will be an important tool. Finally, we give the proof of theorem B.

We will use the notation
\[R(\rho)= \Sigma_1(c_0,\rho)\times ...\times \Sigma_n(c_0,\rho),\]
and think of any element $Q=(\ul{a}^1,...,\ul{a}^n)\in R(\rho)$ as the set $G(\ul{a}^1)\times...\times G(\ul{a}^n)$. Given a function $\varphi$ defined on a neighborhood of $G(\ul{a}^1)\times...\times G(\ul{a}^n)$ we write $\varphi(Q)$ to denote the set $\varphi(G(\ul{a}^1)\times...\times G(\ul{a}^n))$.\\

To each $(t,v)\in J$ we associate the linear map $A(t,v):\C^n\to \C^n$ given by
\[A(t_1,...,t_{n-1},v_1,...,v_n)(z_1,...,z_n)=(e^{t_1+iv_1}\cdot z_1,...,e^{t_{n-1}+iv_{n-1}}\cdot z_{n-1},e^{iv_n}z_n).\]
We also consider the composition of these maps with limit geometries
\[\pi_{\tb^1,...,\tb^n,t,v}:= A(t,v)\circ (k^{\tb^1},...,k^{\tb^n}).\]
These maps are related to the renormalization operators by the following equation
\begin{equation}\label{eq:ren}
    \pi_{\tb^1,...,\tb^n,t,v}\circ (f_{\ul{a}^1},...,f_{\ul{a}^n})= B \circ \pi_{T_{\ul{a}^1,...,\ul{a}^n}(\tb^1,...,\tb^n,t,v)},
\end{equation}
where $B:\C^n\to \C^n$ is an affine function of the form $B(z)=\alpha\cdot z + \beta$ for $\alpha \in \R$. In fact, this equation is the reason why we defined the renormalization operators as we did.

One of the main reasons to use limit geometries is that they appear naturally when one consider compositions of a $C^1$ function with the maps $f_{\ul{a}}$. This is explained in the next lemma.
\begin{lem}\label{ren:atr}
Let $h$ be a $C^1$ function defined from a neighborhood of $K_1\times ...\times K_n$ into $\R^l$, and $r>0$. There exists a function $E:(0,\infty)\to \R$, depending only on $h$, $r$ and the Cantor sets, such that $\lim_{t\to 0}E(t)=0$ and with the following property:

For any $(\ul{a}^1,...,\ul{a}^n)$  such that
\[s=(\log \frac{r_{\ul{a}^1}}{r_{\ul{a}^n}},...,\log \frac{r_{\ul{a}^{n-1}}}{r_{\ul{a}^n}},v_{\ul{a}^1},...,v_{\ul{a}^n})\in J_r,\]
consider the affine function $L:\R^{l}\to \R^{l}$ given by
\[L(z)=\frac{1}{r_{\ul{a}^n}}(z-h(c_{\ul{a}^1},...,c_{\ul{a}^n})).\]
Then for any $\tb^1,...,\tb^n$ ending in $\ul{a}^1,...,\ul{a}^n$ the supremum distance between $L\circ h \circ (f_{\ul{a}^1},...,f_{\ul{a}^n})$ and
\[Dh(c_{\ul{a}^1},...,c_{\ul{a}^n})\circ A(s)\circ (k^{\tb^1},...,k^{\tb^n})\]
is less than $E(\max_{1\leq j\leq n} r_{\ul{a}^j})$.
\end{lem}
This lemma is saying that $h\circ (f_{\ul{a}^1},...,f_{\ul{a}^n})$, modulo composition by an affine function on the left, becomes arbitrarily close to a function of the form $B\circ A(t,v)\circ (k^{\tb^1},...,k^{\tb^n})$.
\begin{proof}
Write $h\circ (f_{\ul{a}^1},...,f_{\ul{a}^n})$ as
\[[h\circ (\phi_{\ul{a}^1}^{-1},...,\phi_{\ul{a}^n}^{-1})] \circ (\phi_{\ul{a}^1}\circ f_{\ul{a}^1},...,\phi_{\ul{a}^n}\circ f_{\ul{a}^n}).\]
Use the fact that $(\phi_{\ul{a}^1}\circ f_{\ul{a}^1},...,\phi_{\ul{a}^n}\circ f_{\ul{a}^n})$ becomes close to a limit geometry $(k^{\tb^1},...,k^{\tb^n})$ and Taylor first order approximation for $h\circ (\phi_{\ul{a}^1}^{-1},...,\phi_{\ul{a}^n}^{-1})$.
\end{proof}

\subsection{Discrete Marstrand Property}\label{Discrete Marstrand Theorem}

In this section we present and prove the discrete Marstrand property. We first state two linear algebra results that we will need.

\begin{lem}
Let $A:\R^n\to \R^d$ be a linear map, $A\neq 0$, and denote by $\sigma$ the smallest non-zero singular value of $A$. Then
\[dist(x,Ker(A))\leq \frac{|Ax|}{\sigma},\]
for all $x\in \R^n$.
\end{lem}
\begin{lem}
Let $E_1,E_2 \subset \R^n$ be linear subspaces such that $E_1+E_2=\R^n$. Denote by $\theta$ the angle between $E_1$ and $E_2$. Define
\[I=E_1\cap E_2,\,L_1=I^{\perp}\cap E_1,\,L_2=I^{\perp}\cap E_2.\]
Let $x=l_1+v+l_2$, with $l_1\in L_1$, $l_2\in L_2$, $v\in I$, then
\begin{align*}
    |l_1|&\leq \frac{dist(x,E_2)}{\sin  \theta},\\
    |l_2|&\leq \frac{dist(x,E_1)}{\sin  \theta}.
\end{align*}
\end{lem}
The next proposition is the main tool that will allow us to obtain the discrete Marstrand property. Given an $\R$-linear map $B: \C^n\to \R^l$ we will say it satisfies the transversality condition if for any set $A\subset \{1,...,n\}$ we have
\[dim(Im(B|_{\{z_j=0:j\notin A\}}))=\min \{l,2\cdot \# A\},\]
in all this subsection $B$ will denote one such map.

\begin{prop}\label{prop:transversality}
Let $r>0$ and $B:\C^n\to \R^l$ a linear map satisfying the transversality condition. There exists a constant $C$, depending only in $r$ and $B$, such that for any pair of subsets $Q_1,Q_2 \subset \C^n$ we have
\[\nu(\{s\in J_r: B\circ A(s)(Q_1) \cap  B\circ A(s)(Q_2)\neq \emptyset\})\leq C \left(\frac{\max \{diam(Q_1),diam(Q_2)\}}{dist(Q_1,Q_2)}\right)^l.\]
\end{prop}

\begin{proof}
Through out the proof we will use the notation $P=O(Q)$, meaning that there is a constant $\tilde{C}$, depending only in $r$ and $B$, such that $P\leq \tilde{C} \cdot Q$.

Given a subset $A\subset \{1,2,...,n\}$ we consider the subspace
\[\C^A=\{(z_1,...,z_n)\in \C^n: z_j=0,\,\forall j\notin A\}.\]
By the transversality condition we can choose $\theta>0$, only depending on $B$, such that the angle between $Ker(B)$ and $\C^A$ is bigger than $\theta$ for any nonempty subset $A$.

Denote $\max\{diam(Q_1),diam(Q_2)\}$ by $\rho$. Fix $c_1,c_2 \in \C^n$ such that
\[dist(c_j,Q_j)< \rho,\,\, dist(Q_1,Q_2)/2\leq |c_2-c_1|\leq 2 dist(Q_1,Q_2)\] and $c_2-c_1$ has all its coordinates in $\C^n$ different from zero. Suppose that $s\in J_r$ is such that
\[BA(s)(Q_1) \cap  BA(s)(Q_2)\neq \emptyset.\]
Then there are $\tilde{c}_1\in Q_1$, $\tilde{c}_2\in Q_2$ verifying $B\circ A(s)(\tilde{c}_1)=B\circ A(s)(\tilde{c}_2)$. We conclude that
\[|B\circ A(s)(c_2-c_1)|=O(\rho).\]
Define $x=\frac{c_2-c_1}{|c_2-c_1|}$, hence
\[|B\circ A(s)(x)|=O\left(\frac{\rho}{dist(Q_1,Q_2)}\right).\]
By the first lineal algebra lemma we get that
\[dist(A(s)(x), Ker(B))=O\left(\frac{\rho}{dist(Q_1,Q_2)}\right).\]
Up until now we have proven that there is a constant $C_1$, depending only on $r$ and $B$, such that
\[dist(A(s)(x), Ker(B))\leq C_1 \cdot \frac{\rho}{dist(Q_1,Q_2)}.\]
Notice that if
\[\frac{\rho}{dist(Q_1,Q_2)} \geq \frac{e^{-r} \sin \theta}{4 C_1},\]
then the proposition follows taking $C =(2r)^{n-1} (2\pi)^n \left(\frac{4C_1}{e^{-r} \sin \theta}\right)^{l}$. This is thanks to the fact that
\[\nu(\{s\in J_r: B\circ A(s)(Q_1) \cap  B\circ A(s)(Q_2)\neq \emptyset\})\leq (2r)^{n-1} (2\pi)^n.\]
For the rest of the proof we suppose $\frac{\rho}{dist(Q_1,Q_2)} < \frac{e^{-r} \sin \theta}{4 C_1}$. Define $a=\frac{e^{-r} \sin \theta}{4n}$ and write $x\in \C^n$ as
\[x=(e^{\chi_1+i\phi_1},...,e^{\chi_n+i\phi_n}).\]
Consider the set
\[A=\{j\in \{1,...,n\}:e^{\chi_j}\geq a e^{-r}\},\]
and the subspace $\C^A$, we will see that $dim_{\R}\C^A\geq l$. Let $u\in \C^A$ be the orthogonal projection of $A(s)(x)$ in $\C^A$. By the definition of $\C^A$ we have $|A(s)x-u|< na$. Given that $\|A(s)\|\geq e^{-r}$ (remember that $s\in J_r$) and $na<e^{-r}/2$ we have
\[|u|\geq |A(s)(x)|-|A(s)x-u|> e^{-r}-n\cdot a > \frac{1}{2}e^{-r}.\]
If $dim_{\R} \C^A<l$, the transversality implies $\C^A \cap Ker(B)=\{0\}$ and then, by the choice of $\theta$, we would get
\[dist(u,Ker(B))\geq |u|\sin \theta \geq \frac{e^{-r}\sin \theta}{2}.\]
But this is not possible since
\begin{align*}
dist(u,Ker(B))&\leq |u-A(s)(x)|+dist(A(s)(x),Ker(B))\\
&\leq na+ C_1 \cdot \frac{\rho}{dist(Q_1,Q_2)} < \frac{e^{-r}\sin \theta}{2}.
\end{align*}

Given that $dim_{\R}\C^A\geq l$, the transversality condition implies $Ker(B)+\C^A=\C^n$. Let $L=\C^A\cap (Ker(B)\cap \C^A)^{\perp}$.
Define the $\R$-linear function $\hat{x}:\C^n\to \C^n$ given by
\[\hat{x}(z_1,...,z_n)=(e^{\chi_1+i\phi_1}\cdot z_1,...,e^{\chi_n+i\phi_n}\cdot z_n).\]
Notice that $A(s)(x)=\hat{x}([s])$, where \[[s]=(e^{t_1+iv_1},...,e^{t_{n-1}+iv_{n-1}},e^{iv_n}).\]
Write $\hat{x}([s])=b_1+b_2$ where $b_1\in L$, $b_2\in Ker(B)$. The second lemma in linear algebra implies
\[|b_1|\leq \frac{dist(b_1,Ker(B))}{\sin \theta}=\frac{dist(\hat{x}([s]),Ker(B))}{\sin \theta}=O\left(\frac{\rho}{dist(Q_1,Q_2)}\right).\]
Given that $b_1\in \C^A$ we get that $|\hat{x}^{-1}(b_1)|=O\left(\frac{\rho}{dist(Q_1,Q_2)}\right)$. Therefore, $[s]=\hat{x}^{-1}(b_1)+\hat{x}^{-1}(b_2)$ implies
\[dist([s],\hat{x}^{-1}(Ker(B)))=O\left(\frac{\rho}{dist(Q_1,Q_2)}\right).\]
This last inequality tells us that the vector $[s]$ is close to a $2n-l$ subspace. Moreover, the last coordinate of this vector has modulus 1. This two properties will allow us to obtain the desired estimate.

Consider the set
\[H=\{(z_1,...,z_n)\in \hat{x}^{-1}( Ker(B)):\,|z_n|=1,\,|z_j|\in [e^{-2r},e^{2r}],\,j=1,...,n-1\}.\]
We have proven that there is a constant $C_2>0$, depending only on $B$ and $r$, such that
\[dist([s],\hat{x}^{-1}(Ker(B))) \leq C_2 \cdot \frac{\rho}{dist(Q_1,Q_2)}.\]
Thus there is a constant $C_3>0$, depending only on $B$ and $r$, such that
\[dist([s],H) \leq C_3 \cdot \frac{\rho}{dist(Q_1,Q_2)}.\]
In fact, let $u=(u_1,...,u_n)\in \hat{x}^{-1}(Ker(B))$ such that $|u-[s]|\leq C_2\cdot \frac{\rho}{dist(Q_1,Q_2)}$. We have
\[1-C_2 \cdot \frac{\rho}{dist(Q_1,Q_2)}\leq |u_n| \leq 1+C_2 \cdot \frac{\rho}{dist(Q_1,Q_2)},\]
and
\[e^{-r}-C_2 \cdot \frac{\rho}{dist(Q_1,Q_2)} \leq |u_j|\leq e^r+C_2 \cdot \frac{\rho}{dist(Q_1,Q_2)},\]
for $j=1,...,n-1$. Therefore, if $\frac{\rho}{dist(Q_1,Q_2)}$ is small enough one has that $u/|u_n|$ is in $H$ and $|[s]-(u/|u_n|)|=O(\frac{\rho}{dist(Q_1,Q_2)})$. If $\frac{\rho}{dist(Q_1,Q_2)}$ is big the proposition follows choosing $C$ properly, as it was done before when we considered the case $\frac{\rho}{dist(Q_1,Q_2)} \geq \frac{e^{-r} \sin \theta}{4 C_1}$.\\

Define the function
\[\varphi:\{(z_1,...,z_n) \in \mathbb{S}^{2n-1} \cap \hat{x}^{-1}( Ker(B)): \,\frac{|z_j|}{|z_n|}\in [e^{-2r},e^{2r}],\,j=1,...,n-1\} \to H,\]
given by
\[\varphi(z_1,...,z_n)=(z_1/|z_n|,...,z_n/|z_n|).\]
Notice that $\varphi$ is surjective and smooth. Moreover, one has that $\|D\varphi\|$ is bounded by a constant depending only on $r$ and $n$. Since the domain of $\varphi$ is contained in a $(2n-l-1)$-dimensional unit sphere inside $\hat{x}^{-1}(Ker(B))$, there exist $w_1,...,w_p\in H$ such that $H$ is covered by the balls $B_{\frac{\rho}{dist(Q_1,Q_2)}}(w_j)$, $j=1,...,p$, and 
\[p=O\left(\left(\frac{\rho}{dist(Q_1,Q_2)}\right)^{-(2n-l-1)}\right).\]

We conclude that
\[[s]\in \bigcup_{j=1}^{p} B_{C_4\cdot \frac{\rho}{dist(Q_1,Q_2)}}(w_j),\]
for a constant $C_4>0$ depending only in $r$ and $B$.

Writing $w_j$ as
\[w_j=(w_{j,1},...,w_{j,n}),\]
we obtain that for some $1\leq j\leq p$
\[|e^{iv_n}-w_{j,n}|\leq C_4\cdot \frac{\rho}{dist(Q_1,Q_2)},\]
\[|e^{t_q+iv_q}-w_{j,q}|\leq C_4\cdot \frac{\rho}{dist(Q_1,Q_2)},\]
$q=1,...,n-1$. Notice that $|w_{j,q}|$ is bounded below by $e^{-2r}$, hence
\begin{align*}
    |t_q-\log |w_{j,q}|| &\leq C_5\cdot \frac{\rho}{dist(Q_1,Q_2)},\\
    \|v_q-arg(w_{j,q})\|&\leq C_5\cdot \frac{\rho}{dist(Q_1,Q_2)},
\end{align*}
$q=1,...,n$, where $C_5$ is a constant depending only in $r$ and $B$, and $arg(w_{j,q})\in \T$ is the argument of $w_{j,q}$. This implies that the set $\{s\in J_r: B\circ A(s)(Q_1) \cap  B\circ A(s)(Q_2)\neq \emptyset\}$ is contained in the union of $p$ sets, each one with a $\nu$-volume of order $O \left( \left( \frac{\rho}{dist(Q_1,Q_2)}\right)^{2n-1} \right)$. Finally, using the order of $p$, we conclude that
\[\nu(\{s\in J_r: B\circ A(s)(Q_1) \cap  B\circ A(s)(Q_2)\neq \emptyset\})=  O \left(\left(\frac{\rho}{dist(Q_1,Q_2)}\right)^{(2n-1)-(2n-l-1)}\right),\]
as we wanted.
\end{proof}

Proposition \ref{prop:transversality} implies that there is a constant $C>0$ such that for any $\tb^1,...,\tb^n$ and $Q_1,Q_2\in R( \rho)$ we have
\[\nu(\{s\in J_r: B\circ \pi_{\tb^1,...,\tb^n,s}(Q_1) \cap  B\circ \pi_{\tb^1,...,\tb^n,s}(Q_2)\neq \emptyset\})\leq C \left(\frac{\rho}{dist(Q_1,Q_2)}\right)^l.\]
The constant $C>0$ depends on $B$, $r$, $c_0$ and the Cantor sets $K_1$,...,$K_n$, but it is independent of $\rho$. For the next proposition we will also need the following fact: if $p$ is big enough then for any $Q_1\in R(\rho)$ and $a\in \mathbb{Z}$ we have
\[\# \{Q_2\in R(\rho):p^{-a}\leq dist(Q_1,Q_2)< p^{-a+1}\}=O((p^{-a})^{d_1+...+d_n}\rho^{-(d_1+...+d_n)}),\]
where $d_j=HD(K_j)$, $j=1,...,n$. For a proof see lemma 1.2.3 in \cite{Z}.
\begin{prop}
Assume $d_1+...+d_n< l$. Let
\[N_{\rho}(\tb^1,..,\tb^n,s)=\# \{(Q_1,Q_2)\in R(\rho)^2: B\circ \pi_{\tb^1,..,\tb^n,s}(Q_1) \cap  B\circ \pi_{\tb^1,..,\tb^n,s}(Q_2)\neq \emptyset \}.\]
Then for any $\tb^1,...,\tb^n$ we have
\[\int_{J_r}N_{\rho}(\tb^1,...,\tb^n,s)ds=O(\rho^{-(d_1+...+d_j)}),\]
and the constant in the $O$ notation is independent of $\tb^1,...,\tb^n$.
\end{prop}
\begin{proof}
Since the Cantor sets $K_1$,..., $K_n$ have bounded geometries then there is a constant $C_1$, independent of $\rho$, such that $dist(Q_1,Q_2)\geq C_1 \rho$ for any $Q_1,Q_2\in R(\rho)$, $Q_1\cap Q_2=\emptyset$. Let $k_0\in \mathbb{Z}$ such that $p^{-k_0}\leq C_1 \rho < p^{-k_0+1}$. Using the previous lemma we have
\begin{align*}
\int_{J_r} N_{\rho}&(\tb^1,...,\tb^n,s)ds\\
&= \sum_{Q_1,Q_2 \in R(\rho)} \nu(\{s\in J_r: B\circ \pi_{\tb^1,...,\tb^n,s}(Q_1) \cap  B\circ \pi_{\tb^1,...,\tb^n,s}(Q_2)\neq \emptyset \})\\
&=\sum_{Q_1 \in R(\rho)} \sum_{k=-\infty}^{k_0} \sum_{dist(Q_1,Q_2)\in [p^{-k},p^{-k+1})} O(\rho^l/[dist(Q_1,Q_2)]^l)\\
&+ \sum_{Q_1 \in R(\rho)} \sum_{Q_2 \cap Q_1\neq \emptyset} (2r)^{n-1}.
\end{align*}
Clearly \[\sum_{Q_1 \in R(\rho)} \sum_{Q_2 \cap Q_1\neq \emptyset} (2r)^{n-1}=O(\# R(\rho))=O(\rho^{-(d_1+...+d_n)}).\]
On the other hand
\begin{align*}
\sum_{Q_1 \in R(\rho)} &\sum_{k=-\infty}^{k_0} \sum_{dist(Q_1,Q_2)\in [p^{-k},p^{-k+1})} O(\rho^l/[dist(Q_1,Q_2)]^l)=\\
&=\sum_{Q_1 \in R(\rho)} \sum_{k=-\infty}^{k_0} O((p^k)^{l-(d_1+...+d_n)}\rho^{l-(d_1+...+d_n)})\\
&=\sum_{Q_1 \in R(\rho)} O((p^{k_0})^{l-(d_1+...+d_n)}\rho^{l-(d_1+...+d_n)})\sum_{k=-\infty}^{0} (p^{l-(d_1+...+d_n)})^k\\ 
&= O(\rho^{-(d_1+...+d_n)}\rho^{-l+(d_1+...+d_l)}\rho^{l-(d_1+...+d_n)})=O(\rho^{-(d_1+...+d_n)}).
\end{align*}
\end{proof}

\begin{prop} \label{prop:marstrand}
Let $b>0$, $F\subset R(\rho)$ such that $\#F\geq b\rho^{-(d_1+...+d_n)}$. Let $(\tb^1,...,\tb^n,s)$ such that $N_{\rho}(\tb^1,...,\tb^n,s)\leq a \rho^{-(d_1+...+d_n)}$, then there exist a subset $T\subset F$ with the properties that
\[B\circ \pi_{\tb^1,...,\tb^n,s}(Q_1) \cap  B\circ \pi_{\tb^1,...,\tb^n,s}(Q_2)= \emptyset,\]
for all $Q_1,Q_2 \in T$, $Q_1\neq Q_2$, and 
\[\#T \geq \frac{b^2}{4a}\rho^{-(d_1+...+d_n)}.\]
\end{prop}
\begin{proof}
For $Q_0\in F$ define 
\[n(Q_0)=\#\{Q\in F: B\circ \pi_{\tb^1,...,\tb^n,s}(Q) \cap  B\circ \pi_{\tb^1,...,\tb^n,s}(Q_0)\neq \emptyset\}.\]
We have
\[\sum_{Q \in F} n(Q)\leq N_{\rho}(\tb^1,...,\tb^n,s) \leq a \rho^{-(d_1+...+d_n)}. \]
Therefore the set $T_0=\{Q\in F: n(Q)\leq 2a/b\}$ has at least $(1/2)\#F$ elements. Finally it is clear that from $T_0$ we can extract a subset $T$ with at least $\frac{1}{2a/b}\#T_0$ elements and such that $B\circ \pi_{\tb^1,...,\tb^n,s}(Q_1) \cap  B\circ \pi_{\tb^1,...,\tb^n,s}(Q_2)= \emptyset$ for any $Q_1,Q_2 \in T$. For this set we have
\[\#T \geq \frac{1}{2a/b}\#T_0 \geq \frac{1}{2a/b}\frac{1}{2}b \rho^{-(d_1+...+d_n)}=\frac{b^2}{4a}\rho^{-(d_1+...+d_n)}.\]
\end{proof}
Notice that since $\int_{J_r} N_{\rho}(\tb^1,...,\tb^n,s)ds=O(\rho^{-(d_1+...+d_m)})$ then choosing $a$ big enough we can guaranteed that the set 
\[\{s\in J_r: N_{\rho}(\tb^1,...,\tb^n,s)> a\rho^{-(d_1+...+d_n)}\}\]
has measure as small as we want. Thus, for every $(\tb^1,...,\tb^n)$ we have that most of the $s\in J_r$ verify the property of the last proposition, i.e. for any family $F\subset R(\rho)$ with $\# F\geq b\rho^{-(d_1+...+d_n)}$ there exist a positive proportion of $F$, $T\subset F$ with $\#T\geq (b^2/(4a))\rho^{-(d_1+...+d_n)}$, such that elements of $T$ project to $\R^l$, in the direction of $s$, to disjoint sets: $B\circ \pi_{\tb^1,...,\tb^n,s}(Q_1) \cap  B\circ \pi_{\tb^1,...,\tb^n,s}(Q_2)=\emptyset$, $\forall Q_1,Q_2 \in T$. This is what we call the discrete Marstrand property.\\

The next lemma guarantees that the property of Prop. \ref{prop:marstrand} still holds for small perturbations of $B\circ \pi_{\tb^1,...,\tb^n,s}$, it is inspired in the presentation given by Shmerkin \cite{shmerkin}.

\begin{lem}\label{perturbacion}
Let $T\subset R(\rho)$, $\phi$ a function defined on a neighborhood of $\cup_{Q\in T} Q$ into $\R^l$, and $L,\tau>0$ real numbers. Suppose that for each $Q\in T$ we have 
\[B_{L^{-1}\rho}(c_Q)\subset \phi(Q) \subset B_{L\rho}(c_Q),\]
for some $c_Q\in \R^l$, and $\phi(Q_1)\cap \phi(Q_2)= \emptyset$, $\forall Q_1,Q_2\in T$, $Q_1\neq Q_2$.
Then for any $\psi$ defined on a neighborhood of $\cup_{Q\in T} Q$, such that $\|\phi- \psi\|<\tau\rho$, there exist $T'\subset T$ such that 
\[\#T'\geq [3L(L+\tau)]^{-l}\cdot \#T\]
and $\psi(Q_1)\cap \psi(Q_2)= \emptyset$, $\forall Q_1,Q_2\in T'$, $Q_1\neq Q_2$.
\end{lem}
\begin{proof}
$\|\psi-\phi\|<\tau \rho $ implies $\psi(Q)\subset B_{(L+\tau)\rho}(c_Q)$. Use Vitali covering lemma for the family $\{B_{(L+\tau)\rho}(c_Q):Q\in T\}$.
We get $T'\subset T$ such that $\{B_{(L+\tau)\rho}(c_Q):Q\in T'\}$ is a pairwise disjoint family and
\[\bigcup_{Q\in T} B_{(L+\tau)\rho}(c_Q)\subset \bigcup_{Q\in T'} B_{3(L+\tau)\rho}(c_Q).\]
From this we get
\begin{align*}
\# T' \cdot [3(L+\tau)\rho]^l w_l &\geq Vol\left(\bigcup_{Q\in T} B_{(L+\tau)\rho}(c_Q)\right)\\
&\geq Vol\left(\bigcup_{Q\in T} B_{L^{-1}\rho}(c_Q)\right)\\
&=\#T\cdot L^{-l}\rho^l w_l,
\end{align*}
where $w_l$ is the volume of the $l$-dimensional unitary ball. Hence
\[\#T'\geq [3L(L+\tau)]^{-l}\cdot \#T.\]
\end{proof}

\subsection{Proof of the Dimension Formula}

In this subsection we prove the desired dimension formula (theorem B). Assume we have $K_1,...,K_n$ satisfying the hypothesis of the scale recurrece lemma. We start by using the discrete Marstrand property and the scale recurrence lemma to obtain for each limit geometry $(\tb^1,...,\tb^n)$ a set of "good" directions to project. Fix $c_0,r>0$ big enough, let $c_1,c_2,c_3,\rho_0$ be the constants given by the scale recurrence lemma. Suppose that $h$ is a $C^1$ function defined in neighborhood of $K_1\times...\times K_n$ such that there is a point $x_0$, in the product of the Cantor sets, where $B=Dh(x_0)$ verifies the transversality hypotheses. By the results in subsection \ref{Discrete Marstrand Theorem}, we can fix $a>0$ big enough such that
\[\nu(J_r\setminus \{s\in J_r: N_{\rho}(\tb^1,...,\tb^n,s)\leq a\rho^{-(d_1+...+d_n)}\})<c_1,\]
for all $(\tb^1,...,\tb^n)$. Define
\[F(\tb^1,...,\tb^n)=\{s\in J_r: N_{\rho}(\tb^1,...,\tb^n,s)\leq a\rho^{-(d_1+...+d_n)}\},\]
and for $(\ul{a}^1,...,\ul{a}^n)\in \Sigma_{1}(c_0,\rho)\times...\times \Sigma_{n}(c_0,\rho)$
\[F(\ul{a}^1,...,\ul{a}^n)=\bigcup_{\tb^1,...,\tb^n} F(\tb^1,...,\tb^n),\]
where the union is over all $\tb^1,...,\tb^n$ ending in $\ul{a}^1,...,\ul{a}^n$ respectively.
We clearly have $\nu(J_r\setminus F(\ul{a}^1,...,\ul{a}^n))<c_1$, thus we can apply the scale recurrence lemma (from now on we assume $\rho<\rho_0$) to obtain sets $F^*(\ul{a}^1,...,\ul{a}^n)$ with the following properties:
\begin{itemize}
\item[(i)] $F^*(\ul{a}^1,...,\ul{a}^n)\subset V_{c_2\rho}(F(\ul{a}^1,...,\ul{a}^n))$.
\item[(ii)] Let $(\ul{a}^1,...,\ul{a}^n)\in R(\rho)$, $(t,v)\in F^*(\ul{a}^1,...,\ul{a}^n)$; there exist at least $c_3\rho ^{-(d_1+...+d_n)}$ tuples $(\ul{b}^1,...,\ul{b}^n)\in R(\rho)$ (with $\ul{b}^1$,...., $\ul{b}^n$ starting with the last letter of $\ul{a}^1$,..., $\ul{a}^n$) such that, if $\tb^1 \in \Sigma_1^{-}$,..., $\tb^n \in \Sigma_n^{-}$ end respectively with $\ul{a}^1$,..., $\ul{a}^n$ and
\[T_{\ul{b}^1,...,\ul{b}^n}(\tb^1,...,\tb^n,t,v)=(\tb^1\ul{b}^1,...,\tb^1\ul{b}^1,\tilde{t},\tilde{v})\]
the $\rho$-neighborhood of $(\tilde{t},\tilde{v})\in J$ is contained in $F^*(\ul{b}^1,...,\ul{b}^n)$.
\item[(iii)] $\nu(F^*(\ul{a}^1,...,\ul{a}^n))\geq \nu(J_r)/2$ for at least half of the $(\ul{a}^1,...,\ul{a}^n)\in R(\rho)$.
\end{itemize}

\begin{thm}\label{thm:qfdim}
Suppose that $d_1+...+d_n<l$ and for any $(\tb^1,...,\tb^n,s)\in \Sigma_1^-\times...\times \Sigma_n^-\times J_r$ there exists $(\ul{c}^1,...,\ul{c}^n)\in \Sigma_1^{fin}\times ...\times \Sigma_n^{fin}$ such that $T_{\ul{c}^1,...,\ul{c}^n}(\tb^1,...,\tb^n,s)=(\tb^1\ul{c}^1,...,\tb^n\ul{c}^n,\tilde{s})$ and $\tilde{s}\in F^*(\tilde{\ul{a}}^1,...,\tilde{\ul{a}}^n)$ for some $(\tilde{\ul{a}}^1,...,\tilde{\ul{a}}^n)$ for which $(\tb^1\ul{c}^1,...,\tb^n\ul{c}^n)$ ends in it. Let $h$ be any $C^{1}$ function defined on a neighborhood of $K_1\times...\times K_n$ into $\R^l$ such that there exist a point $x_0\in K_1\times ...\times K_n$ where $Dh(x_0)$ verifies the transversality hypotheses, then
\[HD(h(K_1\times...\times K_n))=d_1+...+d_n.\] 
\end{thm}
\begin{proof}
Since $h$ is Lipschitz in a neighborhood of $K_1\times ...\times K_n$ and $HD(K_1\times ...\times K_n)=d_1+...+d_n$ we have 
\[HD(h(K_1\times ...\times K_n))\leq d_1+...+d_n.\]
Thus we only need to show $HD(h(K_1\times ...\times K_n))\geq d_1+...+d_n$. Let $\eta>0$ arbitrary we will prove that 
\[HD(h(K_1\times ...\times K_n))\geq d_1+...+d_n-\eta,\]
this will finish the proof of the theorem.

Since $\Sigma_1^-\times...\times \Sigma_n^-\times J_r$ is compact and $\pi_{\tb^1,...,\tb^n,s}$ depends continuously in $(\tb^1,...,\tb^n,s)$ then we can choose a constant $L>0$ (only depending on $r$, $K_1$,..., $K_n$ and $Dh(x_0)$) such that
\begin{equation}
B_{L^{-1}\rho}(c_Q)\subset Dh(x_0)\circ \pi_{\tb^1,...,\tb^n,s}(Q) \subset B_{L\rho}(c_Q),\forall Q\in R(\rho), \label{eq:open}
\end{equation}
for some $c_Q\in \R^l$, which depends on $\tb^1,...,\tb^n,s$ and $Q$.

Choose $\tau>0$ big enough such that 
\begin{equation}\label{eq:lippi}
\|Dh(x_0)\circ \pi_{\tilde{\tb}^1,...,\tilde{\tb}^n,\tilde{s}}-Dh(x_0)\circ \pi_{\tb^1,...,\tb^n,s}\|\leq \frac{1}{3}\tau \rho,
\end{equation}
for all $\tb^j,\tilde{\tb}^j\in \Sigma_j^-$, $j=1,...,n$, $s,\tilde{s}\in J_r$ with $|s-\tilde{s}|<c_2\rho$, $\tb^j\wedge \tilde{\tb}^j\in \Sigma_j(c_0,\rho)$, $j=1,...,n$.

Until now, all the statements where $\rho$ appeared were true for any value small enough. For the rest of the proof we are going to fix a particular value, which we call $\rho_1$ to distinguish it from the "variable" $\rho$. It is choosen such that
\[\rho_1^{\eta}\leq C_3[3L(L+\tau)]^{-l}\frac{c_3^2}{4a},\]
where $C_3>0$ is a constant, independent of $\rho$, that will be fixed later (Eq. (\ref{C3})).

We can choose $\ul{a}_0^1$,..., $\ul{a}_0^n$ and $\tb_0^1,...,\tb_0^n$ ending in it, respectively, with the following properties:
\begin{itemize}
\item[a.] The element
\[s_0=(\log \frac{r_{\ul{a}_0^1}}{r_{\ul{a}_0^n}},...,\log \frac{r_{\ul{a}_0^{n-1}}}{r_{\ul{a}_0^n}},v_{\ul{a}_0^1},...,v_{\ul{a}_0^n})\]
is in $J_r$.
\item[b.] $x_{0}\in G(\ul{a}_0^1)\times ... \times G(\ul{a}_0^n)$.
\item[c.] For any $\ul{b}^1,...,\ul{b}^n$, consider
\[\tilde{s}=(\log \frac{r_{\ul{a}_0^1\ul{b}^1}}{r_{\ul{a}_0^n\ul{b}^n}},...,\log \frac{r_{\ul{a}_0^{n-1}\ul{b}^{n-1}}}{r_{\ul{a}_0^n\ul{b}^n}},v_{\ul{a}_0^1\ul{b}^1},...,v_{\ul{a}_0^n\ul{b}^n}).\]
Then
\[\| Dh(c_{\ul{a}_0^1\ul{b}^1},...,c_{\ul{a}_0^n\ul{b}^n})\circ \pi_{\tb_0^1\ul{b}^1,...,\tb_0^n\ul{b}^n,\tilde{s}} - Dh(x_0) \circ \pi_{T_{\ul{b}^1,...,\ul{b}^n}(\tb_0^1,...,\tb_0^n,s)}\| \leq \frac{1}{3}\tau \rho_1.\]
This is achieved by choosing very long words for $\ul{a}_0^1$,..., $\ul{a}_0^n$. If the number of symbols in $\ul{a}_0^j$, for $1\leq j\leq n$, increases, then $|x_0-(c_{\ul{a}_0^1\ul{b}^1},...,c_{\ul{a}_0^n\ul{b}^n})|$ goes to zero, the same happens for the distance between the last coordinate of $T_{\ul{b}^1,...,\ul{b}^n}(\tb_0^1,...,\tb_0^n,s)$ and $\tilde{s}$. This a consequence of item b. and the fact that $r_{\ul{a}_0^j}r^{\tb_0^j}_{\ul{b}^j}/r_{\ul{a}_0^j\ul{b}^j} \to 1$, $\|v_{\ul{a}_0^j\ul{b}^j}-(v_{\ul{a}_0^j}+v^{\tb_0^j}_{\ul{b}^j})\|\to 0$ as the number of symbols in $\ul{a}_0^j$ goes to infinity.

\item[d.] For any $\ul{b}^1,...,\ul{b}^n$ such that $\tilde{s}\in J_r$, there is an affine function $L$ such that
\[\|L\circ h \circ (f_{\ul{a}_0^1\ul{b}^1},...,f_{\ul{a}_0^n\ul{b}^n})-Dh(c_{\ul{a}_0^1\ul{b}^1},...,c_{\ul{a}_0^n\ul{b}^n})\circ \pi_{\tb_0^1\ul{b}^1,...,\tb_0^n\ul{b}^n,\tilde{s}}\|\leq \frac{1}{3}\tau \rho_1.\]
This is a consequence of lemma \ref{ren:atr}.
\end{itemize}

By hypothesis there exists $(\ul{c}^1,...,\ul{c}^n)\in \Sigma_1^{fin}\times ...\times \Sigma_n^{fin}$ such that $T_{\ul{c}^1,...,\ul{c}^n}(\tb_0^1,...,\tb_0^n,s_0)=(\tb_0^1\ul{c}^1,...,\tb_0^n\ul{c}^n,\tilde{s}_0)$ and $\tilde{s}_0\in F^*(\tilde{\ul{a}}_0^1,...,\tilde{\ul{a}}_0^n)$ for some $(\tilde{\ul{a}}_0^1,...,\tilde{\ul{a}}_0^n)$ for which $(\tb_0^1\ul{c}^1,...,\tb_0^n\ul{c}^n)$ ends in it.

We will define inductively a set 
\[N\subset \Sigma_1^{fin}\times ...\times \Sigma_n^{fin}\times \Sigma_1^-\times ....\times \Sigma_n^-\times J_r.\]
Every $p=(\ul{a}^1,...,\ul{a}^n,\tb^1,...,\tb^n,s)\in N$ should verify:
\begin{itemize}
\item[(i)] $s\in F^*(\tilde{\ul{a}}^1,...,\tilde{\ul{a}}^n)$ for some $(\tilde{\ul{a}}^1,...,\tilde{\ul{a}}^n)$ such that $(\tb^1,...,\tb^n)$ ends in $(\tilde{\ul{a}}^1,...,\tilde{\ul{a}}^n)$.
\item[(ii)] $(\ul{a}^1,...,\ul{a}^n,\tb^1,...,\tb^n,s)=(\ul{a}_0^1\ul{b}^1,...,\ul{a}_0^n\ul{b}^n,T_{\ul{b}^1,...,\ul{b}^n}(\tb_0^1,...,\tb_0^n,s_0))$ for some $(\ul{b}^1,...,\ul{b}^n)\in \Sigma_1^{fin}\times ...\times \Sigma_n^{fin}$.
\end{itemize}
For $p=(\ul{a}^1,...,\ul{a}^n,\tb^1,...,\tb^n,s)\in N$ we will define a set $T'(p)\subset R(\rho_1)$ verifying:
\begin{itemize}
\item[(iii)] $\#T'(p)\geq C_3^{-1}\rho_1^{\eta-(d_1+...+d_n)}$.
\item[(iv)] $h\circ (f_{\ul{a}^1},...,f_{\ul{a}^n})(Q_1)\cap h\circ (f_{\ul{a}^1},...,f_{\ul{a}^n})(Q_2)=\emptyset$, for all $Q_1,Q_2\in T'(p)$, $Q_1\neq Q_2$.
\item[(v)] For all $(\ul{b}^1,...,\ul{b}^n)\in T'(p)$ we have \[T_{\ul{b}^1,...,\ul{b}^n}(\tb^1,...,\tb^n,s)=(\tb^1\ul{b}^1,...,\tb^n\ul{b}^n,\tilde{s})\]
and $\tilde{s}\in F^*(\ul{b}^1,...,\ul{b}^n)$.
\end{itemize}
Elements of $N$ are defined inductively, i.e. every element already defined $p\in N$ generates new elements, which we call the children of $p$ and denote by $Ch(p)$. Thus, $N$ has the structure of a rooted tree. The root of the tree is $p_0=(\ul{a}_0^1\ul{c}^1,...,\ul{a}_0^n\ul{c}^n,\tb_0^1\ul{c}^1,...,\tb_0^n\ul{c}^n,\tilde{s}_0)$, the set $T'(p_0)$ is defined as described below.

Given $p=(\ul{a}^1,...,\ul{a}^n,\tb^1,...,\tb^n,s)$ verifying (i), (ii) (as $p_0$ does) define $T'(p)$ in the following way:

By (i) we know that $s\in F^*(\tilde{\ul{a}}^1,...,\tilde{\ul{a}}^n)$, hence the scale recurrence lemma implies that there exists a set $F\subset R(\rho_1)$ with $\#F\geq c_3\rho_1^{-(d_1+...+d_n)}$ and such that (v) holds for $F$.

Since $F^*(\tilde{\ul{a}}^1,...,\tilde{\ul{a}}^n)\subset V_{c_2\rho_1}(F(\tilde{\ul{a}}^1,...,\tilde{\ul{a}}^n))$, then there exist $s'\in F(\tilde{\ul{a}}^1,...,\tilde{\ul{a}}^n)$ with $|s-s'|\leq c_2\rho_1$. By the definition of $F(\tilde{\ul{a}}^1,...,\tilde{\ul{a}}^n)$ we have $N_{\rho_1}(\tilde{\tb}^1,...,\tilde{\tb}^n,s')\leq a\rho_1^{-(d_1+...+d_n)}$ for some $(\tilde{\tb}^1,...,\tilde{\tb}^n)$ that ends in $(\tilde{\ul{a}}^1,...,\tilde{\ul{a}}^n)$.

Using prop. \ref{prop:marstrand} we obtain a set $T\subset F$ such that
\[Dh(x_0)\circ \pi_{\tilde{\tb}^1,...,\tilde{\tb}^n,s'}(Q_1) \cap  Dh(x_0)\circ \pi_{\tilde{\tb}^1,...,\tilde{\tb}^n,s'}(Q_2)= \emptyset,\,\, \forall Q_1,Q_2\in T,\,\, Q_1\neq Q_2,\]
and $\#T\geq (c_3^2/(4a))\rho_1^{-(d_1+...+d_n)}$.

We want to use lemma \ref{perturbacion} for $\phi=Dh(x_0)\circ \pi_{\tilde{\tb}^1,...,\tilde{\tb}^n,s'}$, $\psi=L\circ h\circ (f_{\ul{a}^1},...,f_{\ul{a}^n})$, where $L$ is some affine function, and the set $T$. Note first that $\tb^j\wedge \tilde{\tb}^j\in\Sigma_j(c_0,\rho_1)$, $1\leq j \leq n$, since both $\tb^j$ and $\tilde{\tb}^j$ end in $\tilde{\ul{a}}^j$. Equation (\ref{eq:lippi}) implies then
\[\|Dh(x_0)\circ \pi_{\tilde{\tb}^1,...,\tilde{\tb}^n,s'}-Dh(x_0)\circ \pi_{\tb^1,...,\tb^n,s}\|\leq \frac{1}{3}\tau \rho_1.\]
On the other hand, item (ii) together with items c. and d. implies
\[\|L\circ h\circ (f_{\ul{a}^1},...,f_{\ul{a}^n})-Dh(x_0)\circ \pi_{\tb^1,...,\tb^n,s}\|\leq \frac{2}{3}\tau \rho_1.\]
We conclude
\[\|L\circ h\circ (f_{\ul{a}^1},...,f_{\ul{a}^n})-Dh(x_0)\circ \pi_{\tilde{\tb}^1,...,\tilde{\tb}^n,s'}\|\leq \tau \rho_1,\]
this together with Eq. (\ref{eq:open}) shows that we can use lemma \ref{perturbacion}. Hence, there is a subset $T'(p)\subset T\subset F$ such that
\[ h\circ (f_{\ul{a}^1},...,f_{\ul{a}^n})(Q_1)\cap h\circ (f_{\ul{a}^1},...,f_{\ul{a}^n})(Q_1) = \emptyset,\,\, \forall Q_1,Q_2 \in T'(p),\,\,Q_1\neq Q_2,\]
and
\[\#T'(p)\geq [3L(L+\tau)]^{-l}\cdot\frac{c_3^2}{4a}\rho_1^{-(d_1+...+d_n)}\geq C_3^{-1}\rho_1^{\eta-(d_1+...+d_n)}.\]
In the way we have defined $T'(p)$ it clearly verifies (iii), (iv), (v).

Given $p=(\ul{a}^1,...,\ul{a}^n,\tb^1,...,\tb^n,s)\in N$, the children of $p$ are defined by
\[Ch(p)=\{(\ul{a}^1\ul{b}^1,...,\ul{a}^n\ul{b}^n,T_{\ul{b}^1,...,\ul{b}^n}(\tb^1,...,\tb^n,s)): (\ul{b}^1,...,\ul{b}^n)\in T'(p)\}.\]
The children of $p$ clearly satisfy (i), (ii).

Now that we have defined $N$ we can finish the proof. For each non-negative integer $k$, consider the set $I(k)$ of elements $p\in N$ generated in the $k$-step of the inductive process.\footnote{They are children of elements generated in the $(k-1)$-step, and the only element in the 0-step is $p_0$.} For each $p=(\ul{a}^1,...,\ul{a}^n,\tb^1,...,\tb^n,s)\in N$ define the set
\[G(p)=h(G(\ul{a}^1)\times...\times G(\ul{a}^n)) \subset \R^l.\]
We clearly have
\[\bigcap_{k\geq 0}\bigcup_{p\in I(k)}G(p)\subset h(K_1\times...\times K_n).\]
The desired result, $HD(h(K_1\times... \times K_n))\geq d_1+...+d_n-\eta$, follows from corollary \ref{distribucion} if we can prove that
\[\sum_{q\in Ch(p)} \left(\frac{diam(G(q))}{diam (G(p))}\right)^{d_1+...+d_n-\eta}\geq 1,\]
and each set $G(p)$ contains a ball with radius proportional to its diameter. All other requirements in the corollary are obviously verified.

Given $p=(\ul{a}^1,...,\ul{a}^n,\tb^1,...,\tb^n,s)\in N$, property (i), (ii) and the observation in c. imply that $s\in J_r$ and
\[s(\ul{a}^1,...,\ul{a}^n)=(\log \frac{r_{\ul{a}^1}}{r_{\ul{a}^n}},...,\log \frac{r_{\ul{a}^{n-1}}}{r_{\ul{a}^n}},v_{\ul{a}^1},...,v_{\ul{a}^n})\]
is very close to $s$. Thus, we can assume $s(\ul{a}^1,...,\ul{a}^n)\in J_{2r}$ and then we can think of $G(\ul{a}^1)\times...\times G(\ul{a}^n)$ as being a square, in fact we have $diam(G(\ul{a}^j))\leq e^{4r} diam(G(\ul{a}^m))$, for any $j, m$. This together with the fact that $Dh(x_0)$ verifies the transversality hypotheses allow us to conclude that $h(G(\ul{a}^1)\times...\times G(\ul{a}^n))$ contains a ball with radius proportional to its diameter. Moreover,
\begin{align*}
(C_3')^{-1} diam(G(\ul{a}^1)\times...\times G(\ul{a}^n)) &\leq diam(h(G(\ul{a}^1)\times...\times G(\ul{a}^n)))\\
&\leq C_3' diam(G(\ul{a}^1)\times...\times G(\ul{a}^n))
\end{align*}
for a constant $C_3'>0$, depending only on $r$, $h$ and the Cantor sets. On the other hand, we can choose $C_4'>0$, independent of $\rho$, such that
\[diam(G(\ul{a}\ul{b}))\geq C'_4 \rho \cdot diam (G(\ul{a})),\]
for any $\ul{a}\in \Sigma_j^{fin}$ and $\ul{b}\in \Sigma_j(c_0,\rho)$, $0 \leq j\leq n$. Therefore, we can choose $C_3>0$, that does not depend on $\rho_1$, such that
\begin{equation}\label{C3}
\left(\frac{diam(G(q))}{diam (G(p))}\right)^{d_1+...+d_n-\eta}\geq C_3 \rho_1^{d_1+...+d_n-\eta},
\end{equation}
for any $q\in Ch(p)$. Now that $C_3$ has been chosen we get
\[\sum_{q\in Ch(p)} \left(\frac{diam(G(q))}{diam (G(p))}\right)^{d_1+...+d_n-\eta}\geq \sum_{q\in Ch(p)} C_3 \rho_1^{d_1+...+d_n-\eta}= C_3 \rho_1^{d_1+...+d_n-\eta}\cdot \#T'(p)\geq 1.\]
\end{proof}

\begin{dmf}\label{thm:dimf}
Let $K_1$,..., $K_n$ be $C^m$, $m\geq 2$, conformal regular Cantor sets generated by expanding maps $g_1$,...,$g_n$, respectively. Suppose all of them are not essentially affine. Assume that there exist periodic points $p_j \in K_j$, with period $n_j$, for $1\leq j\leq n$, such that if we write
\[Dg_j^{n_j}(p_j)=\frac{1}{r_j} R_{-v_j},\]
where $R_v$ is the rotation matrix by an angle $v\in \T$, then
\begin{align*}
(\log r_1,0,..,0&;v_1,0,..,0),\\
&\vdots \\
(0,...,\log r_{n-1}&;0,...,v_{n-1},0),\\
(-\log r_n,...,-\log r_n &;0,...,0,v_n),
\end{align*}
generate a dense subgroup of $J$. Let $h$ be any $C^{1}$ function defined on a neighborhood of $K_1\times...\times K_n$ into $\R^l$ such that there exists a point $x_0\in K_1\times ...\times K_n$ where $Dh(x_0)$ verifies the transversality hypotheses, then
\[HD(h(K_1\times...\times K_n))=\min \{l,HD(K_1)+...+HD(K_n)\}.\]
\end{dmf}

\begin{proof}
We first treat the case $HD(K_1)+...+HD(K_n)<l$. Notice that $K_1,...,K_n$ verify the hypotheses of the scale recurrence lemma, the existence of the periodic points $p_j$ imply that all of the Cantor sets are not essentially real. The desired result follows from the preceding theorem if we show that for any $(\tb^1,...,\tb^n,s)\in \Sigma_1^-\times...\times \Sigma_n^-\times J_r$ there exists $(\ul{c}^1,...,\ul{c}^n)\in \Sigma_1^{fin}\times ...\times \Sigma_n^{fin}$ such that $T_{\ul{c}^1,...,\ul{c}^n}(\tb^1,...,\tb^n,s)=(\tb^1\ul{c}^1,...,\tb^n\ul{c}^n,\tilde{s})$ and $\tilde{s}\in F^*(\tilde{\ul{a}}^1,...,\tilde{\ul{a}}^n)$ for some $(\tilde{\ul{a}}^1,...,\tilde{\ul{a}}^n)$ for which $(\tb^1\ul{c}^1,...,\tb^n\ul{c}^n)$ ends in it.

Let $\ul{a}_j \in \Sigma_j^{fin}$ be the word of length $n_j$ such that the periodic point $p_j$ corresponds to the sequence $\ul{a}_j \ul{a}_j \ul{a}_j...$. For a finite sequence $\ul{a}\in \Sigma_j^{fin}$ and $k\in\mathbb{Z}^+$ we are going to use the notation
\[\ul{a}^k=\underbrace{\ul{a} \ul{a}...\ul{a}}_{k-times}.\]
Choosing $c_0$ big enough and assuming $\rho$ is small, we can find $k_j\in \mathbb{Z}^+$ such that
\[\tilde{\ul{a}}_j:=\ul{a}_j^{k_j} \in \Sigma_j(c_0,\rho).\]
Define 
\[\tb_j=...\tilde{\ul{a}}_j...\tilde{\ul{a}}_j\in \Sigma_j^-.\]
By property (ii) and (iii) of the scale recurrence lemma we know that there are $s_0, s_1 \in J$, $\ul{b}_j\in \Sigma_j(c_0,\rho)$ and $\ul{c}_j \in \Sigma^{fin}_j$, $1\leq j \leq n$, such that
\[T_{\ul{c}_1\ul{b}_1,...,\ul{c}_n\ul{b}_n}(\tb_1,...,\tb_n,s_0)=(\tb_1\ul{c}_1\ul{b}_1,...,\tb_n\ul{c}_n\ul{b}_n,s_1),\]
and the $\rho$-neighborhood of $s_1$ is contained in $F^*(\ul{b}_1,...,\ul{b}_n)$.

Thanks to the continuity of the map $\tb\to k^{\tb}$ we can choose positive integers $l_1,...,l_n$, depending on $\rho$, such that for any $m_j>l_j$, any $\tb^j \in \Sigma_j^-$, $1\leq j\leq n$, and $x$ in the $\rho/2$-neighborhood of $s_0$ we have\footnote{We assume $\tb^j$, $1\leq j \leq n$, ends with the letters in which $\ul{a}_j$, $1\leq j \leq n$, starts, otherwise we consider $\ul{d}^j\tilde{\ul{a}}_j^m$ instead of $\tilde{\ul{a}}_j^m$, for some $\ul{d}^j$, and the proof follows in the same way.}
\[T_{\ul{c}_1\ul{b}_1,...,\ul{c}_n\ul{b}_n}(\tb^1\tilde{\ul{a}}_1^{m_1},...,\tb^n\tilde{\ul{a}}_n^{m_n},x)=(\tb^1\tilde{\ul{a}}_1^{m_1}\ul{c}_1\ul{b}_1,...,\tb^n\tilde{\ul{a}}_n^{m_n}\ul{c}_n\ul{b}_n,\tilde{x})\]
and $\tilde{x} \in F^*(\ul{b}_1,...,\ul{b}_n)$.

Now notice that
\[r^{\tb^j}_{\tilde{\ul{a}}^{l_j+m_j}}=r^{\tb^j}_{\tilde{\ul{a}}^{l_j}}\cdot r^{\tb^j\tilde{\ul{a}}^{l_j}}_{\tilde{\ul{a}}^{m_j}},\]
and if $l_j$ is big enough we have 
\[|\log r^{\tb^j\tilde{\ul{a}}_j^{l_j}}_{\tilde{\ul{a}}_j^{m_j}}-\log r^{\tb_j}_{\tilde{\ul{a}}_j^{m_j}}|<\frac{\rho}{8(2n-1)},\]
for any $m_j\in \mathbb{Z}^+$. Similar formulas hold for $v^{\tb^j\tilde{\ul{a}}_j^{l_j}}_{\tilde{\ul{a}}_j^{m_j}}$.\\
For $m_j\in \mathbb{Z}^+$, $1\leq j\leq n$, consider
\[T_{\tilde{\ul{a}}_1^{l_1+m_1},...,\tilde{\ul{a}}_j^{l_j+m_j}}(\tb^1,...,\tb^n,s)=(\tilde{\tb}^1,...,\tilde{\tb}^n,\tilde{s}).\]
We have
\begin{align*}
\tilde{s}&=s+(\log r_{\tilde{\ul{a}}_1^{l_1+m_1}}^{\tb^1}-\log r_{\tilde{\ul{a}}_n^{l_n+m_n}}^{\tb^n},..., \log r_{\tilde{\ul{a}}_{n-1}^{l_{n-1}+m_{n-1}}}^{\tb^{n-1}}-\log r_{\tilde{\ul{a}}_n^{l_n+m_n}}^{\tb^n};v_{\tilde{\ul{a}}_1^{l_1+m_1}}^{\tb^1},...,v_{\tilde{\ul{a}}_n^{l_n+m_n}}^{\tb^n})\\
&=s+(\log r_{\tilde{\ul{a}}_1^{l_1}}^{\tb^1}-\log r_{\tilde{\ul{a}}_n^{l_n}}^{\tb^n},..., \log r_{\tilde{\ul{a}}_{n-1}^{l_{n-1}}}^{\tb^{n-1}}-\log r_{\tilde{\ul{a}}_n^{l_n}}^{\tb^n};v_{\tilde{\ul{a}}_1^{l_1}}^{\tb^1},...,v_{\tilde{\ul{a}}_n^{l_n}}^{\tb^n})\\
&\hspace{0.7cm}+(\log r_{\tilde{\ul{a}}_1^{m_1}}^{\tb^1\tilde{\ul{a}}_1^{l_1}}-\log r_{\tilde{\ul{a}}_n^{m_n}}^{\tb^n\tilde{\ul{a}}_n^{l_n}},..., \log r_{\tilde{\ul{a}}_{n-1}^{m_{n-1}}}^{\tb^{n-1}\tilde{\ul{a}}_{n-1}^{l_{n-1}}}-\log r_{\tilde{\ul{a}}_n^{m_n}}^{\tb^n\tilde{\ul{a}}_n^{l_n}};v_{\tilde{\ul{a}}_1^{m_1}}^{\tb^1\tilde{\ul{a}}_1^{l_1}},...,v_{\tilde{\ul{a}}_n^{m_n}}^{\tb^n\tilde{\ul{a}}_n^{l_n}}).
\end{align*}
Now notice that
\[(\log r_{\tilde{\ul{a}}_1^{m_1}}^{\tb^1\tilde{\ul{a}}_1^{l_1}}-\log r_{\tilde{\ul{a}}_n^{m_n}}^{\tb^n\tilde{\ul{a}}_n^{l_n}},..., \log r_{\tilde{\ul{a}}_{n-1}^{m_{n-1}}}^{\tb^{n-1}\tilde{\ul{a}}_{n-1}^{l_{n-1}}}-\log r_{\tilde{\ul{a}}_n^{m_n}}^{\tb^n\tilde{\ul{a}}_n^{l_n}};v_{\tilde{\ul{a}}_1^{m_1}}^{\tb^1\tilde{\ul{a}}_1^{l_1}},...,v_{\tilde{\ul{a}}_n^{m_n}}^{\tb^n\tilde{\ul{a}}_n^{l_n}})\]
is in the $\rho/4$ neighborhood of
\[(\log r_{\tilde{\ul{a}}_1^{m_1}}^{\tb_1}-\log r_{\tilde{\ul{a}}_n^{m_n}}^{\tb_n},..., \log r_{\tilde{\ul{a}}_{n-1}^{m_{n-1}}}^{\tb_{n-1}}-\log r_{\tilde{\ul{a}}_n^{m_n}}^{\tb_n};v_{\tilde{\ul{a}}_1^{m_1}}^{\tb_1},...,v_{\tilde{\ul{a}}_n^{m_n}}^{\tb_n}).\]
On the other hand
\begin{align*}
    (\log r_{\tilde{\ul{a}}_1^{m_1}}^{\tb_1}-\log r_{\tilde{\ul{a}}_n^{m_n}}^{\tb_n},..., \log r_{\tilde{\ul{a}}_{n-1}^{m_{n-1}}}^{\tb_{n-1}}-&\log r_{\tilde{\ul{a}}_n^{m_n}}^{\tb_n};v_{\tilde{\ul{a}}_1^{m_1}}^{\tb_1},...,v_{\tilde{\ul{a}}_n^{m_n}}^{\tb_n})=\\
    m_1 (\log r_{\tilde{\ul{a}}_1}^{\tb_1},0,..,0;v_{\tilde{\ul{a}}_1}^{\tb_1},0,..,0)+...&+ m_{n-1}(0,...,\log r_{\tilde{\ul{a}}_{n-1}}^{\tb_{n-1}};0,...,v_{\tilde{\ul{a}}_{n-1}}^{\tb_{n-1}},0)\\&+m_n(-\log r_{\tilde{\ul{a}}_{n}}^{\tb_{n}},...,-\log r_{\tilde{\ul{a}}_{n}}^{\tb_{n}};0,...,0,v_{\tilde{\ul{a}}_{n}}^{\tb_{n}}).
\end{align*}
Moreover
\[r_{\tilde{\ul{a}}_j}^{\tb_j} R_{v_{\tilde{\ul{a}}_j}^{\tb_j}}=[Dg^{k_jn_j}(p_j)]^{-1}=r_j^{k_j}R_{k_jv_j}.\]
Therefore, the density hypothesis in the theorem, together with the next lemma, implies that there exists $m_1,...,m_n \in \Z^+$ such that
\[(\log r_{\tilde{\ul{a}}_1^{m_1}}^{\tb_1}-\log r_{\tilde{\ul{a}}_n^{m_n}}^{\tb_n},..., \log r_{\tilde{\ul{a}}_{n-1}^{m_{n-1}}}^{\tb_{n-1}}-\log r_{\tilde{\ul{a}}_n^{m_n}}^{\tb_n};v_{\tilde{\ul{a}}_1^{m_1}}^{\tb_1},...,v_{\tilde{\ul{a}}_n^{m_n}}^{\tb_n})\]
is in the $\rho/4$-neighborhood of 
\[s_0-s-(\log r_{\tilde{\ul{a}}_1^{l_1}}^{\tb^1}-\log r_{\tilde{\ul{a}}_n^{l_n}}^{\tb^n},..., \log r_{\tilde{\ul{a}}_{n-1}^{l_{n-1}}}^{\tb^{n-1}}-\log r_{\tilde{\ul{a}}_n^{l_n}}^{\tb^n};v_{\tilde{\ul{a}}_1^{l_1}}^{\tb^1},...,v_{\tilde{\ul{a}}_n^{l_n}}^{\tb^n}).\]
Hence $\tilde{s}$ is in the $\rho/2$-neighborhood of $s_0$ and from this we get
\begin{align*}
T_{\tilde{\ul{a}}_1^{l_1+m_1}\ul{c}_1\ul{b}_1,...,\tilde{\ul{a}}_n^{l_n+m_n}\ul{c}_n\ul{b}_n}(\tb^1,...,\tb^n,s)&=
T_{\ul{c}_1\ul{b}_1,...,\ul{c}_n\ul{b}_n}(\tb^1\tilde{\ul{a}}_1^{l_1+m_1},...,\tb^n\tilde{\ul{a}}_n^{l_n+m_n},\tilde{s})\\
&=(\tb^1\tilde{\ul{a}}_1^{l_1+m_1}\ul{c}_1\ul{b}_1,...,\tb^n\tilde{\ul{a}}_n^{l_n+m_n}\ul{c}_n\ul{b}_n,\tilde{s}'),
\end{align*}
and $\tilde{s}'\in F^*(\ul{b}_1,...,\ul{b}_n)$, as we wanted.

If $HD(K_1)+...+HD(K_n)\geq l$, fix $\epsilon>0$ and find conformal regular Cantor sets $\tilde{K}_j\subset K_j$, $1\leq j\leq n$ such that $l-\epsilon<HD(\tilde{K}_1)+...+HD(\tilde{K}_n)<l$, $p_j\in \tilde{K}_j$, and the expanding map of $\tilde{K}_j$ is given by a power of $g_j$, $1\leq j \leq n$ (see \cite{M}, lemma in page 16). We get
\begin{align*}
l\geq HD(h(K_1\times ...\times K_n))&\geq HD(h(\tilde{K}_1 \times...\times \tilde{K}_n))\\
&=HD(\tilde{K}_1)+...+HD(\tilde{K}_n)>l-\epsilon.
\end{align*}
Since $\epsilon$ can be arbitrarily small we obtain $HD(h(K\times ...\times K_n))=l$ as we wanted.
\end{proof}

The following lemma was used in the previous theorem. It also implies that the hypothesis needed for the dimension formula is generic. Its proof relies in the well known Kronecker's theorem.

\begin{lem}\label{lem:genhyp}
Let $\lambda_j<0$, $v_j\in \T$, $1\leq j\leq n$, and consider the set $E(\lambda_1,...,\lambda_n,v_1,...,v_n)\subset \R^{n-1}\times \T^{n}$ given by the vectors
\begin{align*}
(\lambda_1,0,..,0&;v_1,0,..,0),\\
&\vdots \\
(0,...,\lambda_{n-1}&;0,...,v_{n-1},0),\\
(-\lambda_n,...,-\lambda_n &;0,...,0,v_n).
\end{align*}
We have the following properties:
\begin{itemize}
    \item[a.] If $E(\lambda_1,...,\lambda_n,v_1,...,v_n)$ generates a dense subgroup of $\R^{n-1}\times \T^{n}$ then  $E(k_1 \lambda_1,...,k_n \lambda_n,k_1 v_1,...,k_n v_n)$ also generates a dense subgroup, for all $k_1,...,k_n \in \mathbb{Z}\setminus \{0\}$.
    \item[b.] If $E(\lambda_1,...,\lambda_n,v_1,...,v_n)$ generates a dense subgroup of $\R^{n-1}\times \T^{n}$ then it also generates a dense semigroup, i.e. the set of linear combinations of vectors in $E(\lambda_1,...,\lambda_n,v_1,...,v_n)$ with coeficients in $\mathbb{N}$ is dense in $\R^{n-1}\times \T^{n}$.
    \item[c.] The set of values $(\lambda_1,...,\lambda_n,v_1,...,v_n)\in \R_{<0}^{n}\times \T^n$ for which the set $E(\lambda_1,...,\lambda_n,v_1,...,v_n)$ generates a dense subgroup, is a countable intersection of open and dense subsets.
\end{itemize}
\end{lem}
\begin{proof}
The lemma is proved using Kronecker's theorem. It states that a vector $(w_1,...,w_k)\in \T^k$ generates a dense subgroup if and only if there is not $(a_1,...,a_k)\in \Z^k\setminus \{0\}$ such that $a_1w_1+...+a_kw_k=0$.

Let $p:\R\to \T$ be the canonical projection and choose $\tilde{v}_j\in \R$, $1\leq j\leq n$, such that $p(\tilde{v}_j)=v_j$. Define vectors $\alpha_j$, $1\leq j\leq n-1$, and $\beta$ by 
\begin{align*}
\alpha_j&=(0,...,\lambda_j,..,0;0,...,\tilde{v}_j,..,0)\\
\beta&=(-\lambda_n,...,-\lambda_n ;0,...,0,\tilde{v}_n).
\end{align*}
Note that $E(\lambda_1,...,\lambda_n,v_1,...,v_n)$ generates a dense subgroup in $\R^{n-1}\times \T^{n}$ if and only if the set 
\[\{\beta,\alpha_1,...,\alpha_{n-1},e_n,...,e_{2n-1}\}\]
generates a dense subgroup in $\R^{2n-1}$. Here $e_1,...,e_{2n-1}$ is the canonical base of $\R^{2n-1}$.

Moreover, this last property is invariant by linear tansformations in $\R^{2n-1}$. Let $A:\R^{2n-1} \to \R^{2n-1}$ be the linear map such that $A(\alpha_j)=e_j$ for $1\leq j\leq n-1$ and $A(e_j)=e_j$ for $n\leq j\leq 2n-1$. Then, $E(\lambda_1,...,\lambda_n,v_1,...,v_n)$ generates a dense subgroup in $\R^{n-1}\times \T^{n}$ if and only if the set
\[\{A(\beta),e_1,...,e_{n-1},e_n,...,e_{2n-1}\}\]
generates a dense subgroup in $\R^{2n-1}$. It is clear that this happens if and only if the projection of $A(\beta)$ to $\T^{2n-1}$ generates a dense subgroup in $\T^{2n-1}$. Thus, using Kronecker's theorem we see that $E(\lambda_1,...,\lambda_n,v_1,...,v_n)$ generates a dense subgroup in $\R^{n-1}\times \T^{n}$ if and only if there is not $a=(a_1,...,a_{2n-1})\in \Z^{2n-1}\setminus \{0\}$ such that
\[ \langle a, A(\beta)\rangle \in \Z.\] 
Moreover, it is not difficult to see that
\[A(\beta)=(\frac{-\lambda_n}{\lambda_1},...,\frac{-\lambda_n}{\lambda_{n-1}}; \frac{ \tilde{v}_1 \lambda_n}{\lambda_1},..., \frac{\tilde{v}_{n-1}\lambda_n}{\lambda_{n-1}},\tilde{v}_n).\]
From all this, one easily gets that a. is true. For c. notice that the set of values $(\lambda_1,...,\lambda_n,v_1,...,v_n)$ such that $E(\lambda_1,...,\lambda_n,v_1,...,v_n)$ generates a dense subgroup corresponds to the intersection, varying $a\in \Z^{2n-1}\setminus \{0\}$, of the sets
\[\{(\lambda_1,...,\lambda_n,v_1,...,v_n): \langle a, A(\beta)\rangle \notin \Z\},\]
and each one of these sets is open and dense. 

Finally, to justify b. notice that $E(\lambda_1,...,\lambda_n,v_1,...,v_n)$ generates a dense semigroup if and only if the set
\[\{b_0\beta+b_1 \alpha_1+..+b_{n-1}\alpha_{n-1}+b_n e_n+..+b_{2n-1} e_{2n-1}: b_0,..,b_{n-1}\in \N,\,\,b_n,..,b_{2n-1}\in \Z\}\]
is dense. Applying the linear transformation $A$, this is equivalent to
\[\{b_0A(\beta)+b_1 e_1+..+b_{n-1}e_{n-1}+b_n e_n+..+b_{2n-1} e_{2n-1}: b_0,..,b_{n-1}\in \N,\,\,b_n,..,b_{2n-1}\in \Z\}\]
being dense. Now, using the expression for $A(\beta)$ this set becomes
\begin{align*}
\{ ( b_1-b_0\frac{\lambda_n}{\lambda_1},..,b_{n-1}-b_0\frac{\lambda_n}{\lambda_{n-1}}; b_n+b_0 \frac{\tilde{v}_1 \lambda_n}{\lambda_1},.., & b_{2n-2}+b_0\frac{\tilde{v}_{n-1}\lambda_n}{\lambda_{n-1}},b_{2n-1}+b_0\tilde{v}_n) \\
&: b_0,..,b_{n-1}\in \N,\,\,b_n,..,b_{2n-1}\in \Z\}.
\end{align*}
Notice that if $E(\lambda_1,...,\lambda_n,v_1,...,v_n)$ generates a dense subgroup then the set 
\[\{b_0 \cdot (p\left(-\lambda_n /\lambda_1\right),..,p\left(-\lambda_n/\lambda_{n-1}\right); p\left( \tilde{v}_1 \lambda_n/\lambda_1\right),.., p\left(\tilde{v}_{n-1}\lambda_n/\lambda_{n-1}\right),v_n): b_0 \in \N\}\]
is dense in $\T^{2n-1}$. Since $\frac{-\lambda_n}{\lambda_j}<0$ we conclude that the density of the group generated by $E(\lambda_1,...,\lambda_n,v_1,...,v_n)$ implies that
\begin{align*}
\{ ( b_1-b_0\frac{\lambda_n}{\lambda_1},..,b_{n-1}-b_0\frac{\lambda_n}{\lambda_{n-1}}; b_n+\frac{b_0 \tilde{v}_1 \lambda_n}{\lambda_1},.., & b_{2n-2}+\frac{b_0\tilde{v}_{n-1}\lambda_n}{\lambda_{n-1}},b_{2n-1}+b_0\tilde{v}_n) \\
&: b_0,..,b_{n-1}\in \N,\,\,b_n,..,b_{2n-1}\in \Z\}
\end{align*}
is dense.
\end{proof}

\section{Proof of the Scale Recurrence Lemma}

In this section we will present the proof of the scale recurrence lemma, it follows the ideas in \cite{MY} with some modifications. The main new features are the use of the not essentially real hypotheses and the introduction of new objects in order to close a gap in the proof given in \cite{MY}. We start proving some results in a more general setting.

\subsection{General Setting}

We proceed as in 6.1 \cite{MY} using Fourier analysis in the group $J$ instead of $\mathbb{R}$. Let $A$ be a set of indices, $\Lambda$ a finite set and maps $\alpha:\Lambda\to A$, $\omega:\Lambda\to A$. Define $\Lambda_i=\alpha^{-1}(i)$, $\Lambda^j=\omega^{-1}(j)$, $\Lambda_i^j=\Lambda_i\cap \Lambda^j$, $N_i=\#\Lambda_i$, $N_i^j=\#\Lambda_i^j$, $p_i^j=N_i^j/N_i$. The numbers $(p_i^j)$ define a stochastic matrix, it has a probability vector $(p^i)$ verifying
\[\sum_{i \in A}p^i p_i^j=p^j,\,\,\,\sum_{i\in A}p^i=1.\]
Set
\[ p_{\lambda}^{\lambda'}=\begin{cases}
		0  & \mbox{if } \omega(\lambda) \neq \alpha(\lambda') \\
		1/N_{\omega(\lambda)} & \mbox{if } \omega(\lambda) = \alpha(\lambda')
	\end{cases}
\]
and
\[p^{\lambda}=\frac{p^{\alpha(\lambda)}}{N_{\alpha(\lambda)}}.\]
It is easily proved that $(p_{\lambda}^{\lambda'})$ is a stochastic matrix with probability vector $(p^{\lambda})$. Let $J^*=\mathbb{R}^{n-1}\times \mathbb{Z}^n$ denote the Pontryagin dual of $J$. Elements $\xi=(\mu_1,...,\mu_{n-1},m_1,...,m_n)\in J^*$ are homomorphisms from $J$ to $\mathbb{S}^1$, given by 
\[\xi(t_1,...,t_{n-1},v_1,...,v_n)=e^{\left(\sum_{j=1}^{n-1}t_j\mu_j+ \sum_{j=1}^{n}m_jv_j\right)i}.\]
Now suppose that for each $(\lambda,\lambda')\in \Lambda^2$, there is an element $a_{\lambda}^{\lambda'} \in J$. Using this we define, for each $\xi \in J^*$, a linear operator $T_{\xi}:\C^{\Lambda}\to \C^{\Lambda}$ given by $T_{\xi}((z_{\lambda})_{\lambda \in \Lambda})=(w_{\lambda})_{\lambda \in \Lambda}$ where
\[w_\lambda=\sum_{\lambda'\in \Lambda} p_{\lambda}^{\lambda'}\xi(a_{\lambda}^{\lambda'})z_{\lambda'}.\]
We endow the space $\C^{\Ld}$ with the norm
\[\|(z_{\ld})_{\ld\in\Ld}\|^2=\sum_{\ld\in\Ld}p^{\ld}|z_{\ld}|^2.\]
In a similar way to \cite{MY}, a short computation shows that $\left\| T_{\xi}\right\|\leq 1$, for all $\xi\in J^*$.

Assume that we have a family $\{E(\lambda)\}_{\lambda\in \Lambda}$ of bounded measurable subsets of $J$, consider the function
\[n_{\lambda}(x)=\frac{1}{N_{\omega(\lambda)}}\cdot \#\{\lambda'\in \Lambda_{\omega(\lambda)}: B_{\rho}(x+a_{\lambda}^{\lambda'})\subset E(\lambda')\}.\]
Let $0<\tau<1$, denote by $E^*(\lambda)$ the set:
\[E^*(\lambda)=\{x\in J: n_{\lambda}(x)>\tau\}.\]

\begin{prop}\label{fourier}
Suppose there exist $\Delta_0>0$ and $k_{0}\in (0,1)$ such that $\left\| T_{\xi}\right\|<k_0$ for all $\xi=(\mu_1,...,\mu_{n-1},m_1,...,m_n)$, with $|\xi|=\max_j \{|\mu_j|,|m_j|\}\in [1,\Delta_0 \rho^{-1}]$. Then there exist $k_1\in (0,1)$, $\epsilon>0$, and $\tau \in (0,1)$ depending only in $\Delta_0$, $k_0$ (and not in $\rho$) such that if $\nu(E(\lambda))\leq \epsilon$, for all $\lambda \in \Lambda$, then
\[\sum_{\lambda \in \Lambda}p^{\lambda}\nu(E^*(\lambda))\leq k_1 \sum_{\lambda \in \Lambda}p^{\lambda}\nu(E(\lambda)).\]
\end{prop}
\begin{proof}
Consider the functions
\begin{align*} X_{\lambda}&=1_{E(\lambda)},\,\,Y_{\lambda}(x)=\sum_{\lambda'\in \Lambda}p_{\lambda}^{\lambda'}X_{\lambda'}(x+a_{\lambda}^{\lambda'}),\\
Z_{\lambda}(x)&=\frac{1}{\nu(B_{\rho}(0))}\cdot \int_{B_{\rho}(0)}Y_{\lambda}(x-t)d\nu(t)=\frac{1}{\nu(B_{\rho}(0))} 1_{B_{\rho}(0)}\ast Y_{\lambda}(x).
\end{align*}
Note that $Z_{\lambda}(x)\geq n_{\lambda}(x)$, then
\[\left\|Z_{\lambda}\right\|_{L^2}^2\geq \tau^2 \nu(E^*(\lambda))\]
which implies
\begin{equation}\label{eq:ztau}
\sum_{\lambda\in\Lambda}p^{\lambda}\nu(E^*(\lambda))\leq \tau^{-2}\sum_{\lambda\in\Lambda}p^{\lambda}\left\|Z_{\lambda}\right\|_{L^2}^2.
\end{equation}
The Fourier transforms of $X_{\lambda}$, $Y_{\lambda}$, $Z_{\lambda}$ are
\begin{align*}
\hat{X}_{\lambda}(\xi)&=\int_{J}X_{\lambda}(x)\bar{\xi}(x)d\nu(x),\\
\hat{Y}_{\lambda}(\xi)&= \sum_{\lambda'\in \Lambda}p_{\lambda}^{\lambda'}\xi(a_{\lambda}^{\lambda'})\hat{X}_{\lambda'}(\xi),\\
\hat{Z}_{\lambda}(\xi)&=\frac{1}{\nu(B_{\rho}(0))} \hat{1}_{B_{\rho}(0)}\cdot \hat{Y}_{\lambda}(\xi)= \prod_{j=1}^{n-1}\frac{sin(\mu_j\rho)}{\mu_j\rho} \prod_{j=1}^{n}\frac{sin(m_j\rho)}{m_j\rho}\cdot \hat{Y}_{\lambda}(\xi),
\end{align*}
where $\xi=(\mu_1,...,\mu_{n-1},m_1,...,m_n)$. Hence $|\hat{Z}_{\lambda}(\xi)|\leq |\hat{Y}_{\lambda}(\xi)|$, and there exist $\tilde{k}_1\in (0,1)$, depending only in $\Delta_0$, such that $|\hat{Z}_{\lambda}(\xi)|\leq \tilde{k}_1|\hat{Y}_{\lambda}(\xi)|$ if $|\xi|>\Delta_0 \rho^{-1}$. We estimate  $\sum_{\lambda \in \Lambda}p^{\lambda}|\hat{Z}_{\lambda}(\xi)|^2$ in various ways depending on $|\xi|$.

If $|\xi|<1$ then:
\begin{align*}
|\hat{Z}_{\lambda}(\xi)|\leq |\hat{Y}_{\lambda}(\xi)| &\leq \int_{J}|Y_{\lambda}(x)|d\nu(x)\\
&\leq \sum_{\lambda'\in \Lambda}p_{\lambda}^{\lambda'}\nu(E(\lambda')),
\end{align*}
therefore
\begin{align*}
\sum_{\lambda \in \Lambda}p^{\lambda}|\hat{Z}_{\lambda}(\xi)|^2 &\leq \sum_{\lambda' \in \Lambda} \left( \sum_{\lambda \in \lambda}p^{\lambda}p_{\lambda}^{\lambda'}\right)\nu(E(\lambda'))^2\\
&= \sum_{\lambda\in \Lambda}p^{\lambda} \nu(E(\lambda))^2.
\end{align*}
If $1\leq |\xi| \leq \Delta_{0}\rho^{-1}$:
\begin{align*}
\sum_{\lambda \in \Lambda}p^{\lambda}|\hat{Z}_{\lambda}(\xi)|^2&\leq \sum_{\lambda \in \Lambda}p^{\lambda}|\hat{Y}_{\lambda}(\xi)|^2\\
&\leq k_0^2 \sum_{\lambda \in \Lambda}p^{\lambda}|\hat{X}_{\lambda}(\xi)|^2, 
\end{align*}
note that we used the fact that $(\hat{Y}_{\lambda}(\xi))_{\lambda \in \Lambda}=T_{\xi}((\hat{X}_{\lambda}(\xi))_{\lambda \in \Lambda})$.

If $|\xi|>\Delta_{0}\rho^{-1}$:
\begin{align*}
\sum_{\lambda \in \Lambda}p^{\lambda}|\hat{Z}_{\lambda}(\xi)|^2&\leq \tilde{k}_1^2 \sum_{\lambda \in \Lambda}p^{\lambda}|\hat{Y}_{\lambda}(\xi)|^2\\
&\leq \tilde{k}_1^2 \sum_{\lambda \in \Lambda}p^{\lambda}|\hat{X}_{\lambda}(\xi)|^2. 
\end{align*}
Combining all three inequalities we get
\begin{align*}
\int_{J^*}\sum_{\lambda \in \Lambda}p^{\lambda}&|\hat{Z}_{\lambda}(\xi)|^2d\hat{\nu}(\xi) \\
&\leq \int_{|\xi|<1} \sum_{\lambda\in \Lambda}p^{\lambda} \nu(E(\lambda))^2 d\hat{\nu}(\xi)+ k_{0}^2\int_{1\leq |\xi|\leq \Delta_0\rho^{-1}} \sum_{\lambda \in \Lambda}p^{\lambda}|\hat{X}_{\lambda}(\xi)|^2 d\hat{\nu}(\xi)\\
&+ \tilde{k}_1^2\int_{|\xi|>\Delta_0\rho^{-1}} \sum_{\lambda \in \Lambda}p^{\lambda}|\hat{X}_{\lambda}(\xi)|^2 d\hat{\nu}(\xi).
\end{align*}
Hence
\[ \sum_{\lambda \in \Lambda}p^{\lambda}\left\|\hat{Z}_{\lambda}\right\|_{L^2}^2 \leq \hat{\nu}(\{|\xi|<1\})\cdot \epsilon \sum_{\lambda\in \Lambda}p^{\lambda} \nu(E(\lambda))+\max\{k_0^2,\tilde{k}_1^2\} \sum_{\lambda \in \Lambda}p^{\lambda}\left\|\hat{X}_{\lambda}\right\|_{L^2}^2.\]
Using Plancherel theorem and the fact that $\nu(E(\lambda))=\left\|X_{\lambda}\right\|_{L^2}^2$ we get
\[ \sum_{\lambda \in \Lambda}p^{\lambda}\left\|Z_{\lambda}\right\|_{L^2}^2 \leq [\hat{\nu}(\{|\xi|<1\})\cdot \epsilon +\max\{k_0^2,\tilde{k}_1^2\}]\sum_{\lambda\in \Lambda}p^{\lambda} \nu(E(\lambda)).\]
This together with equation (\ref{eq:ztau}) implies
\[ \sum_{\lambda\in\Lambda}p^{\lambda}\nu(E^*(\lambda))\leq \tau^{-2} [\hat{\nu}(\{|\xi|<1\})\cdot \epsilon +\max\{k_0^2,\tilde{k}_1^2\}]\sum_{\lambda\in \Lambda}p^{\lambda} \nu(E(\lambda)).\]
Finally we get the desired inequality setting
\begin{align*}
\tau=[\hat{\nu}(\{|\xi|<1\})\cdot \epsilon +\max\{k_0^2,\tilde{k}_1^2\}]^{1/3},\\
k_1=\tau^{-2} [\hat{\nu}(\{|\xi|<1\})\cdot \epsilon +\max\{k_0^2,\tilde{k}_1^2\}],
\end{align*}
and taking $\epsilon$ small enough such that 
\[\hat{\nu}(\{|\xi|<1\})\cdot \epsilon +\max\{k_0^2,\tilde{k}_1^2\}<1.\]
\end{proof}
Let $\Delta_1>0$ be any positive number and consider
\[\hat{n}_{\lambda}(x)=\frac{1}{\# \Lambda_{\omega(\lambda)}}\cdot \#\{\lambda'\in \Lambda_{\omega(\lambda)}:B_{\Delta_1 \rho}(x+a_{\lambda}^{\lambda'})\cap E(\lambda')\neq \emptyset \},\]
define $\hat{E}(\lambda)=\{x\in J: \hat{n}_{\lambda}(x)>\tau\}$. For a set $E\subset J$ denote by $V_{\delta}(E)$ the $\delta$-neighborhood of $E$.

\begin{cor}\label{cor:general}
Under the same hypothesis as Prop. \ref{fourier}, let $k_4>0$ such that $k_1<k_4<1$, if we choose $\Delta>0$ big enough and $\epsilon_1>0$ small enough such that
\begin{align*}
k_1\left(1+\frac{1+\Delta_1}{\Delta}\right)^{2n-1}&<k_4,\\
\epsilon_1 \left(1+\frac{1+\Delta_1}{\Delta}\right)^{2n-1} &< \epsilon,
\end{align*}
then $\nu(V_{\Delta \rho}(E(\lambda)))\leq \epsilon_1$, for all $\lambda\in \Lambda$, implies that
\[\sum_{\lambda \in \Lambda}p^{\lambda}\nu(V_{\Delta \rho}(\hat{E}(\lambda)))\leq k_4 \sum_{\lambda \in \Lambda}p^{\lambda}\nu(V_{\Delta \rho}(E(\lambda))).\]
\end{cor}
\begin{proof}
First observe that (see \cite{GL})
\[\nu(V_{\rho+\Delta\rho+\Delta_1\rho}(E(\lambda)))\leq \left(1+\frac{1+\Delta_1}{\Delta}\right)^{2n-1} \nu(V_{\Delta\rho}(E(\lambda))).\]
Now consider the family $A(\lambda)=V_{\rho+\Delta\rho+\Delta_1\rho}(E(\lambda))$, by our choice of $\epsilon_1$ we can apply proposition \ref{fourier} to $A(\lambda)$. Notice also that if $x\in V_{\Delta \rho}(\hat{E}(\lambda))$ then there exist $y\in B_{\Delta\rho}(x)$ and $\tau\cdot \#\Lambda_{\omega(\lambda)}$ elements $\lambda'\in\Lambda_{\omega(\lambda)}$ such that 
\[B_{\Delta_1 \rho}(y+a_{\lambda}^{\lambda'})\cap E(\lambda')\neq \emptyset,\]
thus 
\[B_{\rho}(x+a_{\lambda}^{\lambda'})\subset V_{\rho+\Delta\rho+\Delta_1\rho}(E(\lambda'))=A(\lambda').\]
This shows that $V_{\Delta \rho}(\hat{E}(\lambda))\subset A^*(\lambda)$. Applying proposition \ref{fourier} to $A(\lambda)$ gives
\begin{align*}\sum_{\lambda \in \Lambda}p^{\lambda}\nu(V_{\Delta \rho}(\hat{E}(\lambda))) &\leq \sum_{\lambda \in \Lambda}p^{\lambda}\nu(A^*(\lambda))\\ & \leq k_1 \sum_{\lambda \in \Lambda}p^{\lambda}\nu(V_{\rho+\Delta\rho+\Delta_1\rho}(E(\lambda))) \\
&\leq k_1\left(1+\frac{1+\Delta_1}{\Delta}\right)^{2n-1} \sum_{\lambda \in \Lambda}p^{\lambda}\nu(V_{\Delta\rho}(E(\lambda)))\\
&\leq k_4 \sum_{\lambda \in \Lambda}p^{\lambda}\nu(V_{\Delta \rho}(E(\lambda))).
\end{align*}
\end{proof}

\subsection{Proof of Theorem A}

In this subsection we will prove the multidimensional conformal scale recurrence lemma, first we will fix the values of the main parameters playing a role in the proof, this is done to make it clear that there are not contradictions between their values. Start choosing a positive constant $\mu$ such that
\begin{equation}
-\log r^{\tb}_{\ul{c}}<\mu,
\end{equation}
for any $\tb\in \Sigma_j^{-}$, and $\ul{c}=(c_0,c_1)\in \Sigma_j^{fin}$ a finite sequence with only two symbols. The choice of $\mu$ and the equation $r^{\tb}_{\ul{b}\ul{c}}=r^{\tb}_{\ul{b}}\cdot r^{\tb\ul{b}}_{\ul{c}}$ implies that
\[\log r^{\tb}_{\ul{b}\ul{c}}>\log r^{\tb}_{\ul{b}}-\mu.\]
This is saying that as the length of $\ul{b}$ increases the number $\log r^{\tb}_{\ul{b}}$ decreases by steps no bigger than $\mu$.

Now choose $c>0$ such that 
\[c^{-1}diam(G(\ul{a}^j))\leq diam(G^{\tb^j}(\ul{a}^j)) \leq c\, diam(G(\ul{a}^j)),\]
for $\ul{a}^j\in \Sigma_j^{fin}$, $\tb^j \in \Sigma_j^-$. Theses constants only depend on $K_1$,..., $K_n$. Fix $\tilde{c}_0>0$ such that
\begin{equation}\label{eq:c0}
2\log(c \tilde{c}_0)>\mu.
\end{equation}
We will use proposition \ref{fourier} with the following data
\begin{align*}
\Lambda&=\Sigma_1 (\tilde{c}_0,\rho)\times...\times \Sigma_n (\tilde{c}_0,\rho),\\
A&=\mathbb{A}_1\times ...\times \mathbb{A}_n,\\
\alpha(\ul{a}^1,...,\ul{a}^n)&=(a^1_0,...,a^n_0),\\
\omega(\ul{a}^1,...,\ul{a}^n)&=(a^1_{m_1},...,a^n_{m_n}),\\
a_{\lambda}^{\lambda'}&=(\log \frac{r_{\ul{b}^1}^{\tb^1}}{r_{\ul{b}^n}^{\tb^n}},...,\log \frac{r_{\ul{b}^{n-1}}^{\tb^{n-1}}}{r_{\ul{b}^n}^{\tb^n}},v_{\ul{b}^1}^{\tb^1},...,v_{\ul{b}^n}^{\tb^n}),
\end{align*}
where $\ul{a}^j=(a^j_0,...,a^j_{m_j})$, $\lambda=(\ul{a}^1,...,\ul{a}^n)$, $\lambda'=(\ul{b}^1,...,\ul{b}^n)$ and $\tb^j \in \Sigma_j^-$ finishes in $\ul{a}^j$, $1\leq j\leq n$.\footnote{For every $\ul{a}^j\in \Sigma_j^{fin}$ we choose, arbitrarily, an element $\tb^j\in \Sigma_j^-$ that ends in $\ul{a}^j$, using this we define $a_{\ld}^{\ld'}$.} Assume that the hypothesis of \ref{fourier} holds, namely that there exist $\Delta_0$, $k_0$ such that $\left\| T_{\xi}\right\|\leq k_0$ for all $|\xi|\in [1,\Delta_0\rho^{-1}]$, this will be verified in the next subsection. Applying the proposition in this setting gives constants $k_1$, $\tau$, $\epsilon$.

Now fix $k_4,k_5>0$ such that $k_1<k_4<k_5<1$ and $\delta>0$ such that
\begin{equation}\label{eq:delta}
k_4+2\cdot 3^n k_1^{-1} LC_2^{-1}C_4\delta<k_5,
\end{equation}
where $C_2>0$, $L>0$ are constants such that
\begin{align}
L^{-1}\rho^{-(d_1+...+d_n)}&\leq \# \Lambda_i \leq L \rho^{-(d_1+...+d_n)},\label{eq:L}\\
C_2\rho^{d_1+....+d_n}&\leq p^{\ld},\label{eq:c2}
\end{align}
for any $i\in A$, $\ld\in \Ld$, and $C_4>0$ is defined by Eq. (\ref{eq:C4}). All these constants depend only in $\tilde{c}_0$.

Fix $r>0$ such that 
\begin{equation}\label{eq:r}
2r>\delta^{-1}(9+\frac{1}{4})\log (c\tilde{c}_0).
\end{equation}
The choice of $\mu$ and $\tilde{c}_0$ allow us to find $\rho_{1}>0$, small enough, such that for any family of intervals $I_1,..., I_{n-1}$, with $diam(I_j)\geq 2 \log(c \tilde{c}_0)$, any $x=(t,v)\in J$ with $dist(t_j,I_j)\leq \delta^{-1}(9+\frac{1}{4})\log (c\tilde{c}_0)$ and any $\lambda\in \Lambda$, there exists $\lambda_0=(\ul{b}^1,...,\ul{b}^n) \in \Sigma_1^{fin}\times...\times \Sigma_n^{fin}$ with the property 
\[x+a_{\lambda}^{\lambda_0}\in I_1\times...\times I_{n-1}\times \mathbb{T}^{n},\]
and $diam (G({\ul{b}^j}))>\rho_1$, $1\leq j\leq n$.
Choose $c_0>\tilde{c}_0$ big enough such that \footnote{If $\lambda=(\ul{a}^1,...,\ul{a}^n)$, $\lambda_0=(\ul{b}^1,....,\ul{b}^n)$, with $\omega(\ld_0)=\alpha(\ld)$, then $\ld_0\ld=(\ul{b}^1\ul{a}^1,...,\ul{b}^n\ul{a}^n)$.}
\[\lambda_0 \lambda\in \Sigma_1(c_0,\rho)\times...\times \Sigma_n (c_0,\rho)\]
for all $\lambda_0=(\ul{b}^1,...,\ul{b}^n)$, such that $diam (G({\ul{b}^j}))>\rho_1$, $1\leq j\leq n$, and any $\lambda\in \Lambda$ with $\omega(\lambda_0)=\alpha(\lambda)$.

We also fix a constant $\Delta_1>0$ that should be big enough to verify Eq. (\ref{eq:Delta}), this is a condition that only depends on $c_0$. Finally, we choose $\Delta$, $\epsilon_1$ as in corollary \ref{cor:general}.

Notice that $\Lambda'=\Sigma_1(c_0,\rho)\times...\times \Sigma_n(c_0,\rho)$ contains $\Lambda$. Choose a function $\varphi:\Lambda'\to\Lambda$ such that:
\begin{itemize}
\item[(a)] If $\ld=(\ul{a}^1,...,\ul{a}^n)\in \Lambda'$ and $\varphi(\ld)=(\ul{b}^1,...,\ul{b}^n)$ then either $\ul{a}^j$ ends with $\ul{b}^j$ or $\ul{b}^j$ ends with $\ul{a}^j$, for every $1\leq j\leq n$.
\item[(b)] $\varphi(\ld)=\ld$, $\forall \ld \in \Ld$.
\end{itemize}
Thanks to properties (a),(b) of $\varphi$ there are constants $T_1$, $T_2$ depending only in $c_0,\, \tilde{c}_0$, and not $\rho$, such that 
\[1\leq T_1\leq \#\varphi^{-1}(\lambda)\leq T_2, \forall \ld\in \Ld.\]
We show that we can suppose $F(\lambda)=F(\varphi(\lambda))$. Assume that for the given values of $c_0$ and $r$ there exist $c_1,c_2,c_3,\rho_0>0$ such that the scale recurrence lemma is verified in the special case when \[F(\lambda)=F(\varphi(\lambda)), \forall \ld \in \Ld'.\]
We find new values for $c_1,c_2,c_3,\rho_0>0$ verifying the lemma in the general case. In fact, we do not need to change $c_2,c_3,\rho_0>0$, just redefine $c_1$ as $c_1/T_2$. Given a family $\{F(\ld)\}_{\ld\in\Ld'}$ with $\nu(J_r\setminus F(\ld))\leq c_1/T_2$ consider 
\[\tilde{F}(\ld)=\bigcap_{\ld'\in\,\varphi^{-1}(\varphi(\ld))}F(\ld').\]
This new family verifies $\tilde{F}(\lambda)=\tilde{F}(\varphi(\lambda))$, moreover $\nu(J_r\setminus \tilde{F}(\ld))\leq c_1$, then there exists $\tilde{F}^*(\lambda)$ with the propierties of the scale recurrence lemma. Taking $F^*(\ld)=\tilde{F}^*(\ld)$ gives the lemma in the general case.

We will prove that for the scale recurrence lemma to hold it is enough to prove the following statement:

\begin{stm}\label{statement}
For the given values of $c_0$ and $r$, there exist $c_1,c_2,c_3,\rho_0>0$ with the following properties: given $0<\rho<\rho_0$, and a family $F(\ld)$ of subsets of $J_r$, $\ld\in \Ld=\Sigma_1(\tilde{c}_0,\rho)\times...\times \Sigma_n(\tilde{c}_0,\rho)$, such that
\[\nu(J_r\setminus F(\ld))\leq c_1, \forall \ld,\] 
there is another family $F^*(\ld)$ of subsets of $J_r$ satisfying:
\begin{itemize}
\item[(i)] For any $\ld \in \Ld$, $F^*(\ld)$ is contained in the $c_2 \rho$-neighborhood of $F(\ld)$.
\item[(ii)] Let $\ld=(\ul{a}^1,...,\ul{a}^n)\in \Ld$, $(t,v)\in F^*(\ld)$; there exist at least $c_3\rho ^{-(d_1+...+d_n)}$ elements $\ld'=(\ul{b}^1,....,\ul{b}^n)\in \Ld'$ (with $\ul{b}^j$ starting with the last letter of $\ul{a}^j$) such that, if $\tb^j \in \Sigma_j^{-}$, $1\leq j\leq n$, verify $\tb^j \wedge \ul{a}^j\in \Sigma_j(c_0,\rho)$, $1\leq j\leq n$ and
\[T_{\ul{b}^1,...,\ul{b}^n}(\tb^1,...,\tb^n,t,v)=(\tb^1\ul{b}^1,...,\tb^n\ul{b}^n,\tilde{t},\tilde{v})\]
the $\rho$-neighborhood of $(\tilde{t},\tilde{v})\in J$ is contained in $F^*(\varphi(\ld'))$.
\item[(iii)] $\nu(F^*(\ld))\geq \nu(J_r)/2$ for at least $T_2/(T_2+T_1)$ of the $\ld\in \Ld$.
\end{itemize}
\end{stm}

The difference between statement \ref{statement} and the scale recurrence lemma is that $\Ld$ is parametrizing the sets $F(\ld)$ instead of $\Ld'$, however $\Ld'$, which is much bigger than $\Ld$, still parametrizes the set of renormalization operators.\\
Let $\{F(\ld)\}_{\ld\in \Ld'}$ be a family of sets as in the scale recurrence lemma, we can suppose that $F(\ld)=F(\varphi(\ld))$. Now assume that statement \ref{statement} holds, then we can apply it to the restricted family $\{F(\ld)\}_{\ld\in \Ld}$, this produces another family $\{F^*(\ld)\}_{\ld\in \Ld}$. We extend it to $\ld\in \Ld'$ by $F^*(\ld)=F^*(\varphi(\ld))$, it is easily seen that $\{F^*(\ld)\}_{\ld\in \Ld'}$ verifies the desired properties:
\begin{itemize}
\item[(i)] $F^*(\ld)=F^*(\varphi(\ld))\subset V_{c_2\rho}(F(\varphi(\ld)))=V_{c_2\rho}(F(\ld))$.
\item[(ii)] If $\ld=(\ul{a}^1,...,\ul{a}^n)\in \Ld'$, $(t,v)\in F^*(\ld)=F^*(\varphi(\ld))$, $\varphi(\ld)=(\ul{c}^1,...,\ul{c}^n)$, then there exist at least $c_3\rho ^{-(d_1+...+d_n)}$ elements $\ld'=(\ul{b}^1,...,\ul{b}^n)\in \Ld'$ such that for \[T_{\ul{b}^1,...,\ul{b}^n}(\tb^1,...,\tb^n,t,v)=(\tb^1\ul{b}^1,...,\tb^n\ul{b}^n,\tilde{t},\tilde{v})\]
the $\rho$-neighborhood of $(\tilde{t},\tilde{v})\in J$ is contained in $F^*(\varphi(\ld'))=F^*(\ld')$, for any $\tb^j \in \Sigma_j^{-}$, $1\leq j\leq n$, verifying $\tb^j \wedge \ul{c}^j\in \Sigma_j(c_0,\rho)$, $1\leq j\leq n$; in particular for any $\tb^1,...,\tb^n$ ending in $\ul{a}^1,...,\ul{a}^n$, respectively.
\item[(iii)] Let $\Ld_1\subset \Ld$ the set of $\ld$ such that $\nu(F^*(\ld))\geq \nu(J_r)/2$, we know that 
\[\frac{\#\Ld_1}{\#\Ld}\geq \frac{T_2}{T_2+T_1}.\]
Let $A=\#\varphi^{-1}(\Ld_1)$, $B=\#\varphi^{-1}(\Ld\setminus \Ld_1)$, then
\[\frac{B}{A}\leq \frac{\#(\Ld\setminus \Ld_1)\cdot T_2}{\# \Ld_1 \cdot T_1} \leq \frac{(1-T_2/(T_2+T_1))\cdot \#\Ld \cdot T_2}{(T_2/(T_2+T_1)) \cdot \#\Ld \cdot T_1}=1,\]
and from this we get that
\[\frac{A}{A+B}=\frac{1}{1+(B/A)}\geq \frac{1}{2}\]
of the $\ld \in \Ld'$ verifies $\nu(F^*(\ld))\geq \nu(J_r)/2$.
\end{itemize}
From now on we will focus in the proof of statement \ref{statement}.\\

Suppose we have a family of sets $\{F(\ld)\}_{\ld \in \Ld}$, define
\[E_0(\ld)=J_r\setminus V_{\Delta\rho}(F(\ld)), \ld\in \Ld.\]
Now we define recursively two families of sets $\{E_m(\ld)\}_{\ld\in \Ld}$ and $\{\tilde{E}_m(\ld)\}_{\ld\in \Ld}$. The set $\tilde{E}_m(\ld)$ is defined as the $x\in J_r$ such that\footnote{When $\ld'\in \Ld'\setminus\Ld$ the element $a_{\ld}^{\ld'}$ is defined in the same way as when $\ld,\ld' \in \Ld$.}
\[\#\{\ld'\in \Ld':\alpha(\ld')=\omega(\ld);\,\,B_{\Delta_1\rho}(x+a_{\lambda}^{\lambda'})\subset J_r\setminus E_m(\varphi(\ld'))\}\leq c_3\rho^{-(d_1+...+d_n)},\]
and $E_{m+1}(\ld)=E_0(\ld)\cup \tilde{E}_m(\ld)$.\footnote{In fact, since $\tilde{E}_j(\ld) \subset \tilde{E}_{j+1}(\ld)$, we have  $E_m(\ld)=E_0(\ld)\cup \tilde{E}_0(\ld)\cup...\cup \tilde{E}_{m-1}(\ld)$.} The value of $c_3$ will be fixed during the proof of the next lemma. Note that for $x\in\tilde{E}_0(\ld)$ one has
\[\{\ld'\in \Ld':\,B_{\Delta_1\rho}(x+a_{\lambda}^{\lambda'})\subset J_r\setminus E_1(\varphi(\ld'))\} \subset \{\ld'\in \Ld':\,B_{\Delta_1\rho}(x+a_{\lambda}^{\lambda'})\subset J_r \setminus E_0(\varphi(\ld'))\},\]
hence $\tilde{E}_0(\ld)\subset \tilde{E}_1(\ld)$. Analogously one proves that $\tilde{E}_m(\ld)\subset \tilde{E}_{m+1}(\ld)$, $E_m(\ld)\subset E_{m+1}(\ld)$ for all $m$.
\begin{lem}\label{lem:en}
If $c_1$, $c_3$, $\rho_0$ are sufficiently small then
\begin{align}
\sum_{\ld \in \Ld}p^{\ld}\nu(V_{\Delta \rho}(\tilde{E}_m(\ld)))&\leq k_5 \sum_{\ld \in \Ld}p^{\ld}\nu(V_{\Delta \rho}(E_m(\ld))), \label{eq:tilde1} \\
\sum_{\ld \in \Ld}p^{\ld}\nu(V_{\Delta \rho}(E_m(\ld))) &\leq \frac{\nu(J_{r+\Delta \rho}\setminus J_r)+c_1}{1-k_5} \label{eq:tilde2}
\end{align}
\end{lem}
Before proving the lemma we will use it to prove statement \ref{statement}. Consider $E_{\infty}(\ld)=\bigcup_{m\geq 0}E_m(\ld)$, thanks to Eq. (\ref{eq:tilde2}) we have
\begin{equation}\label{eq:einfty}\sum_{\ld \in \Ld}p^{\ld}\nu(V_{\Delta \rho}(E_{\infty}(\ld))) \leq \frac{\nu(J_{r+\Delta \rho}\setminus J_r)+c_1}{1-k_5}.\end{equation}
Now define $F^*(\ld)=J_r\setminus E_{\infty}(\ld)$, we will prove that this family of sets has the desired properties.
\begin{itemize}
\item[(i)] Since $E_0(\ld)\subset E_{\infty}(\ld)$ then $F^*(\ld)\subset J_r\setminus E_0(\ld)=V_{\Delta \rho}(F(\ld))$, choosing $c_2=\Delta$ gives the first property.
\item[(ii)] Let $x\in F^*(\ld)$, then $x\notin \tilde{E}_m(\ld)$ and the set
\[A_m=\{\ld'\in \Ld': \alpha(\ld')=\omega(\ld);\,\,B_{\Delta_1 \rho}(x+a_{\ld}^{\ld'})\subset J_r\setminus E_m(\varphi(\ld'))\}\]
has more than $c_3\rho^{-(d_1+...+d_n)}$ elements, for all $m$. Moreover, since $E_m(\varphi(\ld'))\subset E_{m+1}(\varphi(\ld'))$ one has $A_{m+1}\subset A_m$ and then \[\#\left(\bigcap_{m\geq 0}A_m\right)\geq c_3\rho^{-(d_1+...+d_n)}.\]
Therefore 
\[\#\{\ld'\in \Ld': \alpha(\ld')=\omega(\ld);\,\,B_{\Delta_1 \rho}(x+a_{\ld}^{\ld'})\subset F^*(\varphi(\ld'))\}\geq c_3\rho^{-(d_1+...+d_n)}.\]
To finish, it is enough to prove that for any $\tilde{\tb}^1$,..., $\tilde{\tb}^n$ such that $\tilde{\tb}^j\wedge \ul{a}^j\in \Sigma_j(c_0,\rho)$, $1\leq j\leq n$, one has
\[B_{\rho}(x+(\log \frac{r_{\ul{b}^1}^{\tilde{\tb}^1}}{r_{\ul{b}^n}^{\tilde{\tb}^n}},...,\log \frac{r_{\ul{b}^{n-1}}^{\tilde{\tb}^{n-1}}}{r_{\ul{b}^n}^{\tilde{\tb}^n}},v_{\ul{b}^1}^{\tilde{\tb}^1},...,v_{\ul{b}^n}^{\tilde{\tb}^n}))\subset B_{\Delta_1 \rho}(x+a_{\ld}^{\ld'}),\]
where $\ld'=(\ul{b}^1,..,\ul{b}^n)$, $\ld=(\ul{a}^1,..,\ul{a}^n)$, $a_{\ld}^{\ld'}=(\log \frac{r_{\ul{b}^1}^{\tb^1}}{r_{\ul{b}^n}^{\tb^n}},..,\log \frac{r_{\ul{b}^{n-1}}^{\tb^{n-1}}}{r_{\ul{b}^n}^{\tb^n}},v_{\ul{b}^1}^{\tb^1},..,v_{\ul{b}^n}^{\tb^n})$, for some $\tb^j$ ending in $\ul{a}^j$, $1\leq j\leq n$. This is accomplished by taking $\Delta_1$ big. More precisely, since
\[ \|D(k^{\tb^j}\circ (k^{\tilde{\tb}^j})^{-1})(z)-I\|\leq C diam (G(\tb^j\wedge \tilde{\tb}^j))\leq \tilde{C}\rho,\] 
for all $1\leq j\leq n$, for some constant $\tilde{C}$ only depending on $c_0$, we can conclude that 
\[\left| \log r_{\ul{b}}^{\tb^j}-\log r_{\ul{b}}^{\tilde{\tb}^j}\right|\leq \tilde{C}_1 \rho,\,\,\,\left| v_{\ul{b}}^{\tb^j}-v_{\ul{b}}^{\tilde{\tb}^j}\right|\leq \tilde{C}_1 \rho,\]
for all $1\leq j\leq n$, for some constant $\tilde{C}_1$ only depending on $c_0$. Therefore, imposing
\begin{equation}\label{eq:Delta}
1+2(2n-1)\tilde{C}_1<\Delta_1
\end{equation}
would be sufficient to guarantee the second property.
\item[(iii)] By Eq. (\ref{eq:einfty}), choosing $c_1$, $\rho_0$ small such that 
\[\frac{\nu(J_{r+\Delta \rho}\setminus J_r)+c_1}{1-k_5}<\frac{C_1T_1}{2(T_1+T_2)}\nu(J_r),\]
where $C_1$ is a constant such that $p^{\ld}\geq C_1 (\#\Ld)^{-1}$, $\forall \ld \in \Ld$, implies that
\[\sum_{\ld \in\Ld} p^{\ld}\nu(F^*(\ld))\geq \left(1-\frac{C_1T_1}{2(T_1+T_2)}\right)\nu(J_r).\]
Let $A=\{\ld : \nu(F^*(\ld))\geq \nu(J_r)/2\}$, hence
\begin{align*}
\left(1-\frac{C_1T_1}{2(T_1+T_2)}\right)\nu(J_r) &\leq \sum_{\ld \in A} p^{\ld}\nu(F^*(\ld)) + \sum_{\ld \in \Ld\setminus A} p^{\ld}\nu(F^*(\ld))\\
&\leq \nu(J_r) \sum_{\ld \in A} p^{\ld} + \frac{\nu(J_r)}{2} \sum_{\ld \in \Ld\setminus A} p^{\ld}\\
&= \nu(J_r)-\frac{\nu(J_r)}{2} \sum_{\ld \in \Ld\setminus A} p^{\ld}.
\end{align*}
From this inequality we get
\[\frac{C_1}{2}\frac{\#(\Ld\setminus A)}{\#\Ld}\leq \frac{1}{2} \sum_{\ld \in \Ld\setminus A} p^{\ld} \leq \frac{C_1T_1}{2(T_1+T_2)}.\]
Finally this implies $\# A>(T_2/(T_1+T_2))\#\Ld$, as we wanted.
\end{itemize}
\begin{proof}[Proof of lemma \ref{lem:en}:]
Choose $c_3>0$, $\epsilon_2>0$ small such that
\begin{equation}\label{eq:epsilon2}
\left( c_3+ C_2^{-1}\frac{\epsilon_2}{\epsilon_1}\right)\rho^{-(d_1+...+d_n)}<(1-\tau)N_{\omega(\ld)},\,\,\forall \ld \in \Ld,
\end{equation}
where $C_2$ is a constant such that $p^{\ld}\geq C_2 \rho^{d_1+...+d_n}$, $\forall \ld \in \Ld$. We suppose that $c_1$, $\rho_0$ are small enough such that
\begin{equation}\label{eq:nose}
\frac{\nu(J_{r+\Delta \rho}\setminus J_r)+c_1}{1-k_5}<\epsilon_2.
\end{equation}
We will proceed by induction following the scheme 
\begin{center}
\textit{Eq. (\ref{eq:tilde2}) for $m$ $\Rightarrow$ Eq. (\ref{eq:tilde1}) for $m$ $\Rightarrow$ Eq. (\ref{eq:tilde2}) for $m+1$.}
\end{center}
For the base case, Eq. (\ref{eq:tilde2}) for $m=0$, notice that
\[V_{\Delta \rho}(E_0(\ld))\subset J_{r+\Delta \rho}\setminus F(\ld).\]
Therefore
\[\nu(V_{\Delta \rho}(E_0(\ld))) \leq \nu(J_{r+\Delta \rho}\setminus J_r)+\nu(J_r \setminus F(\lambda))\leq \frac{\nu(J_{r+\Delta \rho}\setminus J_r)+c_1}{1-k_5}.\]
Now we prove \textit{"Eq. (\ref{eq:tilde2}) for $m$ $\Rightarrow$ Eq. (\ref{eq:tilde1}) for $m$"}: define $\Ld_b=\{\ld\in \Ld: \nu(V_{\Delta \rho}(E_m(\ld)))<\epsilon_1\}$ and 
\[ A(\lambda)=\begin{cases}
		E_m(\ld)  & \mbox{if } \ld \in \Ld_b \\
		\emptyset & \mbox{otherwise}
	\end{cases}
\]
We will show that
\[\tilde{E}_m(\ld)\cap J_{r-2\log (c\tilde{c}_0)-\Delta_1\rho}\subset \hat{A}(\ld).\]
Here $\hat{A}(\ld)$ is the set form corollary \ref{cor:general}, i.e. the $x\in J$ such that
\[\frac{1}{\# \Lambda_{\omega(\lambda)}}\cdot \#\{\lambda'\in \Lambda_{\omega(\lambda)}:B_{\Delta_1 \rho}(x+a_{\lambda}^{\lambda'})\subset J\setminus A(\lambda') \} \leq 1-\tau.\]
Using Eq. (\ref{eq:nose})
\[\epsilon_2>\sum_{\ld \in \Ld}p^{\ld}\nu(V_{\Delta \rho}(E_m(\ld)))\geq C_2 \rho^{d_1+...+d_n}\epsilon_1 \#(\Ld\setminus\Ld_b),\]
hence
\[\#(\Ld\setminus\Ld_b)\leq C_2^{-1}\frac{\epsilon_2}{\epsilon_1}\rho^{-(d_1+...+d_n)}.\]
Let $x\in \tilde{E}_m(\ld)\cap J_{r-2\log (c\tilde{c}_0)-\Delta_1\rho}$, using the last inequality, Eq. (\ref{eq:epsilon2}) and the fact that $a_{\ld}^{\ld'}\in J_{2\log (c \tilde{c}_0)}$, $\forall \ld, \ld'\in \Ld$, we have
\begin{align*}
\#\{\lambda'\in \Ld_{\omega(\ld)}:B_{\Delta_1 \rho}(x+a_{\lambda}^{\lambda'})&\subset J\setminus A(\lambda') \} \\ 
&\leq \#\{\lambda'\in \Lambda_b:B_{\Delta_1 \rho}(x+a_{\lambda}^{\lambda'})\subset J\setminus E_m(\varphi(\lambda')) \}\\ &\,\,\,\,\,+ \#(\Ld\setminus \Ld_b)\\
&\leq \#\{\lambda'\in \Lambda':B_{\Delta_1 \rho}(x+a_{\lambda}^{\lambda'})\subset J_r\setminus E_m(\varphi(\lambda'))\}\\
&\,\,\,\,+C_2^{-1}\frac{\epsilon_2}{\epsilon_1}\rho^{-(d_1+...+d_n)}\\
&\leq (c_3+C_2^{-1}\frac{\epsilon_2}{\epsilon_1})\rho^{-(d_1+...+d_n)}\\
&<(1-\tau)N_{\omega(\ld)}.
\end{align*}
Hence $x\in \hat{A}(\ld)$, and we have shown
\[\tilde{E}_m(\ld)\cap J_{r-2\log (c\tilde{c}_0)-\Delta_1\rho}\subset \hat{A}(\ld).\]
Clearly the family $A(\ld)$ satifies the hypothesis of corollary \ref{cor:general}, i.e. $\nu(V_{\Delta \rho}(A(\ld)))<\epsilon_1$. Using the corollary gives
\begin{align}
\sum_{\ld\in\Ld}p^{\ld}\nu(V_{\Delta\rho}(\tilde{E}_m(\ld)\cap J_{r-2\log (c\tilde{c}_0)-\Delta_1\rho}))&\leq \sum_{\ld\in\Ld}p^{\ld}\nu(V_{\Delta\rho}(\hat{A}(\ld)))\nonumber \\
&\leq k_4 \sum_{\ld\in\Ld}p^{\ld}\nu(V_{\Delta\rho}(A(\ld)))\nonumber \\
&\leq k_4\sum_{\ld\in\Ld}p^{\ld}\nu(V_{\Delta\rho}(E_m(\ld))). \label{eq:enintervalo}
\end{align}
Now we estimate $\tilde{E}_m(\ld)$ outside of $J_{r-2\log (c\tilde{c}_0)-\Delta_1\rho}$. Assume that $\rho_0$ is small enough such that $2\Delta \rho<\frac{1}{4}\log (c\tilde{c}_0)$, $\Delta_1 \rho<\frac{1}{4}\log (c\tilde{c}_0)$. By Eq. (\ref{eq:r}), for every $j=1,...,n$, there exist intervals
\begin{align*}
I_j^- &\subset [-r,-r+\delta^{-1}(9\log (c\tilde{c}_0)+2\Delta \rho)]\subset [-r,r],\\
I_j^+ &\subset [r-\delta^{-1}(9\log (c\tilde{c}_0)+2\Delta \rho),r]\subset [-r,r],
\end{align*}
with $diam(I_j^-)=diam(I_j^+)=9\log (c\tilde{c}_0)$ such that
\begin{align*}
    \sum_{\ld\in \Ld}p^{\ld}\nu(V_{\Delta \rho}(E_m(\ld)\cap J^-_j))&\leq \delta \sum_{\ld\in \Ld}p^{\ld}\nu(V_{\Delta \rho}(E_m(\ld))),\\
    \sum_{\ld\in \Ld}p^{\ld}\nu(V_{\Delta \rho}(E_m(\ld)\cap J^+_j))&\leq \delta \sum_{\ld\in \Ld}p^{\ld}\nu(V_{\Delta \rho}(E_m(\ld))),
\end{align*}
where $J^{+(-)}_j=\{(t_1,...,t_{n-1},v)\in J_r:t_j\in I^{+(-)}_j\}$.

For every pair $(A,B)$ such that $A,B \subset \{1,...,n-1\}$ and $A\cap B=\emptyset$ consider $y_{A,B}=(t_1(A,B),...,t_{n-1}(A,B),0)\in J$ given by
\[ t_j(A,B)=
	\begin{cases}
		-r+\log(c\tilde{c}_0)  & \mbox{if } j \in A, \\
		r-\log(c\tilde{c}_0)  & \mbox{if } j \in B, \\
		0 & \mbox{otherwise}.
	\end{cases}
\]
We also define sets $\tilde{J}_{A,B}, J_{A,B} \subset J$ in the following way:
\begin{align*}
J_{A,B}=\{(t_1,...,t_{n-1},v)\in J_{r}:t_j\in I^{+}_j\,& \mbox{if } j\in B, t_j\in I^{-}_j\, \mbox{if } j\in A\},& \\ 
\tilde{J}_{A,B}=\{(t_1,...,t_{n-1},v)\in J_r:t_j\in \tilde{I}^{+}_j\,& \mbox{if } j\in B, t_j\in \tilde{I}^{-}_j\, \mbox{if } j\in A\},
\end{align*}
where the interval $\tilde{I}^{u}_j$ has the same center as $I^{u}_j$ and has length $2\log(c\tilde{c}_0)$, here $u=+$ or $-$.

Fix $\ld\in \Ld$, from the choice of $\rho_1$ and $c_0$ we know that there is $\ld_{A,B}\in \Sigma_1^{fin}\times...\times \Sigma_n^{fin}$ such that
\[y_{A,B}+a_{\ld}^{\ld_{A,B}}\in \tilde{J}_{A,B},\]
and $\ld_{A,B}\ld'\in \Ld'$, $\forall \ld'\in \Ld_{\omega(\ld_{A,B})}$. Let
\[x\in V_{\Delta\rho}(\tilde{E}_m(\ld)\cap J_r\setminus J_{r-2\log (c\tilde{c}_0)-\Delta_1 \rho}),\]
then there is $y\in \tilde{E}_m(\ld)\cap J_r\setminus J_{r-2\log (c\tilde{c}_0)-\Delta_1 \rho}$, $y\in B_{\Delta \rho}(x)$, and for $y$ we have
\[\#\{\ld'\in\Ld': B_{\Delta_1\rho}(y+a_{\ld}^{\ld'})\subset J_r\setminus E_m(\varphi(\ld'))\}<c_3\rho^{-(d_1+...+d_n)}.\]
Write $y=(t_1,...,t_{n-1},v)$ and consider the sets
\begin{align*}
    A&=\{j\in [1,n-1]\cap \Z: t_j<-r+2\log (c\tilde{c}_0)+\Delta_1 \rho\},\\
    B&=\{j\in [1,n-1]\cap \Z: t_j>r-2\log (c\tilde{c}_0)-\Delta_1 \rho\}.
\end{align*}
Since $y\notin J_{r-2\log (c\tilde{c}_0)-\Delta_1 \rho}$ we know that $A\cup B\neq \emptyset$ and we can consider $\lambda_{A,B}$. Given that $\#\Ld_{\omega(\ld_{A,B})}\geq L^{-1}\rho^{-(d_1+...+d_n)}$, we conclude that
\begin{align*}
\#\{\ld'\in\Ld: B_{\Delta_1\rho}(y+a_{\ld}^{\ld_{A,B}\ld'})\cap E_m(\varphi(\ld_{A,B}\ld'))\neq \emptyset\}&\geq (L^{-1}-c_3)\rho^{-(d_1+...+d_n)}\\
&>\frac{L^{-1}}{2}\rho^{-(d_1+...+d_n)},
\end{align*}
here we are assuming that $c_3<L^{-1}/2$. Notice that $B_{\Delta_1\rho}(y+a_{\ld}^{\ld_{A,B}\ld'})\cap J_r\subset J_{A,B}$, $\forall \ld'\in \Ld_{\omega(\ld_0)}$, therefore
\[\#\{\ld'\in\Ld_{\omega(\ld_{A,B})}:x+a_{\ld}^{\ld_{A,B}\ld'}\in V_{\Delta\rho+\Delta_1\rho}(E_m(\varphi(\ld_{A,B}\ld'))\cap J_{A,B})\}>\frac{L^{-1}}{2}\rho^{-(d_1+...+d_n)}.\]
Hence $V_{\Delta\rho}(\tilde{E}_n(\ld)\cap J_r\setminus J_{r-2\log (c\tilde{c}_0)-\Delta_1 \rho})$ is contained in 
\[ \{x: \sum_{(A,B)} \sum_{\ld'\in \Ld} 1_{V_{\Delta\rho+\Delta_1\rho}(E_m(\varphi(\ld_{A,B}\ld'))\cap J_{A,B})-a_{\ld}^{\ld_{A,B}\ld'}}(x)>\frac{L^{-1}}{2}\rho^{-(d_1+...+d_n)}\},\]
where the first sum is over all pairs $(A,B)$ such that $A,B\subset \{1,...,n\}$, $A\cap B=\emptyset$, $A\cup B\neq \emptyset$. Now using Chebyshev's inequality we get
\begin{align*}
\nu(V_{\Delta\rho}&(\tilde{E}_m(\ld)\cap J_r\setminus J_{r-2\log (c\tilde{c}_0)-\Delta_1 \rho}))\\&\leq \frac{1}{(L^{-1}/2)\rho^{-(d_1+...+d_n)}}\sum_{(A,B)}\sum_{\ld'\in \Ld}\nu(V_{\Delta\rho+\Delta_1\rho}(E_m(\varphi(\ld_{A,B}\ld'))\cap J_{A,B}))\\
&\leq 2\left(1+\frac{\Delta_1}{\Delta}\right)^{2n-1}L\rho^{d_1+...+d_n}\sum_{(A,B)}\sum_{\ld'\in \Ld}\nu(V_{\Delta\rho}(E_m(\varphi(\ld_{A,B}\ld'))\cap J_{A,B}))\\
&\leq 2\left(1+\frac{\Delta_1}{\Delta}\right)^{2n-1} C_4L\rho^{d_1+...+d_n}\sum_{(A,B)} \sum_{\ld'\in \Ld}\nu(V_{\Delta\rho}(E_m(\ld')\cap J_{A,B})).
\end{align*}
Where $C_4>0$ is a constant, only depending on $\tilde{c}_0$, such that
\begin{equation}\label{eq:C4}
\#\{\ld'\in \Ld_{\omega(\ld_{A,B})}:\varphi(\ld_{A,B}\ld')=\ld_1\}<C_4, \forall \ld_1 \in \Ld.
\end{equation}
Now we will sum over $\ld$. By the definition of $\Delta$ we know that 
\[\left(1+\frac{\Delta_1}{\Delta}\right)^{2n-1}<k_1^{-1},\]
and using $p^{\ld'}\geq C_2\rho^{d_1+...+d_n}$ we get
\begin{align*}
\sum_{\ld\in \Ld}p^{\ld}\nu(V_{\Delta\rho}&(\tilde{E}_m(\ld)\cap J_r\setminus J_{r-2\log (c\tilde{c}_0)-\Delta_1 \rho}))\\&\leq 2\left(1+\frac{\Delta_1}{\Delta}\right)^{2n-1} C_4 L\rho^{d_1+..+d_n}\sum_{(A,B)} \sum_{\ld'\in \Ld}\nu(V_{\Delta\rho}(E_m(\ld')\cap J_{A,B}))\\&\leq 2k_1^{-1}C_4L\rho^{d_1+..+d_n}\frac{1}{C_2\rho^{d_1+..+d'_n}}\sum_{(A,B)} \sum_{\ld'\in \Ld}p^{\ld'}\nu(V_{\Delta\rho}(E_m(\ld')\cap J_{A,B}))\\&\leq 2\cdot 3^n k_1^{-1} LC_2^{-1}C_4\delta \sum_{\ld\in \Ld}p^{\ld}\nu(V_{\Delta\rho}(E_m(\ld))).
\end{align*}
Putting this inequality together with (\ref{eq:enintervalo}) gives
\begin{align*}
\sum_{\ld\in \Ld}p^{\ld}\nu(V_{\Delta\rho}(\tilde{E}_m(\ld)))&\leq (k_4+2\cdot 3^n k_1^{-1} LC_2^{-1}C_4\delta)\sum_{\ld\in \Ld}p^{\ld}\nu(V_{\Delta\rho}(E_m(\ld)))\\
&\leq k_5 \sum_{\ld\in \Ld}p^{\ld}\nu(V_{\Delta\rho}(E_m(\ld))),
\end{align*}
this is Eq. (\ref{eq:tilde1}) for $m$. Here we have used eq. (\ref{eq:delta}) where $\delta$ was chosen.

To finish, we prove \textit{"Eq. (\ref{eq:tilde2}) for $m$ and Eq. (\ref{eq:tilde1}) for $m$ $\Rightarrow$ Eq. (\ref{eq:tilde2}) for $m+1$"}. Since $E_{m+1}(\ld)=E_0(\ld)\cup \tilde{E}_m(\ld)$ and $\nu(V_{\Delta\rho}(E_0(\ld)))\leq \nu(J_{r+\Delta \rho}\setminus J_r)+c_1$ we get
\begin{align*}
\sum_{\ld\in\Ld}p^{\ld}\nu(V_{\Delta\rho}(E_{m+1}(\ld)))&\leq \sum_{\ld\in\Ld}p^{\ld}\nu(V_{\Delta\rho}(E_0(\ld)))+\sum_{\ld\in\Ld}p^{\ld}\nu(V_{\Delta\rho}(\tilde{E}_{m}(\ld)))\\
&\leq \nu(J_{r+\Delta \rho}\setminus J_r)+c_1+k_5\sum_{\ld\in\Ld}p^{\ld}\nu(V_{\Delta\rho}(E_{m}(\ld)))\\
&\leq \nu(J_{r+\Delta \rho}\setminus J_r)+c_1+k_5\frac{\nu(J_{r+\Delta \rho}\setminus J_r)+c_1}{1-k_5}\\ 
&=\frac{\nu(J_{r+\Delta \rho}\setminus J_r)+c_1}{1-k_5}
\end{align*}
\end{proof}

\subsection{Proof of Hypothesis on Proposition \ref{fourier}}
In this subsection we prove that there exist $0<k_0<1$, $\Delta_0>0$ such that $\|T_{\xi}\|\leq k_0$ for all $|\xi|\in [1,\Delta_0\rho^{-1}]$, and these constants does not depend on $\rho$. Remember that the operator $T_{\xi}:\C^{\Ld}\to \C^{\Ld}$ is given by $T_{\xi}((z_{\ld})_{\ld \in\Ld})=(w_{\ld})_{\ld \in\Ld}$ where
\[w_{\ld}=\sum_{\ld'\in\Ld} p_{\ld}^{\ld'}\xi(a_{\ld}^{\ld'})z_{\ld'}.\]
Notice that $\C^{\Ld}$ can be decomposed in two ways
\[\C^{\Ld}=\bigoplus_{i\in A} \C^{\Ld_i},\,\,\,\C^{\Ld}=\bigoplus_{j\in A} \C^{\Ld^j},\]
and the operator $T_{\xi}$ sends $\C^{\Ld_i}$ into $\C^{\Ld^i}$. Let $\|\cdot\|_i$, $\|\cdot\|^{j}$ be the restriction to $\C^{\Ld_i}$, $\C^{\Ld^j}$, respectively, of the norm $\|\cdot\|$ on $\Ld$. Note that
\[\|z\|^2=\sum_{i\in A}\|\pi_i(z)\|_i^2=\sum_{j\in A}(\|\pi^j(z)\|^j)^2, \forall z\in \C^{\Ld},\]
where $\pi_i:\C^{\Ld}\to \C^{\Ld_i}$, $\pi^j:\C^{\Ld}\to \C^{\Ld^j}$ are the projections given by the decompositions. This implies that $\|T_{\xi}\|\leq k_0$ if and  only if $\|T_{\xi}|_{\C^{\Ld_i}}\|\leq k_0$, $\forall i\in A$, where $T_{\xi}|_{\C^{\Ld_i}}$ is the restriction
\[T_{\xi}|_{\C^{\Ld_i}}:(\C^{\Ld_i},\|\cdot\|_i)\to (\C^{\Ld^i},\|\cdot\|^i).\]
We start by supposing that there exist $\rho>0$, $|\xi|\in [1,\Delta_0\rho^{-1}]$, $i\in A$ such that 
\[\|T_{\xi}|_{\C^{\Ld_i}}\|\geq (1-\eta_0)^{1/2}.\]
From this we will derive a series of inequalities depending on parameters $\eta_0$, $\eta_1$, $\eta_2$,... each new parameter $\eta_{j+1}$ will depend on $\eta_{j}$, not in $\rho$, and $\lim_{\eta_j\to 0}\eta_{j+1}=0$. Finally, we will see that with the appropriate value of $\Delta_0$ the last $\eta_j$ will be bounded away from zero and then also $\eta_0$, this will complete the proof.

By our assumption there is $z=(z_{\ld})\in\C^{\Ld_i}$ such that $\sum_{\ld\in \Ld_i}p^{\ld}|z_\ld|^2=1$ and for $w=T_{\xi}(z)$
\[\|w\|^2=\sum_{\ld\in\Ld^i}p^{\ld}|w_{\ld}|^2\geq 1-\eta_0.\]
Note that
\[|w_{\ld}|^2=\left| \sum_{\ld'\in \Ld}p_{\ld}^{\ld'}\xi(a_{\ld}^{\ld'})z_{\ld'}\right|^2\leq \left( \frac{1}{\#\Ld_i}\sum_{\ld'\in \Ld_i}|z_{\ld'}|\right)^2\leq \frac{1}{\#\Ld_i}\sum_{\ld'\in \Ld_i}|z_{\ld'}|^2.\]
Consider the set
\[\tilde{\Ld}=\{\ld \in \Ld^i: |w_{\ld}|^2 \geq (1-\eta_1)\frac{1}{\#\Ld_i}\sum_{\ld'\in \Ld_i}|z_{\ld'}|^2\},\]
where $\eta_1=\eta_0^{1/2}$. Hence
\begin{align*}
1-\eta_0&\leq \sum_{\ld\in \Ld^i}p^{\ld}|w_{\ld}|^2= \sum_{\ld\in \tilde{\Ld}}p^{\ld}|w_{\ld}|^2+ \sum_{\ld\in \Ld^i\setminus \tilde{\Ld}}p^{\ld}|w_{\ld}|^2\\
&<\sum_{\ld\in \tilde{\Ld}}p^{\ld}\left(\sum_{\ld'\in \Ld}p_{\ld}^{\ld'}|z_{\ld'}|^2\right)+ \sum_{\ld\in \Ld^i\setminus \tilde{\Ld}}p^{\ld}\left((1-\eta_1)\sum_{\ld'\in \Ld}p_{\ld}^{\ld'}|z_{\ld'}|^2\right)\\
&=\sum_{\ld,\ld'\in \Ld} p^{\ld}p_{\ld}^{\ld'}|z_{\ld'}|^2-\eta_1\sum_{\ld\in \Ld^i\setminus \tilde{\Ld}}p^{\ld}\left(\sum_{\ld'\in \Ld}p_{\ld}^{\ld'}|z_{\ld'}|^2\right).
\end{align*}
Since $\sum_{\ld,\ld'\in \Ld} p^{\ld}p_{\ld}^{\ld'}|z_{\ld'}|^2=\sum_{\ld'\in\Ld}p^{\ld'}|z_{\ld'}|^2=1$ and $\eta_1^2=\eta_0$ we get
\begin{align*}
\eta_1&>\sum_{\ld\in \Ld^i\setminus \tilde{\Ld}}p^{\ld}\left(\sum_{\ld'\in \Ld}p_{\ld}^{\ld'}|z_{\ld'}|^2\right)\\
&= \sum_{\ld\in \Ld^i\setminus \tilde{\Ld}}\frac{p^{\ld}}{p^i}\sum_{\ld'\in \Ld}p^{\ld'}|z_{\ld'}|^2\\
&=\sum_{\ld\in \Ld^i\setminus \tilde{\Ld}}\frac{p^{\ld}}{p^i}\\
&\geq C_2\rho^{d_1+...+d_n} \#(\Ld^i\setminus \tilde{\Ld}),
\end{align*}
Putting $\eta_2=\eta_1/C_2$ we obtain
\[\#(\Ld^i\setminus \tilde{\Ld})\leq \eta_2 \rho^{-(d_1+...+d_n)}.\]
Proceeding as in \cite{MY} define $Z_{\ld}^{\ld'}=\xi(a_{\ld}^{\ld'})z_{\ld'}$, then
\[\frac{1}{2}\sum_{\ld'_0\in \Ld_i}\sum_{\ld'_1\in \Ld_i}p_{\ld}^{\ld'_0}p_{\ld}^{\ld'_1}|Z_{\ld}^{\ld'_0}-Z_{\ld}^{\ld'_1}|^2=\sum_{\ld'\in \Ld_i}p_{\ld}^{\ld'}|z_{\ld'}|^2-|w_{\ld}|^2.\]
If $\ld\in \tilde{\Ld}$
\[\frac{1}{2}\sum_{\ld'_0\in \Ld_i}\sum_{\ld'_1\in \Ld_i}p_{\ld}^{\ld'_0}p_{\ld}^{\ld'_1}|Z_{\ld}^{\ld'_0}-Z_{\ld}^{\ld'_1}|^2\leq \frac{\eta_1}{\#\Ld_i}\sum_{\ld'\in \Ld_i}|z_{\ld'}|^2=\frac{\eta_1}{p^i},\]
hence
\[\sum_{\ld'_0\in \Ld_i}\sum_{\ld'_1\in \Ld_i}|Z_{\ld}^{\ld'_0}-Z_{\ld}^{\ld'_1}|^2\leq \frac{2(\#\Ld_i)^2}{p^i}\eta_1.\]
Now set $\tilde{Z}_{\ld}^{\ld'}=z_{\ld'}-\xi(-a_{\ld}^{\ld'})w_{\ld}$, then
\begin{align*}
|\tilde{Z}_{\ld}^{\ld'}|^2&=|\xi(a_{\ld}^{\ld'})z_{\ld'}-w_{\ld}|^2=\left|\xi(a_{\ld}^{\ld'})z_{\ld'}-\frac{1}{\#\Ld_i}\sum_{\ld'_0\in \Ld_i}Z_{\ld}^{\ld'_0}\right|^2\\
&=\left|\frac{1}{\#\Ld_i}\sum_{\ld'_0\in \Ld_i}(Z_{\ld}^{\ld'}-Z_{\ld}^{\ld'_0})\right|^2\leq \frac{1}{\#\Ld_i}\sum_{\ld'_0\in \Ld_i}|Z_{\ld}^{\ld'}-Z_{\ld}^{\ld'_0}|^2.
\end{align*}
Summing over $\ld'$ gives
\[\sum_{\ld'\in \Ld_i}|\tilde{Z}_{\ld}^{\ld'}|^2\leq \frac{1}{\#\Ld_i}\sum_{\ld'\in \Ld_i}\sum_{\ld'_0\in \Ld_i}|Z_{\ld}^{\ld'}-Z_{\ld}^{\ld'_0}|^2\leq \frac{2\#\Ld_i}{p^i}\eta_1\leq \eta_3 \rho^{-(d_1+...+d_n)},\]
for $\eta_3$ a constant multiple of $\eta_1$. Pick $\ld_0,\ld_1\in \tilde{\Ld}$, then
\[z_{\ld'}=\xi(-a_{\ld_0}^{\ld'})w_{\ld_0}+\tilde{Z}_{\ld_0}^{\ld'}=\xi(-a_{\ld_1}^{\ld'})w_{\ld_1}+\tilde{Z}_{\ld_1}^{\ld'},\]
and from this\footnote{Redefine $\eta_3$ as $4\eta_3$.}
\[\sum_{\ld'\in \Ld_i}|\xi(-a_{\ld_0}^{\ld'})w_{\ld_0}-\xi(-a_{\ld_1}^{\ld'})w_{\ld_1}|^2=\sum_{\ld'\in \Ld_i}|\tilde{Z}_{\ld_1}^{\ld'}-\tilde{Z}_{\ld_0}^{\ld'}|^2\leq \eta_3 \rho^{-(d_1+...+d_n)}.\]
Observe that
\begin{align*}
|\xi(-a_{\ld_0}^{\ld'})w_{\ld_0}-\xi(-a_{\ld_1}^{\ld'})w_{\ld_1}| &= |\xi(a_{\ld_1}^{\ld'}-a_{\ld_0}^{\ld'})w_{\ld_0}-w_{\ld_1}|\\
&\geq min\{|w_{\ld_0}|,|w_{\ld_1}|\}\cdot \, 2 \sin \left(\frac{\left\langle \xi,a_{\ld_1}^{\ld'}-a_{\ld_0}^{\ld'}\right\rangle + \phi}{2} \right),
\end{align*}
where $\langle \xi, (t,v)\rangle=\sum_{j=1}^{n-1}\mu_j t_j+\sum_{j=1}^{n}m_j v_j\in \T$, for $\xi=(\mu,m)\in \R^{n-1}\times \Z^n$, and $\phi$ is the argument of the complex number $w_{\ld_0}/w_{\ld_1}$. Using this inequality together with\footnote{Here we assume $\eta_1<3/4$, which can be assumed without loss of generality.}
\[ |w_{\ld}|^2 \geq (1-\eta_1)\frac{1}{\#\Ld_i}\sum_{\ld'\in \Ld_i}|z_{\ld'}|^2 = \frac{1-\eta_1}{p^i} \geq \frac{1}{4},\,\,\,\forall \ld\in \tilde{\Ld},\]
we see that
\[\sum_{\ld'\in \Ld_i}\sin^2 \left(\frac{\left\langle \xi,a_{\ld_1}^{\ld'}-a_{\ld_0}^{\ld'}\right\rangle + \phi}{2} \right)\leq \eta_3 \rho^{-(d_1+...+d_n)}.\]
Let $\eta_4=\eta_3^{1/3}$, the previous inequality implies that
\[\sin \left(\frac{\left\langle \xi,a_{\ld_1}^{\ld'}-a_{\ld_0}^{\ld'}\right\rangle + \phi}{2} \right)\leq \eta_4,\]
for all $\ld'\in \Ld_i$ but $\eta_4\rho^{-(d_1+...+d_n)}$ $\ld'$'s. From this we get
\[\|\langle \xi,a_{\ld_1}^{\ld'}-a_{\ld_0}^{\ld'}\rangle+\phi \|\leq \eta_5\]
for all $\ld'\in \Ld_i$ but $\eta_5\rho^{-(d_1+...+d_n)}$ $\ld'$'s, where $\eta_5$ is a constant multiple of $\eta_4$.

Let $j_0$ such that $|\xi|=|\mu_{j_0}|$ or $|\xi|=|m_{j_0}|$. We will fix some specific 
$\ld_0, \ld_1 \in \tilde{\Ld}$ of the form
\begin{align*}
    \ld_0&=(\ul{d}^0,....,\ul{d}^{j_0-1},\ul{a}^0,\ul{d}^{j_0+1},...,\ul{d}^n),\\
    \ld_1&=(\ul{d}^0,....,\ul{d}^{j_0-1},\ul{a}^1,\ul{d}^{j_0+1},...,\ul{d}^n).
\end{align*}
Notice that $\ld_0, \ld_1$ only differ on the $j_0$ coordinate. Moreover, if $j_0\neq n$ we have
\begin{align*}
a_{\ld_1}^{\ld'}-a_{\ld_0}^{\ld'}&= (0,...,0,\log \frac{r_{\ul{b}^{j_0}}^{\tb^1}}{r_{\ul{b}^{j_0}}^{\tb^0}},0,...,0,v_{\ul{b}^{j_0}}^{\tb^1}-v_{\ul{b}^{j_0}}^{\tb^0},0,...,0)
\end{align*}
where $\ld'=(\ul{b}^1,...,\ul{b}^n)$, and $\tb^0,\, \tb^1 \in \Sigma_{j_0}^-$ end with $\ul{a}^0$, $\ul{a}^1$, respectively. If $j_0=n$, then
\begin{align*}
a_{\ld_1}^{\ld'}-a_{\ld_0}^{\ld'}&= (\log \frac{r^{\tb^0}_{\ul{b}^n}}{r^{\tb^1}_{\ul{b}^n}},...,\log \frac{r^{\tb^0}_{\ul{b}^n}}{r^{\tb^1}_{\ul{b}^n}},0,...,0,v_{\ul{b}^{n}}^{\tb^1}-v_{\ul{b}^{n}}^{\tb^0}).
\end{align*}
We remark that $a_{\ld_1}^{\ld'}-a_{\ld_0}^{\ld'}$ only depends on the $j_0$ Cantor set $K_{j_0}$.

Given that $K_{j_0}$ is not essentially affine there is $\tilde{\tb}^0, \tilde{\tb}^1 \in \Sigma_{j_0}^-$ and $x_0\in K_{j_0}^{\tilde{\tb}^0}$ such that
\[D^2[k^{\tilde{\tb}^1}\circ (k^{\tilde{\tb}^0})^{-1}](x_0)\neq 0.\]
For any $\tb^0, \tb^1 \in \Sigma_{j_0}^-$ we define $F_{\tb^0,\tb^1}:=k^{\tb^1}\circ (k^{\tb^0})^{-1}$. Since $x_0\in K_{j_0}^{\tilde{\tb}^0}$ then $DF_{\tilde{\tb}^0,\tilde{\tb}^1}(x_0)$ is a conformal matrix. Denote by $C\subset GL(2,\R)$ the set of $2\times 2$ conformal matrices. Let $P:U\to C$ be a smooth function from a neighborhood $U\subset GL(2,\R)$ of $DF_{\tilde{\tb}^0,\tilde{\tb}^1}(x_0)$ into $C$, such that $P(A)=A$ for all $A\in C\cap U$. We will use the notation $\mathbb{D}F_{\tb^0,\tb^1}(x)= P(DF_{\tb^0,\tb^1}(x))$. The properties of $P$ and the fact that $K_{j_0}$ is not essentially real allow us to conclude that $D\mathbb{D}F_{\tilde{\tb}^0,\tilde{\tb}^1}(x_0)= D^2 F_{\tilde{\tb}^0,\tilde{\tb}^1}(x_0)$. Now notice that $C$ can be naturally identified with $\C^*$ and in this sense we can chose a branch of logarithm $\log$ defined in $P(U)$ (for $U$ small). Then lemma \ref{lem:HigherDerivative} will imply that 
\[\beta := D \log \mathbb{D}F_{\tilde{\tb}^0,\tilde{\tb}^1}(x_0)\neq 0\]
is a conformal matrix. In the rest of the proof we will make an abuse of notation, $v_{\ul{b}^{j_0}}^{\tb^1}-v_{\ul{b}^{j_0}}^{\tb^0}$ will not represent an element of $\T$ but the imaginary part of
\[\log e^{(v_{\ul{b}^{j_0}}^{\tb^1}-v_{\ul{b}^{j_0}}^{\tb^0})i},\]
in this way we have chosen a representative in the class defined by $v_{\ul{b}^{j_0}}^{\tb^1}-v_{\ul{b}^{j_0}}^{\tb^0}$. Define the following vectors in $\R^2$
\begin{align*}
    d^{\ld'}_{\ld_1,\ld_0}=(\log \frac{r_{\ul{b}^{j_0}}^{\tb^1}}{r_{\ul{b}^{j_0}}^{\tb^0}},v_{\ul{b}^{j_0}}^{\tb^1}-v_{\ul{b}^{j_0}}^{\tb^0}),
\end{align*}
and
\[ \tilde{\xi}=
	\begin{cases}
		(\mu_{j_0},m_{j_0})  & \mbox{if } j_0\neq n, \\
		(-(\mu_1+...+\mu_{n-1}),m_n)  & \mbox{if } j_0=n.
	\end{cases}
\]
Notice that $1\leq |\tilde{\xi}|\leq n \Delta_0 \rho^{-1}$ and 
\[\langle \xi,a_{\ld_1}^{\ld'}-a_{\ld_0}^{\ld'}\rangle=\langle \tilde{\xi}, d_{\ld_1, \ld_0}^{\ld'}\rangle \mod 2\pi \Z,\]
where the $\langle \cdot, \cdot \rangle$ in the right hand side of the equation refers to the usual inner product on $\R^2$.

Since $k^{\tb}$ depends continuously on $\tb$ we get that
\begin{equation}\label{eq:beta}
|D \log \mathbb{D}F_{\tb^0,\tb^1}(x)-\beta|\leq \delta_1,
\end{equation}
for all $\tb^0$, $\tb^1$, $x$ close enough to $\tilde{\tb}^0$, $\tilde{\tb}^1$, $x_0$; the value of $\delta_1$ will be fixed later. We assume that $\eta_2$ is small such that the proportion of $\tilde{\Ld}$ inside $\Ld^i$ is big enough to exist $\ld_0, \ld_1\in \tilde{\Ld}$, with the form specified before, verifying that $\tb^0$, $\tb^1$ are close to $\tilde{\tb^0}$, $\tilde{\tb^1}$ so that Eq. (\ref{eq:beta}) holds.\footnote{Here we also need to suppose that $\rho_0$ is small enough.} From now on, $\ld_0$ and $\ld_1$ are fixed as these values.

Now fix $\ul{c}^0\in \Sigma_{j_0}^{fin}$ such that any $x$ in the convex hull of $G^{\tb^0}(\ul{c}^0)$ is close enough to $x_0$ in order to have Eq. (\ref{eq:beta}). Denote by $\Sigma_{j_0}(\tilde{c}_0,\rho,\ul{c}^0)$ all the elements of $\Sigma_{j_0}(\tilde{c}_0,\rho)$ starting with $\ul{c}^0$, this is a positive proportion of $\Sigma_{j_0}(\tilde{c}_0,\rho)$ (independent of $\rho$). Then, if we assume $\eta_5$ small enough we can guarantee that for a proportion of $\ul{b}\in \Sigma_{j_0}(\tilde{c}_0,\rho,\ul{c}^0)$, as big as we want, there exist $\ul{b}^j\in \Sigma_j(\tilde{c}_0,\rho)$, $j\neq j_0$, such that $\ld'=(\ul{b}^1,...,\ul{b}^{j_0-1},\ul{b},\ul{b}^{j_0+1},...,\ul{b}^n)$ verifies
\begin{equation}\label{eq:discreto}
|\langle \tilde{\xi}, d_{\ld_1, \ld_0}^{\ld'}\rangle+\phi-2m(\ul{b})\pi |\leq \eta_5,
\end{equation}
where $m(\ul{b})$ is an integer depending on $\ul{b}$. Denote the set of such $\ul{b}$ by $\tilde{\Sigma}_{j_0}(\tilde{c}_0,\rho,\ul{c}^0)$.

For simplicity write $F=k^{\tb^1}\circ (k^{\tb^0})^{-1}$, instead of $F_{\tb^0,\tb^1}$, then we have
\begin{equation}\label{eq:dfv}
F(c_{\ul{b}}^{\tb^0})=c_{\ul{b}}^{\tb^1},\,\,\,DF(c_{\ul{b}}^{\tb^0})=\|DF(c_{\ul{b}}^{\tb^0})\|\cdot R_{v_{\ul{b}}^{\tb^1}-v_{\ul{b}}^{\tb^0}}.
\end{equation}

We will show that the distance between $\log \mathbb{D}F(c_{\ul{b}}^{\tb^0})$ and $d_{\ld_1,\ld_0}^{\ld'}$ is of order $\rho$, for any $\ld'\in \Ld$. Let $z_0, z_1 \in G^{\tb^0}(\ul{b})$ such that $r^{\tb^0}_{\ul{b}}=|z_0-z_1|$, using Taylor expansion at $c^{\tb^0}_{\ul{b}}$ we get
\[F(z_0)-F(z_1)=DF(c^{\tb^0}_{\ul{b}})(z_0-z_1)+O(|z_0-z_1|^2),\]
where the constant in the $O$ notation does not depend on $\rho$, $\ul{b}$, $\tb^1$ or $\tb^0$. Hence
\begin{align*}
r_{\ul{b}}^{\tb^0}\|DF(c_{\ul{b}}^{\tb^0})\|&=|z_0-z_1|\|DF(c_{\ul{b}}^{\tb^0})\|\\
&\leq  |F(z_0)-F(z_1)|+O(|z_0-z_1|^2)\\
&\leq r_{\ul{b}}^{\tb^1}+O((r_{\ul{b}}^{\tb^0})^2).
\end{align*}
Dividing by $r_{\ul{b}}^{\tb^0}$ and using the fact that $r_{\ul{b}}^{\tb^0}$ is of order $\rho$ we get
\[ \|DF(c_{\ul{b}}^{\tb^0})\|-\frac{r_{\ul{b}}^{\tb^1}}{r_{\ul{b}}^{\tb^0}} \leq O(\rho).\]
A simililar argument gives
\[\frac{r_{\ul{b}}^{\tb^1}}{r_{\ul{b}}^{\tb^0}}-\|DF(c_{\ul{b}}^{\tb^0})\| \leq O(\rho).\]
Therefore
\[ \left| \|DF(c_{\ul{b}}^{\tb^0})\|-\frac{r_{\ul{b}}^{\tb^1}}{r_{\ul{b}}^{\tb^0}} \right| \leq O(\rho).\]
Now, given the fact that $\|DF(c_{\ul{b}}^{\tb^0})\|$ and $r_{\ul{b}}^{\tb^1}/r_{\ul{b}}^{\tb^0}$ are uniformly bounded away from zero, we obtain that there is a constant $C_3>0$, independent of $\rho$, such that
\begin{equation}\label{eq:dfr}
\left| \log \|DF(c_{\ul{b}}^{\tb^0})\|- (\log r_{\ul{b}}^{\tb^1}-\log r_{\ul{b}}^{\tb^0}) \right| \leq C_3\rho.
\end{equation}
Given $\ul{b}^1,\ul{b}^2\in \Sigma_{j_0}(\tilde{c}_0,\rho,\ul{c}^0)$, by the choice of $\ul{c}^0$, using Taylor aproximation and Eq. (\ref{eq:beta}) we will have that
\begin{equation}\label{eq:btilde}
\log \mathbb{D}F(c^{\tb^0}_{\ul{b}^1})-\log \mathbb{D}F(c^{\tb^0}_{\ul{b}^2})=\beta_1(c^{\tb^0}_{\ul{b}^1}-c^{\tb^0}_{\ul{b}^2}),
\end{equation}
for some $\beta_1$ such that $\|\beta_1-\beta\|\leq \delta_1$.

The idea to finish the proof is the following: we use Eq. (\ref{eq:discreto}) to see that the set of $d_{\ld_1, \ld_0}^{\ld'}$ projected to the line generated by $\tilde{\xi}$ is close to an arithmetic progression, then two points will be either very close or very far from each other. Equations (\ref{eq:dfv}), (\ref{eq:dfr}), (\ref{eq:btilde}) allow to translate this fact about $d_{\ld_1,\ld_0}^{\ld'}$ to the analogous one about the set of $c^{\tb^0}_{\ul{b}}$. Finally, we will use the fact that $K_{j_0}$ is not essentially real to estimate $|c^{\tb^0}_{\ul{b}^1}-c^{\tb^0}_{\ul{b}^2}|$ from $\langle \tilde{\xi}, \beta_1 (c^{\tb^0}_{\ul{b}^1}-c^{\tb^0}_{\ul{b}^2})\rangle$, thus it will happen that $|c^{\tb^0}_{\ul{b}^1}-c^{\tb^0}_{\ul{b}^2}|$ is either too big or too small which will bring us into a contradiction with the boundeness of the geometry of the Cantor set.

Any pair $\ul{b}_1, \ul{b}_2 \in \tilde{\Sigma}_{j_0}(\tilde{c}_0,\rho,\ul{c}^0)$ should verify one of two options:
\begin{itemize}
\item[(i)] If $m(\ul{b}_1)=m(\ul{b}_2)$, using Eq. (\ref{eq:discreto}) for $\ld'_1$, $\ld'_2$ assosiated to $\ul{b}_1$, $\ul{b}_2$, respectively, we get
\[ |\langle \tilde{\xi},d_{\ld_1,\ld_0}^{\ld'_1}\rangle-\langle \tilde{\xi},d_{\ld_1,\ld_0}^{\ld'_2}\rangle |\leq 2\eta_5.\]
This together with Eq. (\ref{eq:dfv}), (\ref{eq:dfr}) gives
\[| \langle \tilde{\xi}, \log \mathbb{D}F(c_{\ul{b}_1}^{\tb^0})-\log \mathbb{D}F(c_{\ul{b}_2}^{\tb^0})\rangle| \leq 2\eta_5+2C_3 |\tilde{\xi}|\rho. \]
Considering Eq. (\ref{eq:btilde}) leads to
\[ |\langle \beta_1^T \tilde{\xi},c_{\ul{b}_1}^{\tb^0}-c_{\ul{b}_2}^{\tb^0}\rangle|=|\langle \tilde{\xi},\beta_1(c_{\ul{b}_1}^{\tb^0}-c_{\ul{b}_2}^{\tb^0})\rangle | \leq 2\eta_5+2C_3 |\tilde{\xi}|\rho,\]
where $\beta_1^T$ is the transpose of $\beta_1$.
\item[(ii)] If $m(\ul{b}_1)\neq m(\ul{b}_2)$ a similar process arrives to
\[ |\langle \beta_1^T \tilde{\xi},c_{\ul{b}_1}^{\tb^0}-c_{\ul{b}_2}^{\tb^0}\rangle| \geq \pi-2C_3 |\tilde{\xi}|\rho.\]
\end{itemize}
Now we use the hypothesis that $K_{j_0}$ is not essentially real. First, we choose a constant $C_5>0$, depending only in $\tilde{c}_0$ and the Cantor set $K_{j_0}$, such that for any $\ul{a}\in \Sigma_{j_0}^{fin}$ one has
\[\{f_{\ul{a}}^{-1}(G(\ul{b})):\ul{b}\in\Sigma_{j_0}(\tilde{c}_0,\rho,\ul{a})\}\subset  \{G(\ul{b}):\ul{b}\in\Sigma_{j_0}(C_5,\tilde{\rho})\},\]
for some $\tilde{\rho}>0$, which depends on $\rho$ and $\ul{a}$. 
Lemma \ref{lem:cono} proves that there is an angle $\alpha \in (0,\pi/2)$ and numbers $\rho_2>0$, $a\in (0,1)$ such that for any limit geometry $k^{\tb}$, $x\in G^{\tb}(\theta_0)$, line $L$, $s\in \mathbb{A}_{j_0}$, $D$ discretization of $K_{j_0}(\theta_0, s)$ of order less than $\rho_2$
\[\#\{\ul{a}\in D:\,\, G^{\tb}(\ul{a})\cap Cone(x,L,\alpha)\neq \emptyset \}\leq a\cdot \#D.\]
Remember that a discritization $D$ of $K_{j_0}(\theta_0, s)$ of order $\rho$ is a subset of $\Sigma_{j_0} (C_5,\rho)$ such that
\[\bigcup_{\ul{a}\in D}K_{j_0}(\ul{a})=K_{j_0}(\theta_0, s),\]
for some pre-fixed constant $C_5$.

Fix $\delta_1$ by requiring that $\|\beta_1-\beta\|<\delta_1$ implies that
\[m(\beta)/2\leq m(\beta_1^T) \leq \|\beta_1^T\| \leq 2\|\beta\|\]
and the angle between $\beta^T w$ and $\beta_1^T w$ is less than $\alpha/2$, for any $w\in \R \setminus \{0\}$. Remember that $m(A)=\inf_{w\neq 0}\frac{|A w|}{|w|}$ and that $\beta$ is conformal, hence $m(\beta^T)=m(\beta)=\|\beta\|=\|\beta^T\|$. Fix $\tilde{a}\in (a,1)$, assuming $\eta_5$ small enough we can guarantee that 
\[\#\tilde{\Sigma}_{j_0}(\tilde{c}_0,\rho, \ul{c}^0)> \tilde{a} \cdot \#\Sigma_{j_0}(\tilde{c}_0,\rho, \ul{c}^0).\]
This allows us to find a finite sequence $\ul{c}^0$,...,$\ul{c}^m$ of elements of $\Sigma^{fin}$ such that:
\begin{itemize}
\item $\ul{c}^{j+1}$ starts with $\ul{c}^j$ and has one more letter.
\item $\#\left(\Sigma_{j_0}(\tilde{c}_0,\rho, \ul{c}^j)\cap \tilde{\Sigma}_{j_0}(\tilde{c}_0,\rho, \ul{c}^0)\right)> \tilde{a}\cdot \#\Sigma_{j_0}(\tilde{c}_0,\rho, \ul{c}^j)$.
\item $\Sigma_{j_0}(\tilde{c}_0,\rho, \ul{c}^j)\cap \tilde{\Sigma}_{j_0}(\tilde{c}_0,\rho, \ul{c}^0) \not\subset \Sigma_{j_0}(\tilde{c}_0,\rho, \ul{c}^{j+1})$.\footnote{Actually, for this property to be true we take $\tilde{a}$ big such that $\#\Sigma_{\tilde{c}_0}(\rho, \ul{c}^{j+1})<\tilde{a}\cdot \Sigma_{\tilde{c}_0}(\rho, \ul{c}^j)$.}
\item $\ul{c}^j \in \Sigma_{j_0}(\tilde{c}_0,\rho)$ only for $j=m$.
\end{itemize}
Fix an integer $m_0<m$ such that for any $\ul{b}\in \Sigma_{j_0}(\tilde{c}_0,\rho,\ul{c}^j)$ and $j<m_0$ we have $f_{\ul{c}^j}^{-1}(G(\ul{b}))=G(\ul{b}')$ for $\ul{b}'\in \Sigma_{j_0}(C_5,\tilde{\rho})$, for $\tilde{\rho}<\rho_2$ (this only requires that $m-m_0$ is big enough). For each $\ul{c}^j$, $j<m$, we will choose two elements $\ul{a}^{1,j}, \ul{a}^{2,j}\in \Sigma_{j_0}(\tilde{c}_0,\rho,\ul{c}^j)\cap \tilde{\Sigma}_{j_0}(\tilde{c}_0,\rho,\ul{c}^0)$ in the following way:
\begin{itemize}
\item First, we choose any $\ul{a}^{1,j} \in \Sigma_{j_0}(\tilde{c}_0,\rho,\ul{c}^j s')\cap \tilde{\Sigma}_{j_0}(\tilde{c}_0,\rho,\ul{c}^0)$, where $s'$ is a letter in $\mathbb{A}_{j_0}$ such that $\ul{c}^{j+1}\neq \ul{c}^j s'$.
\item If $j\geq m_0$ then we choose any $\ul{a}^{2,j}\in \Sigma_{j_0}(\tilde{c}_0,\rho,\ul{c}^{j+1})\cap \tilde{\Sigma}_{j_0}(\tilde{c}_0,\rho,\ul{c}^0)$.
\item Suppose $j<m_0$. Given $\ul{b}\in \Sigma_{j_0}(\tilde{c}_0,\rho,\ul{c}^{j+1})$, it can be written as \[\ul{b}=(c^j_{-k},..,c^j_{-1},c^j_0,c^{j+1}_0,b_1,...,b_p),\]
where $c^j_0$, $c^{j+1}_0$ are the last letters of $\ul{c}^j$, $\ul{c}^{j+1}$, respectively. Using this notation we can define the set
\[D=\{(c^j_0,c^{j+1}_0,b_1,..,b_p)\in \Sigma^{fin}: (c^j_{-k},..,c^j_{-1},c^j_0,c^{j+1}_0,b_1,..,b_p)\in \Sigma_{j_0}(\tilde{c}_0,\rho,\ul{c}^{j+1})\}.\]
This set is a discretization of $K_{j_0}(c^j_0,c^{j+1}_0)$ and by our choice of $m_0$ it has order less than $\rho_2$. Now use lemma \ref{lem:cono} for the limit geometry $K_{j_0}^{\tb^0 \ul{c}^j}$, point $x=(F^{\tb^0}_{\ul{c}^j})^{-1}(c^{\tb^0}_{\ul{a}^{1,j}})$ and line $L$ such that $F^{\tb^0}_{\ul{c}^j}(L)$ is orthogonal to the line generetad by $\beta^T \tilde{\xi}$. Hence
\[\#\{\ul{a}\in D:\,\, G^{\tb^0 \ul{c}^j}(\ul{a})\cap Cone(x,L,\alpha)\neq \emptyset \}\leq a\cdot \#D,\]
and then
\[\#\{\ul{b}\in \Sigma_{j_0}(\tilde{c}_0,\rho,\ul{c}^{j+1}):\,\, G^{\tb^0}(\ul{b})\cap Cone(c^{\tb^0}_{\ul{a}^{1,j}},F^{\tb^0}_{\ul{c}^j}(L),\alpha)\neq \emptyset \}\leq a\cdot \# \Sigma_{j_0}(\tilde{c}_0,\rho,\ul{c}^{j+1}).\]
Since $\tilde{a}>a$ then there are elements $\ul{a} \in \Sigma_{j_0}(\tilde{c}_0,\rho,\ul{c}^{j+1})\cap \tilde{\Sigma}_{j_0}(\tilde{c}_0,\rho,\ul{c}^0)$ such that 
\[G^{\tb^0}(\ul{a})\cap Cone(c^{\tb^0}_{\ul{a}^{1,j}},F^{\tb^0}_{\ul{c}^j}(L),\alpha)= \emptyset.\]
We choose $\ul{a}^{2,j}$ as any such element. It easily follows from the choice of $\ul{a}^{2,j}$ that
\[\measuredangle (c^{\tb^0}_{\ul{a}^{1,j}}-c^{\tb^0}_{\ul{a}^{2,j}}, F^{\tb^0}_{\ul{c}^j}(L))>\alpha. \]
\end{itemize}
Let $\tilde{L}$ be the line orthogonal to the vector $\beta_1^T \tilde{\xi}$. The previous inequality and the choice of $\delta_1$ implies that
\begin{equation}\label{eq:angestimate}
\measuredangle (c^{\tb^0}_{\ul{a}^{1,j}}-c^{\tb^0}_{\ul{a}^{2,j}}, \tilde{L})>\alpha/2.
\end{equation}
Now we have all the ingredients to finish the proof. For any $j$ the pair $\ul{a}^{1,j}, \ul{a}^{2,j}$ verifies either option (i) or (ii), note that (ii) implies
\begin{align*}
|c^{\tb^0}_{\ul{a}^{1,j}}-c^{\tb^0}_{\ul{a}^{2,j}}| \geq \frac{|\langle \beta_1^T \tilde{\xi}, c^{\tb^0}_{\ul{a}^{1,j}}-c^{\tb^0}_{\ul{a}^{2,j}}\rangle|}{ \|\beta_1^T\|\cdot |\tilde{\xi}| } &\geq \pi \|\beta_1\|^{-1}|\tilde{\xi}|^{-1}-2C_3 \|\beta_1\|^{-1}\rho\\
&\geq \frac{\pi \|\beta\|^{-1}}{2}|\tilde{\xi}|^{-1}-4C_3 \|\beta\|^{-1}\rho\\
&\geq \left(\frac{\pi \|\beta\|^{-1}}{2n}\Delta_0^{-1} -4C_3 \|\beta\|^{-1}\right)\rho.
\end{align*}
Hence, choosing $\Delta_0$ small enough, we can guarantee that option (ii) is not verified for $j=n-1$. On the other hand, if $\ul{a}^{1,j}, \ul{a}^{2,j}$ verifies option (i) then using (\ref{eq:angestimate}) we get
\begin{align*}
|c^{\tb^0}_{\ul{a}^{1,j}}-c^{\tb^0}_{\ul{a}^{2,j}}| \cdot \sin (\alpha/2) &\leq |c^{\tb^0}_{\ul{a}^{1,j}}-c^{\tb^0}_{\ul{a}^{2,j}}| \cdot \sin \measuredangle (c^{\tb^0}_{\ul{a}^{1,j}}-c^{\tb^0}_{\ul{a}^{2,j}}, \tilde{L})\\
&= |c^{\tb^0}_{\ul{a}^{1,j}}-c^{\tb^0}_{\ul{a}^{2,j}}| \cdot \cos \measuredangle (c^{\tb^0}_{\ul{a}^{1,j}}-c^{\tb^0}_{\ul{a}^{2,j}}, \R \beta_1^T \tilde{\xi})\\
&= \frac{|\langle \beta_1^T \tilde{\xi} , c^{\tb^0}_{\ul{a}^{1,j}}-c^{\tb^0}_{\ul{a}^{2,j}}\rangle|}{|\beta_1^T \tilde{\xi}|}\\
&\leq \frac{2\eta_5 + 2C_3 |\tilde{\xi}| \rho}{\frac{1}{2}\|\beta\| \cdot |\tilde{\xi}|}\\
&\leq 4\eta_5 |\tilde{\xi}|^{-1}\|\beta\|^{-1}+4C_3\rho \|\beta\|^{-1},
\end{align*}
we obtained
\[|c^{\tb^0}_{\ul{a}^{1,j}}-c^{\tb^0}_{\ul{a}^{2,j}}| \leq (\sin(\alpha/2))^{-1} (4\eta_5 |\tilde{\xi}|^{-1}\|\beta\|^{-1}+4C_3\rho \|\beta\|^{-1}).\]
Hence, assuming $\eta_5$ and $\rho$ small enough, we can guarantee that option (i) is not verified for $j=0$. Therefore, there exists $j$ such that $\ul{a}^{1,j}, \ul{a}^{2,j}$ verifies option (ii) and $\ul{a}^{1,j+1}, \ul{a}^{2,j+1}$ verifies option (i), from the inequalities obtained we see that
\begin{align*}
\frac{|c^{\tb^0}_{\ul{a}^{1,j+1}}-c^{\tb^0}_{\ul{a}^{2,j+1}}|}{|c^{\tb^0}_{\ul{a}^{1,j}}-c^{\tb^0}_{\ul{a}^{2,j}}|}&\leq \left(\sin \frac{\alpha}{2}\right)^{-1} \frac{4\eta_5 |\tilde{\xi}|^{-1}\|\beta\|^{-1}+4C_3\rho \|\beta\|^{-1}}{\frac{\pi \|\beta\|^{-1}}{2}|\tilde{\xi}|^{-1}-4C_3\|\beta\|^{-1}\rho}\\
&= \left(\sin \frac{\alpha}{2}\right)^{-1} \frac{4\eta_5 +4C_3\rho |\tilde{\xi}|}{(\pi /2)-4C_3\rho|\tilde{\xi}|}\\
&\leq \left(\sin \frac{\alpha}{2}\right)^{-1} \frac{4\eta_5 +4nC_3\Delta_0}{(\pi /2)-4nC_3\Delta_0},
\end{align*}
here we used $|\tilde{\xi}|\in [1,n\Delta_0\rho^{-1}]$. We obtained
\[\frac{|c^{\tb^0}_{\ul{a}^{1,j+1}}-c^{\tb^0}_{\ul{a}^{2,j+1}}|}{|c^{\tb^0}_{\ul{a}^{1,j}}-c^{\tb^0}_{\ul{a}^{2,j}}|}\leq \left(\sin \frac{\alpha}{2}\right)^{-1} \frac{4\eta_5 +4nC_3\Delta_0}{(\pi /2)-4nC_3\Delta_0},\]
notice that the right hand side of the inequality goes to zero as $\Delta_0$ and $\eta_5$ go to zero, however the left hand side is bounded away from zero thanks to the bounded geometry of the Cantor set $K_{j_0}$. We conclude that for $\Delta_0$ small enough $\eta_5$ is bounded away from zero, as we wanted to prove.

\bibliographystyle{unsrt}
\bibliography{bibliography}

\end{document}